\documentclass[10pt]{amsart}
\usepackage[margin=2.5cm]{geometry}
\usepackage{tikz-cd}
\usepackage{graphicx}					
\usepackage{amssymb}
\usepackage{amsmath}
\usepackage{subfig, graphicx}
\usepackage{mathrsfs}
\usepackage{amssymb, amsthm, indentfirst}
\usepackage{relsize}
\usepackage{hyperref}
\usepackage{stmaryrd}
\usepackage{bm}
\usepackage{graphics}

\numberwithin{equation}{section}

\usepackage{color}
\theoremstyle{plain}
\newtheorem*{thma}{Theorem A}
\newtheorem*{thmb}{Theorem B}
\newtheorem*{thmc}{Theorem C}
\newtheorem*{thmd}{Theorem D}
\newtheorem{thm}{Theorem}[section]

\newtheorem*{propa}{Proposition A}
\newtheorem{corollary}[thm]{Corollary}
\theoremstyle{definition}

\newtheorem{rmk}{Remark}
\def\Gal{\operatorname{Gal}}

\newtheorem*{hypothesis*}{Hypothesis}

\newcommand{\As}{\mathop{\rm As}\nolimits}

\newcommand{\tr}{\mathop{\rm tr}\nolimits}

\newcommand{\vol}{{\rm vol}}

\newcommand{\mf}[1]{\mathfrak{#1}}

\newcommand{\ms}[1]{\mathscr{#1}}
\newcommand{\ratint}{\mathbb Z}

\newcommand{\intring}[1]{\mathcal{O}_{#1}}
\newcommand{\cpxnum}{\mathbb{C}}

\newcommand{\bfb}[1]{{\mathbf #1}}

\newcommand{\nami}[1]{\widetilde{#1}}
\newcommand{\yama}[1]{\widehat{#1}}

\def\parref#1{\ref{#1}}
\def\thmref#1{Theorem~\parref{#1}}

\def\secref#1{\S\parref{#1}}

\def\lmref#1{Lemma~\parref{#1}}
\def\subsecref#1{\S\parref{#1}}

\def\makeop#1{\expandafter\def\csname#1\endcsname
  {\mathop{\rm #1}\nolimits}\ignorespaces}
\newcommand{\wh}{\widehat}
\newcommand\stt[1]{\left\{#1\right\}}
\newcommand{\pMX}[4]{\begin{pmatrix}
{#1}& {#2}\\
{#3}&{#4}\end{pmatrix} }
\newtheorem{lm}[thm]{Lemma}
\def\longto{\longrightarrow}
\def\ul{\underline}

\newcommand{\lmid}[1]{\raise3pt\hbox{\vrule width 0.4pt height #1pt depth #1pt}}
\newcommand{\mapdiag}[5]{\xymatrix{{#1}\hspace{-12mm}&\ #2 \ar[r] & #3 \\ & #4 \ar@{}[u]|{\mws{$\in$}} \ar@{|->}[r] & #5 \ar@{}[u]|{\mws{$\in$}}}}
\def\mat(#1,#2,#3,#4){\left(\begin{array}{cc}#1 &#2 \\ #3& #4\end{array}\right)}

\def\GL{\mathop{\mathrm{GL}}\nolimits}
\def\SL{\mathop{\mathrm{SL}}\nolimits}

\def\Gal{\mathop{\mathrm{Gal}}\nolimits}

\def\ord{\mathrm{ord}}

\def\Sp{\mathop{\mathrm{Sp}}\nolimits}

\numberwithin{equation}{subsection}
\numberwithin{figure}{subsection}

\def\B(#1,#2){{\rm Ind}_{B(F)}^{\GL_2(F)}(#1\boxtimes #2)}

\title{Gamma factors for Asai representations of $\GL_2$}
\author{SHIH-YU CHEN}
\author{YAO CHENG} 
\author{ISAO ISHIKAWA}


\address{Shih-Yu Chen, Institute of Mathematics, Academia Sinica, Taipei 10617, Taiwan}
\email{r98221018@ntu.edu.tw}
\address{Yao Cheng, Institute of Mathematics, Academia Sinica, Taipei 10617, Taiwan}
\email{d02221003@ntu.edu.tw}
\address{Isao Ishikawa, RIKEN Center for Advanced Intelligence Project, Tokyo 103-0027, Japan}
\email{isao.ishikawa@riken.jp}

\def\cL{{\mathcal L}}
\def\ot{\otimes}

\def\b{\bar}
\def\t{\tilde}

\def\bp{\boxplus}
\def\op{\oplus}

\def\cS{\mathfrak{S}}
\def\cW{\mathscr{W}}
\def\cB{\mathcal{B}}

\def\C{\mathbb{C}}
\def\R{\mathbb{R}}
\def\Z{\mathbb{Z}}
	
\def\SL{{\rm{SL}}}
\def\GL{{\rm{GL}}}
\def\GSp{{\rm GSp}}

\def\Sp{{\rm Sp}}

\def\A{{\mathbb A}}
\def\C{{\mathbb C}}
\def\R{{\mathbb R}}
\def\Q{{\mathbb Q}}
\def\Z{{\mathbb Z}}

\def\<{\langle}
\def\>{\rangle}
\def\B{B}
\def\P{P}
\def\U{U}
\def\G{{\bf G}}

\def\x{\times}
\def\t{\tilde}

\def\bp{\begin{pmatrix}}
\def\ep{\end{pmatrix}}

\def\<{\langle}
\def\>{\rangle}

\def\w_E{\varpi_E}
				
\begin{document}

\maketitle
\begin{abstract}
Let $E$ be a quadratic semisimple extension of a local field $F$ of characteristic zero. We determine explicit relation between gamma factors for Asai representations of ${\rm R}_{E/F}\GL_{2/E}$ defined by the Weil-Deligne representations and local zeta integrals. When $E=F\times F$, the results were due to Henniart and Jacquet. We completed the theory in this article based on explicit calculation.
\end{abstract}
\tableofcontents

\section{Introdution}
\subsection{Main results}
Let $F$ be a local field of characteristic zero and $E$ be a quadratic semisimple $F$-algebra. Fix a non-trivial additive character $\psi$ of $F$ and an element $\xi \in E^{\times}$ such that ${\rm tr}_{E/F}(\xi)=0$. Let $W_F'$ be the Weil-Deligne group of $F$. We identity the Langlands dual group of ${\rm R}_{E/F}\GL_{2/E}$ with $(\GL_2(\C)\times\GL_2(\C))\rtimes {\rm Gal}(\overline{F}/F)$, where the action of ${\rm Gal}(\overline{F}/F)$ on $\GL_2(\C)\times\GL_2(\C)$ is the permutation of components induced by the homomorphism ${\rm Gal}(\overline{F}/F)\rightarrow {\rm Gal}(E/F).$ Let 
$$r : (\GL_2(\C)\times\GL_2(\C))\rtimes {\rm Gal}(\overline{F}/F) \rightarrow \GL(\C^2 \otimes \C^2)\times {\rm Gal}(\overline{F}/F)$$
be the Asai representation defined by
\begin{align}\label{E:Asai representation}
r(g_1,g_2,\sigma) = \begin{cases}(g_1\otimes g_2,\sigma) & \mbox{ if $\sigma\vert_E$ is trivial},\\
(g_2\otimes g_1,\sigma) & \mbox{ if $\sigma\vert_E$ is not trivial}.        \end{cases}
\end{align}

Let $\pi$ and $ \tau$ be irreducible admissible representations of $\GL_2(E)$ and $\GL_n(F)$ with central characters $\omega_{\pi}$ and $\omega_{ \tau}$, respectively. Denote $\phi_{\pi} : W_F' \rightarrow (\GL_2(\C)\times\GL_2(\C))\rtimes {\rm Gal}(\overline{F}/F)$ and $\phi_{ \tau} : W_F' \rightarrow \GL_n(\C) \rtimes {\rm Gal}(\overline{F}/F)$ the parameters associated to $\pi$ and $ \tau$ via the local Langlands correspondence. Let $(r\circ \phi_{\pi}) \otimes \phi_{ \tau} : W_F' \rightarrow \GL_{4n}(\C) \rtimes {\rm Gal}(\overline{F}/F)$ be the parameter defined by the tensor representation 
$$\GL(\C^2\otimes \C^2)\times\GL_n(\C) \rightarrow \GL_{4n}(\C).$$
Let ${\rm As}\,\pi$ be the irreducible admissible representation of $\GL_4(F)$ associated to the parameter $r \circ \phi_{\pi}$ via the local Langlands correspondence. Denote $L_{\rm Gal}(s, {\rm As}\,\pi \otimes  \tau)$ and $\varepsilon_{\rm Gal}(s,{\rm As}\,\pi\otimes  \tau,\psi)$ the $L$-factor and $\varepsilon$-factor associated to $(r \circ \phi_{\pi}) \otimes \phi_{ \tau}$, respectively. Put
$$\gamma_{\rm Gal}(s,{\rm As}\,\pi \otimes  \tau,\psi) = \varepsilon_{\rm Gal}(s,{\rm As}\,\pi\otimes  \tau,\psi)L_{\rm Gal}(1-s, {\rm As}\,\pi^{\vee} \otimes  \tau^{\vee})L_{\rm Gal}(s, {\rm As}\,\pi \otimes  \tau)^{-1}.$$
Here $\pi^\vee$ (resp. $\tau^\vee$) is the admissible dual of $\pi$ (resp. $\tau$).

In this article, we compare the local factors defined by local zeta integrals with the local factors defined by the Weil-Deligne representation as above in the cases $n=1,2$. When $n=1$, we consider the local factors $L_{\rm RS}(s, {\rm As}\,\pi \otimes  \tau)$ and $\varepsilon_{\rm RS}(s,{\rm As}\,\pi\otimes  \tau,\psi,\xi)$ defined by the Rankin-Selberg local zeta integrals when $\pi$ is generic (Cf. \cite{Jac72}, \cite{F88}, \cite{F93}, and \cite{Kab04}). When $n=2$, we consider the local factors $L_{\rm PSR}(s, {\rm As}\,\pi \otimes  \tau)$ and $\varepsilon_{\rm PSR}(s,{\rm As}\,\pi\otimes  \tau,\psi,\xi)$ defined by the local zeta integrals of Piatetski-Shapiro and Rallis when both $\pi$ and $ \tau$ are generic (Cf. \cite{PSR1987} and \cite{Ikeda1989}). The main results of this paper are the following 
\begin{thma}[Corollary \ref{equality of epsilon factors} and Theorem \ref{T:main theorem}]
Let $n=1$. Assume $\pi$ is generic. We have
\begin{align*}
L_{\rm RS}(s,{\rm As}\,\pi\otimes \tau)& = L_{\rm Gal}(s,{\rm As}\,\pi\otimes \tau),\\
\varepsilon_{\rm RS}(s,{\rm As}\,\pi\otimes \tau,\psi,\xi) &= \omega_{\pi}(\xi)\omega_{\tau}(\xi^2)|\xi^2|_F^{s-1/2}\lambda_{E/F}(\psi)^{-1}\varepsilon_{\rm Gal}(s,{\rm As}\,\pi\otimes \tau,\psi).
\end{align*}
Here $\lambda_{E/F}(\psi)$ is the Langlands constant for $E/F$ with respect to $\psi$.
\end{thma}

\begin{thmb}[Theorem \ref{T:twisted Asai gamma factor}]
Let $n=2$. Assume both $\pi$ and $ \tau$ are generic, and $ \tau$ is a subquotient of a principal series representation. We have
\begin{align*}
\gamma_{\rm PSR}(s,{\rm As}\,\pi \otimes  \tau,\psi,\xi) = \omega(4\xi^2)^{-1}|4\xi^2|_F^{-2s+1}\omega_{E/F}(-1)\gamma_{\Gal}(s,{\rm As}\,\pi \otimes  \tau,\psi).
\end{align*}
Here $\omega = \omega_{\pi}\vert_{F^{\times}}\cdot\omega_{ \tau}$ and $\omega_{E/F}$ is the quadratic character of $F^{\times}$ associated with $E/F$ by local class field theory.
\end{thmb}
\begin{rmk}\label{R:1}\noindent
\begin{itemize}
\item[(1)] Let $n=1$. When $E=F\times F$, Theorem A was proved in \cite{Jac72} and \cite{Hen00}.  When $E$ is a field and $F$ is non-archimedean, the equality for $L$-factors follow from the results in \cite[Theorem 1.6]{AR05}, \cite[Section 1.5, Th\'eor\`eme]{Hen10}, and \cite[Theorem 4.2]{Mat10}. \item[(2)] Let $n=2$. When $E=F\times F$, Theorem B was proved in \cite[Theorem 3]{Ikeda1989}.
\end{itemize}
\end{rmk}

\begin{rmk}
Our theorems has an important application for construction of twisted triple product $p$-adic $L$-functions (\cite{Ish17}). We obtain an explicit interpolation formula for the twisted triple product $p$-adic $L$-function along Hida families through explicit computation of the local period integrals in Ichino's formula (\cite{Ich08}). Theorem A and Theorem B are essential to confirm the non-triviality and the good $p$-adical behavior of the local period integrals.
\end{rmk}

Combining Theorem B with \cite[Theorem 1.2]{WTG2008}, we can prove the dichotomy on trilinear forms.

\begin{thmc}[Corollary \ref{C:dichotomy of trilinear forms}]
Assume $\omega_{\pi}\vert_{F^{\times}}\cdot\omega_{ \tau}=1$. Then ${\rm Hom}_{\GL_2(F)}(\pi\otimes \tau,\C) \neq 0$ if and only if $$\omega_{E/F}(-1)\varepsilon_{\Gal}\left(\frac{1}{2},{\rm As}\,\pi\otimes \tau\right)=1.$$ 
Here $\omega_{E/F}$ is the quadratic character of $F^{\times}$ associated with $E/F$ by local class field theory.
\end{thmc}

\begin{rmk}\noindent
When $F$ is non-archimedean, except when both $\pi$ and $\tau$ are supercuspidal, the result is proved by Prasad in \cite{Pra92}. On the other hand,  when $F$ is archimedean, it follows from comparing the value of epsilon factors with a result of  Loke in \cite{Lok01}.
\end{rmk}

Now we switch to global situation. Let $F$ be a number field and $E$ be a quadratic semisimple $F$-algebra. Fix a non-trivial additive character $\psi$ of $\A_F/F$ and an element $\xi \in E^{\times}$ such that ${\rm tr}_{E/F}(\xi)=0$. Let $\pi$ and $ \tau$ be irreducible cuspidal automorphic representations of $\GL_2(\A_E)$ and $\GL_2(\A_F)$ with central characters $\omega_{\pi}$ and $\omega_{ \tau}$, respectively. Put $\omega = \omega_{\pi}\vert_{\A_F^{\times}}\cdot \omega_{ \tau}$. 

\begin{thmd}[Theorem \ref{T:twisted Asai factor}]
For each place $v$ of $F$, we have
\begin{align*}
L_{\rm PRS}(s,{\rm As}\,\pi_{v}\otimes \tau_{v}) &= L_{\rm Gal} (s,{\rm As}\,\pi_{v}\otimes \tau_{v}),\\
\varepsilon_{\rm PRS} (s,{\rm As}\,\pi_{v}\otimes \tau_{v},\psi_v,\xi) &= \omega_v(4\xi^2)^{-1}|4\xi^2|_{F_v}^{-2s+1}\omega_{E_v/F_v}(-1)\varepsilon_{\rm Gal} (s,{\rm As}\,\pi_{v}\otimes \tau_{v},\psi_v).
\end{align*}
\end{thmd}

\begin{rmk}\noindent
When $E=F\times F$, Theorem D was proved in \cite[Theorem 4.4.1]{Rama2000} as an application of the functoriality of $\pi$ to $\GL_4(\A_F)$. 
\end{rmk}

\subsection{Structure of the article}

The proof for Theorem A is separated into two cases according to $F$ is archimedean or not. 

In \S\,\ref{S:2}, we consider the case when $F$ is non-archimedean. In \S\,\ref{SS:2.1}, we recall the definition of Asai local factors defined by local zeta integrals and state the main results. To prove Theorem A, the problem boils down to two logically independent problems: (i) the case when is supercuspidal, (ii) multiplicative of gamma factors. By a standard global-to-local argument in Corollary \ref{equality of epsilon factors}, the case when $\pi$ is supercuspidal follow from the functoriality of global Asai transfer to $\GL_4$ proved in \cite{Kri03} and standard unramified calculation of local zeta integrals. The main innovation is the proof of multiplicative of gamma factors which is given in \S\,\ref{SS:2.2}. 

In $\S\ref{S3}$, we consider the case when $F$ is archimedean. Indeed,  we develop the theory concerning the analytic properties of the zeta integrals associated to the Asai representation for the case $E=\C$ and $F=\R$, and prove analogy results (Cf. \thmref{T:main theorem}) to that of \cite[Theorem 17.2]{Jac72} and \cite[Theorem 2.1]{Jac09}. To the best of our knowledge, unlike the non-archimedean case,  the analytic properties of the zeta integrals for this case can not be found in the literatures.

The proofs for Theorems B-D are given in \S\,\ref{S:4}. In \S\,\ref{S:4.1} and \S\,\ref{S:4.2}, we recall some notation and definitions of the intertwining operators on $\GSp_3$ and good sections for induced representations on $\GSp_3$. In \S\,\ref{S:4.3}, we recall the definition of twisted Asai local factors defined by local zeta integrals and state the main results. The proof of Theorems C and D are based on global-to-local arguments, the functoriality of global Asai transfer to $\GL_4$, and Theorem B. The proof of Theorem B is given by \S\,\ref{S:4.4}. We give a detailed proof following the idea of Ikeda in \cite[Theorem 3]{Ikeda1989}. Finally, in the Appendix, we prove some results which are used in the proofs in 
\secref{S:4}.

\subsection{Notation}
Denote $B$ the standard Borel subgroup of $\GL_2$ consisting of upper triangular matrices and $U$ be its unipotent radical. 

Let $F$ be a local field of characteristic zero. When $F$ is non-archimedean, denote $\mathcal{O}_F$ the ring of integers of F, $\varpi_F$ a prime element, and ${\rm ord}_F$ be the valuation on $F$ normalized so that ${\rm ord}_F(\varpi_F)=1$.  Let $|\cdot|_F$ be the absolute value on $F$ normalized so that $|\varpi_F|_F^{-1}$ is equal to the cardinality of $\mathcal{O}_F / \varpi_F\mathcal{O}_F$. When $F$ is archimedean, let $|\cdot|=|\cdot|_{\R}$ be the usual absolute value on $\R$ and $|z|_\C=z\overline{z}$ on $\C$. For a finite dimensional vector space $V$ over $F$, denote $\mf{S}(V)$ the space of Bruhat-Schwartz functions on $V$.

An additive character $\psi$ of $F$ means a continuous homomorphism $\psi : F \rightarrow \C^{\times}$. For $a \in F^{\times}$, let $\psi^a$ be the additive character defined by $\psi^a(x)=\psi(ax)$. When $F$ is non-archimedean, we denote $c(\psi)$ the smallest integer such that $\psi$ is trivial on $\varpi_F^{c(\psi)}\mathcal{O}_F$.  

A character $\chi$ of $F^{\times}$ means a continuous homomorphism $\chi : F^{\times}\rightarrow \C^{\times}$. When $F$ is non-archimedean, we denote $c(\chi)$ the smallest non-negative integer such that $\chi$ is trivial on $1+\varpi_F^{c(\chi)}\mathcal{O}_F$. Denote
$$L(s,\chi),\quad \varepsilon(s,\chi,\psi),\quad \gamma(s,\chi,\psi)$$  
the local factors of a character $\chi$ of $F^{\times}$ with respect to an additive character $\psi$ of $F$ defined in \cite{Tate1979}. When $\chi=1$, we denote $\zeta_F(s)=L(1,s)$. Then 
$$\zeta_F(s)=\begin{cases}(1-|\varpi_F|_F^s)^{-1} & \mbox{ if $F$ is non-archimedean},\\
\pi^{-s/2}\Gamma(s/2) & \mbox{ if $F=\R$},\\
2(2\pi)^{-s}\Gamma(s) & \mbox{ if $F=\C$}.
\end{cases}$$
Here $\Gamma(s)$ is the gamma function.

All representations of the $F$-points of a linear algebraic group over $F$ is assumed to be smooth. Let $\pi$ be a representation of $\GL_2(F)$ with finite length. Let $\psi$ be an additive character of $F$. The space of Whittaker functionals of $\pi$ with respect to $\psi$ has dimension less than or equal to one. When non-zero Whittaker functional exist, we denote $\ms{W}(\pi, \psi)$ to be the corresponding space of Whittaker functions. Recall that $\ms{W}(\pi, \psi)$ consists of smooth functions $W : \GL_2(F)\rightarrow \C$ such that
$$W\left(\bp 1 & x \\ 0 & 1 \ep g \right)=\psi(x)W(g)$$
for $x \in F$ and $g \in \GL_2(F)$. Moreover the right translation $\rho$ of $\GL_2(F)$ on $\ms{W}(\pi, \psi)$ is equivalent to $\pi$.  Recall that when $F$ is non-archimedean, smoothness means locally constant. Let $K$ be a maximal compact
subgroup of $\GL_2(F)$.  Denote by $\ms{W}(\pi,\psi)_0$ the subspace of right $K$-finite functions. Note that $\ms{W}(\pi,\psi)_0=\ms{W}(\pi,\psi)$ when $F$ is non-archimedean.

Let $\mu$ and $\nu$ be characters of $F^{\times}$. We recall the definition of the principal series representation ${\rm Ind}_{B(F)}^{\GL_2(F)}(\mu,\nu)$ of $\GL_2(F)$. Denote $\mathcal{B}(\mu,\nu)$ the space of smooth functions $f : \GL_2(F) \rightarrow \C$ such that
$$f\left(\bp a & b \\ 0 & d \ep g\right)=\mu(a)\nu(d)\left|\frac{a}{d}\right|_F^{{1}/{2}}f(g),$$
for $\bp a & b \\ 0 & d \ep \in B(F)$ and $g \in \GL_2(F)$. Denote $\rho$ the right translation of $\GL_2(F)$ on $\mathcal{B}(\mu,\nu)$. The representation $(\rho,\mathcal{B}(\mu,\nu))$ of $\GL_2(F)$ is denoted by ${\rm Ind}_{B(F)}^{\GL_2(F)}(\mu,\nu)$. Similarly, let $\mathcal{B}(\mu,\nu)_0$ be the subspace of right $K$-finite functions. Assume $F$ is non-archimedean, for $f \in \mathcal{B}(\mu,\nu)$, define $W_{\psi,f}\in \ms{W}(\pi, \psi)$ by
$$W_{\psi,f}(g) := \lim_{n\rightarrow\infty}\int_{\varpi_F^{-n}\intring{F}}f\left(\!\mat(0,-1,1,0)\mat(1,x,0,1)g\!\right)\psi(-x)\,dx.$$

\section{Asai local factors for the non-archimedean case}\label{S:2}
In this section, we consider the case $n=1$. Let $F$ be a non-archimedean local field of characteristic zero and $q$ be the cardinality of the residual field of $F$. Fix a non-trivial additive character $\psi$ of $F$.

\subsection{Asai local factors via Rankin-Selberg integrals}\label{SS:2.1}

Let $E$ be a quadratic semisimple $F$-algebra, namely, $E$ is equal to a quadratic extension field over $F$ or $F\times F$. If $E=F\times F$, we embed $F$ into $E$ via the diagonal embedding.  Fix an element $\xi \in E^{\times}$ such that ${\rm tr}_{E/F}(\xi)=0$. Let $\psi_{\xi}$ be an additive character of $E$ defined by $\psi_{\xi}(x) = \psi(\tr_{E/F}(\xi x))$.

Let $\pi$ be an irreducible generic representation of $\GL_2(E)$ with central character $\omega$. For any $\Phi\in \mf{S}(F^2)$ and $W\in \ms{W}(\pi,\psi_\xi)$, we define a function on $s\in\cpxnum$ by
\begin{equation}
Z(s,W,\Phi):=\int_{U(F)\backslash \GL_2(F)} W(g) \, \Phi((0,1)g) \, |\det(g)|_{F}^s \, dg. \label{zeta integral}
\end{equation}
where we normalize the invariant measure so that $\vol(\GL_2(\intring{F}),dg)=1$.  
We note that $Z(s,W,\Phi)$ converges absolutely for sufficiently large ${\rm Re}(s)$, and is analytically continued to the whole complex plane as a meromorphic function.  
Moreover, it is an element of $\cpxnum[q^{s}, q^{-s}]$. 
The $\cpxnum$-vector space generated by $Z(s,W,\Phi)$'s for $W\in\ms{W}(\pi,\psi_\xi)$ and $\Phi\in\mf{S}(F^2)$ is actually a fractional ideal of $\cpxnum[q^{s}, q^{-s}]$ containing $1$.  Thus, there exists $P(X)\in\cpxnum[X]$ with $P(0)=1$ such that $P(q^{-s})^{-1}$ is a generator of this ideal (see \cite[p.801]{Kab04} or \cite[Appendix, Theorem]{F93}).  

We define the {\em Asai $L$-function} by
\[L_{\rm RS}(s,\As\pi) := \frac{1}{P(q^{-s})}.\]
More generally, for any character $\chi\colon F^\times \rightarrow \cpxnum^\times$, we define
\[L_{\rm RS}(s, \As\pi\otimes\chi) :=L_{\rm RS}(s, \As(\pi\otimes\nami{\chi}))\]
where $\nami{\chi}\colon E^\times\rightarrow\cpxnum^\times$ is a character such that 
\[\nami{\chi}|_{F^\times}=\chi.\]
We note that this definition is independent of the choice of $\nami{\chi}$.
This function satisfies the following functional equation (see \cite[Appendix, Theorem]{F93}): For any $W\in \ms{W}(\pi, \psi_\xi)$, we have
\begin{align}
\label{FEasai}
\frac{Z(1-s,W\otimes \chi^{-1}\omega^{-1},\yama{\Phi})}{L_{\rm RS}(1-s,\As\pi^\vee\otimes\chi^{-1})}=
\varepsilon_{\rm RS}(s,\As\pi\otimes\chi, \psi, \xi)\frac{Z(s,W\otimes\chi,\Phi)}{L_{\rm RS}(s,\As\pi\otimes\chi)},
\end{align}
where there exists $c\in\cpxnum^\times$ and $m\in\ratint$ depending only on $\pi$, $\psi$, and $\xi$ such that 
\[\varepsilon_{\rm RS}(s,\As\pi\otimes\chi, \psi, \xi):=cq^{-ms},\]
and
\begin{equation}\label{fourier transform}
\yama{\Phi}(x,y):=\int_{F\times F}\Phi(u,v)\psi(uy-vx) \, du \, dv.
\end{equation}
Here $du\,dv$ is the {\em self-dual} measure associated with $F\times F\rightarrow\cpxnum; (x,y)\mapsto \psi(x+y)$. We note that the epsilon factor $\varepsilon_{\rm RS}(s,\As\pi,\psi,\xi)$ is the same as $\varepsilon(s,r(\pi),\psi_\xi)$ appeared in \cite[Appendix, Theorem]{F93} and different from $\varepsilon(s, \pi, \As, \psi_\xi)$ only by $\omega(-1)$ appeared in \cite[Theorem 3]{Kab04}.  In the case that $\chi$ is trivial, we denote  $\varepsilon_{\rm RS}(s, \As\pi\otimes\chi, \psi, \xi)$ by just $\varepsilon_{\rm RS}(s, \As\pi, \psi, \xi)$. Put
$$\gamma_{\rm RS}(s,{\rm As}\,\pi \otimes \sigma,\psi,\xi) = \varepsilon_{\rm RS}(s,{\rm As}\,\pi\otimes \sigma,\psi,\xi)L_{\rm RS}(1-s, {\rm As}\,\pi^{\vee} \otimes \sigma^{\vee})L_{\rm RS}(s, {\rm As}\,\pi \otimes \sigma)^{-1}.$$
\begin{rmk}\noindent\label{epsilon for split}
\begin{itemize}
\item[(1)] In the case of $E=F\times F$, let $\xi=(\xi_0, -\xi_0)$ and $\pi=\pi_1\otimes\pi_2$ where $\pi_i$ is an irreducible generic representation of $\GL_2(F)$ with central character $\omega_i$.  The definition of Asai $L$-function is the same as that defined in \cite[Theorem 14.8, (1)]{Jac72}.  With the epsilon factors, the relation of $\varepsilon_{\rm RS}$ and one defined in \cite[Theorem 14.8, (3)]{Jac72} is given as follows:
\begin{align*}
\varepsilon_{\rm RS}(s, \As\pi, \psi, \xi)&=\omega_2(-1)\omega(\xi_0)|\xi_0|_F^{2s-1}\varepsilon(s, \pi_1\otimes\pi_2, \psi)\\
&=\omega(\xi)|\xi|_E^{s-1/2}\varepsilon(s, \pi_1\otimes\pi_2, \psi)
\end{align*}
where the epsilon factor $\varepsilon(s, \pi_1\otimes\pi_2, \psi)$ in the right hand side is defined in \cite[Theorem 14.8, (3)]{Jac72}. 
\item[(2)] If we drop the assumption that $\pi$ is irreducible, but assume that $\pi$ has finite length and
the space of Whittaker functionals of $\pi$ with respect to $\psi_{\xi}$ is one-dimensional. 
Then, by \cite[Theorem 1]{Kab04}, we can define Asai local factors for $\pi$ following the same definition. Moreover, if $\pi'$ is any irreducible generic subquotient of $\pi$, then
$$\gamma_{RS}(s,{\rm As}\,\pi',\psi,\xi)=\gamma_{RS}(s,{\rm As}\,\pi,\psi,\xi).$$
\end{itemize}
\end{rmk}
We note that for any $a\in F^\times$, we have
\begin{align}\label{E:dependence on psi}
\begin{split}\
\varepsilon_{\rm RS}(s,\As\pi,\psi^a, \xi)&=\omega^2(a) \, |a|_{F}^{4s-2} \, \varepsilon_{\rm RS}(s,\As\pi,\psi, \xi),\\
\varepsilon_{\rm RS}(s, \As\pi,\psi, a\xi)&=\omega(a) \, |a|_{F}^{2s-1} \, \varepsilon_{\rm RS}(s,\As\pi,\psi, \xi).
\end{split}
\end{align}

 Let $\sigma$ be the non-trivial automorphism of $E$ over $F$. If $\chi$ is a character of $E^{\times}$, let $\chi^{\sigma}$ be a character of $E^{\times}$ defined by $\chi^{\sigma}(x)=\chi(x^{\sigma})$.
 
Following theorem is the main result of this section.
\begin{thm}[Multiplicativity of gamma factors]\label{multiplicativity for ps}
Assume $E$ is a field, and $c(\psi)=c(\psi_{\xi})=0$. Let $\pi$ be an irreducible generic subquotient of a principal series representation ${\rm Ind}_{B(E)}^{\GL_2(F)}(\mu,\nu)$. We have
\begin{align*}
\gamma_{\rm RS}(s, \As\pi, \psi, \xi)&=\nu(-1)\gamma(s,\mu|_{F^\times},\psi)\gamma(s, \nu|_{F^\times}, \psi)\gamma(s, \mu\nu^\sigma, \psi_\xi).
\end{align*}
\end{thm}

\begin{corollary}\label{equality of epsilon factors}
We have
 \begin{align*}
\varepsilon_{\rm RS}(s,{\rm As}\,\pi,\psi,\xi) &= \omega(\xi)|\xi^2|_F^{s-1/2}\lambda_{E/F}(\psi)^{-1}\varepsilon_{\rm Gal}(s,{\rm As}\,\pi,\psi)
 \end{align*}
Here $\lambda_{E/F}(\psi)$ is the Langlands constant for $E/F$ with respect to $\psi$. 
\end{corollary}

\begin{proof}
As we mentioned in Remarks \ref{R:1}-(1) and \ref{epsilon for split}-(1), the equality for $\varepsilon$-factors is equivalent to the equality for $\gamma$-factors and it is  suffices to consider the case when $E$ is a field.

First we assume $\pi$ be an irreducible generic subquotient of a principal series representation ${\rm Ind}_{B(E)}^{\GL_2(F)}(\mu,\nu)$. By (\ref{E:dependence on psi}), we may assume  $c(\psi)=c(\psi_{\xi})=0$. By the property of Asai representation defined by (\ref{E:Asai representation}) (Cf. \cite[Lemma 7.1 (d)]{Pra92}), we have
\begin{align*}
\gamma_{\Gal}(s,\, \As\pi,\, \psi)
&= \lambda_{E/F}(\psi)\gamma(s,\, \mu|_{F^\times},\, \psi)\gamma(s,\, \nu|_{F^\times},\, \psi)\gamma(s,\, \mu\nu^\sigma,\, \psi\circ\tr_{E/F})\\
&=\nu(-1)\omega(\xi)^{-1}|\xi^2|_F^{1/2-s}\gamma(s,\, \mu|_{F^\times},\, \psi)\gamma(s,\, \nu|_{F^\times},\, \psi)\gamma(s,\, \mu\nu^\sigma,\, \psi_\xi).
\end{align*}
The last equality follows from the fact that $|\xi|_{E}=|\varpi_{F}|_{F}^{-c(\omega_{E/F})}$, since we assume that $c(\psi)=c(\psi_\xi)=0$. The assertion then follows from Theorem \ref{multiplicativity for ps}.

Assume $\pi$ is a supercuspidal representation. Let ${\bf E}/\bfb{F}$ be a quadratic extension of number fields such that there exist a finite place $v_0$ of $\bf F$ such that ${\bf E}_{v_0}=E$ and ${\bf F}_{v_0}=F$. Fix a non-trivial additive character ${\bm \psi}$ of $\A_{\bf F}/{\bf F}$ and an element $\bm{ \xi} \in {\bf E}^{\times}$ such that ${\rm tr}_{{\bf E}/{\bf F}}({\bm \xi})=0$. By \cite[Proposition 5.1]{Sha90}, there exist an irreducible cuspidal automorphic representation $\bm \pi$ of $\GL_2(\A_{\bf E})$ such that
\begin{itemize}
\item ${\bm \pi}_{v_0}=\pi$.
\item ${\bm \pi}_{v}$ is spherical for any finite place $v \neq v_0$.
\end{itemize}
Denote $\bm \omega$ the central character of $\bm \pi$. By Theorem \ref{multiplicativity for ps} and (\ref{E:relation between epsilon factors}), for all places $v \neq v_0$ of {\bf F}, we have
\begin{align}\label{E:2.1.4}
\gamma_{\rm RS}(s,{\rm As}\,{\bm \pi}_v,{\bm \psi}_v,\xi) = {\bm \omega}_v(\xi)|\xi^2|_{{\bf F}_v}^{s-1/2}\lambda_{{\bf E}_v/{\bf F}_v}(\psi)^{-1}\gamma_{\rm Gal}(s,{\rm As}\,{\bm \pi}_v,{\bm \psi}_v).
\end{align}
On the other hand, by \cite[Theorem 6.7]{Kri03}, the irreducible admissible representation ${\rm As}\,{\bm \pi} = \otimes_v{\rm As}\,{\bm \pi}_{v}$ is an isobaric automorphic representation of $\GL_4(\A_{\bf F})$. Since ${\rm As}\,{\bm \pi}$ is isobaric, it follows from the global functional equation for Rankin-Selberg $L$-functions that the global automorphic $L$-function $L_{\rm Gal}(s,{\rm As}\,{\bm \pi}) = \prod_{v}L_{\rm Gal}(s,{\rm As}\,{\bm \pi}_{v})$ has meromorphic continuation to $s \in \C$ and satisfies the functional equation 
\begin{align}\label{E:2.1.5}
L_{\rm Gal}(s,{\rm As}\,{\bm \pi}) = \varepsilon_{\rm Gal}(s,{\rm As}\,{\bm \pi}) L_{\rm Gal}(1-s,{\rm As}\,{\bm \pi}^{\vee}).
\end{align}
The assertion then follows from (\ref{E:dependence on psi}), (\ref{E:2.1.4}), (\ref{E:2.1.5}), and the global functional equation for $L_{\rm RS}(s,{\rm As}\,{\bm \pi})$. This completes the proof.

\end{proof}

\subsection{Proof of Theorem \ref{multiplicativity for ps}} \label{SS:2.2}
As explained in Remark \ref{epsilon for split}-(2), we may assume $\pi = {\rm Ind}_{B(E)}^{\GL_2(E)}(\mu,\nu)$ even though the induced representation might not be irreducible. 

It suffices to prove that there exist $W \in \ms{W}(\pi,\psi_\xi)$ and $\Phi \in \mf{S}(F^2)$ such that $Z(s,W,\Phi) \neq 0$ and
\begin{align*}
Z(1-s,W,\widehat{\Phi})&=\nu(-1)\gamma(s,\mu|_{F^\times},\psi)\gamma(s, \nu|_{F^\times}, \psi)\gamma(s, \mu\nu^\sigma, \psi_\xi)Z(s,W,\Phi).
\end{align*}
We have the following three cases: 
\begin{itemize}
\item[(i)] Exactly one of $\mu$ and $\nu$ is unramified, 
\item[(ii)] Both $\mu$ and $\nu$ are ramified, \label{case2}
\item[(iii)] Both $\mu$ and $\nu$ are unramified.
\end{itemize}
In case (iii), take $W \in \ms{W}(\pi,\psi_\xi)$ be a non-zero Whittaker function fixed by $\GL_2(\mathcal{O}_E)$, $\Phi = \mathbb{I}_{\mathcal{O}_F\oplus\mathcal{O}_F}\in \mf{S}(F^2)$, and $\chi=1$. With this datum, we can calculate both sides of \ref{FEasai} and deduce the formula for gamma factors. The calculation is standard and is left to the readers.

\subsubsection{The proof of case (i)}
\label{epsilon for special rep}
\begin{proof}
We may assume that $\mu$ is ramified and $\nu$ is unramified. Twist by $\nu^{-1}$, we may further assume $\nu=1$.  Put
$$r:=\lceil c(\mu)/e_{E/F} \rceil.$$
Here $e_{E/F}$ is the ramified index for $E/F$.  
Let $f\in\mathcal{B}(\mu,1)$ be a section characterized by 
\[f\left(\!\mat(1,0,x,1)\!\right)=\mu^{-1}(x)|x|_{E}^{-1}\mathbb{I}_{|x|_E\ge1}(x).\]
Note that $\rho(k)f=f$ for $k = \bp a& b \\c&d\ep \in \GL_2(\mathcal{O}_E)$ with $c \in \varpi_E^{c(\mu)}\mathcal{O}_E$ (Cf. \cite[Section 2.1]{Sc02}).

We divide this problem into the following two cases:
\begin{enumerate}
\item $\mu|_{F^\times}$ is ramified, \label{ke-su  1}
\item $\mu|_{F^\times}$ is unramified.\label{ke-su 2}
\end{enumerate}

\noindent {\bf Case (\ref{ke-su 1}): $\mu|_{F^\times}$ is ramified}\\
Denote $c_\mu=c(\mu|_{F^\times})$ and  
\[
\delta:=
\begin{cases}
0&\text{ if }c(\mu)/e_{E/F}\in\ratint,\\
1&\text{ if }c(\mu)/e_{E/F}\notin\ratint.
\end{cases}
\]
Note that when $E/F$ is ramified, $c(\chi)$ is even for any character $\chi$ of $E^{\times}$ with $\chi|_{F^\times}=1$. In particular, if $\delta=1$, then $c_\mu=r$. 

Let
\begin{align*}
g&:=|\varpi_F|_F^{-r}\int_{\varpi^{c_\mu}\intring{F}}\rho\!\left(\!\mat(1,0,x,1)\!\right)f\,dx.
\end{align*}
We explicitly determine $W_{\psi_\xi, g}\left(\!\mat(a,0,0,1)\!\right)$. Let $\theta\in\intring{E}$ be an element such that $\intring{E}=\intring{F}[\theta]$ and $$\psi_\xi(c+d\theta)=\psi(d)$$
for $c,d \in \mathcal{O}_F$. 
 By straightforward computation, for $a \in F^{\times}$, we have
\begin{align*}
&W_{\psi_\xi, g}\left(\!\mat(a,0,0,1)\!\right)\\
&=|a|_{F}\int_E\sum_{u\in\intring{F}/\varpi_{F}^{r-c_\mu}\intring{F}}\mu(x^{-1})|x|_F^{-1}f\left(\!\mat(1,0,x^{-1}+\varpi_F^{c_\mu}u,1)\!\right)\psi_{\xi}(-ax)\,dx\\
&=|a|_{F}\int_E\sum_{u\in\intring{F}/\varpi_{F}^{r-c_\mu}\intring{F}}\mu^{-1}(1+\varpi_{F}^{c_\mu}xu)|1+\varpi_{F}^{c_\mu}xu|_E^{-1}\mathbb{I}_{E\setminus\varpi_E\intring{E}}(x^{-1}+\varpi_F^{c_\mu}u)\psi_{\xi}(-ax)\,dx\\
&=|a|_{F}\int_{\intring{E}}\sum_{u\in\intring{F}/\varpi_{F}^{r-c_\mu}\intring{F}}\mu^{-1}(1+\varpi_{F}^{c_\mu}xu)\psi_{\xi}(-ax)\,dx\\
&=|a|_{F}\sum_{m=0}^{c_\mu}\int_{\intring{E}}A_m(x)\psi_{\xi}(-ax)\,dx\\
&=|a|_{F}\sum_{m=1}^{r-c_\mu}
\int_{\intring{F}}A_m(x\theta)\psi(-ax)\,dx
\end{align*}
where for $m=0,\dots, r-c_\mu$, we put
\[A_m(x):=\sum_{u\in\left(\intring{F}/\varpi_{F}^{r-c_\mu-m}\intring{F}\right)^\times}\mu^{-1}(1+\varpi_{F}^{c_\mu+m}xu).\]
Note that if $m<r-c_\mu$ and $a\notin \varpi_F^{-m-c_{\mu}+r}\mathcal{O}_F$, we have
\[\int_{\intring{E}}A_m(x)\psi_{\xi}(-ax)\,dx=0.\]
Therefore, if $\delta=1$, for $a \in F^{\times}$ we have
\begin{align*}
W_{\psi_\xi, g}\left(\!\mat(a,0,0,1)\!\right)=|a|_F\mathbb{I}_{\intring{F}}(a).
\end{align*}

Suppose $\delta=0$.
Let $m=r-c_\mu$, then for $a \in F^{\times}$ we have
\[\int_{\intring{F}}A_m(x\theta)\psi(-ax)\,dx=\mathbb{I}_{\intring{F}}(a).\]
Let $0\le m<r-c_\mu$. Since $A_m(xv)=A_m(x)$ for any $v\in\intring{F}^\times$, we have
\begin{align*}
&\int_{\intring{F}}A_m(x\theta)\psi(-ax)\,dx\\
&=\int_{\intring{F}}A_m(x\theta)\frac{|\varpi_{F}|_{F}^{r-m-c_\mu}}{1-|\varpi_{F}|_{F}}\sum_{w\in(\intring{F}/\varpi_{F}^{r-m-c_\mu}\intring{F})^\times}\psi(-\varpi_{F}^{m+c_\mu-r}wx)\,dx\cdot\mathbb{I}_{\varpi_{F}^{m+c_\mu-r}\intring{F}^\times}(a).
\end{align*}
Since for any $k\in\ratint$, $x\in\intring{F}$ with $r-k\ge1$,
\[\frac{|\varpi_{F}|_{F}^{r-k}}{1-|\varpi_{F}|_{F}}\sum_{w\in(\intring{F}/\varpi_{F}^{r-k}\intring{F})^\times}\psi(-\varpi_{F}^{k-r}wx)
=
\begin{cases}
0&x\notin\varpi_{F}^{r-k-1}\intring{F},\\
-(|\varpi_{F}|_F^{-1}-1)^{-1}&x\in\varpi_{F}^{r-k-1}\intring{F}^\times,\\
1&x\in\varpi_{F}^{r-k}\intring{F}.
\end{cases}\]
Hence, for $a \in F^{\times}$ we have
\begin{align*}
&\int_{\intring{F}}A_m(x\theta)\psi(-ax)\,dx\\
&=\left[\left(1-|\varpi_{F}|_{F}\right)-|\varpi_{F}|_{F}^{r-m-c_\mu}\sum_{u\in\left(\intring{F}/\varpi_{F}^{r-m-c_\mu}\intring{F}\right)^\times}\mu^{-1}(1+\varpi_{F}^{r-1}u)\right]\mathbb{I}_{\varpi_F^{m+c_\mu-r}\intring{F}^\times}(a)\\
&=(1-|\varpi_{F}|_{F}+|\varpi_{F}|_{F})\mathbb{I}_{\varpi_F^{m+c_\mu-r}\intring{F}^\times}(a)\\
&=\mathbb{I}_{\varpi_F^{m+c_\mu-r}\intring{F}^\times}(a).
\end{align*}
Thus, for $a \in F^{\times}$ we have
\begin{align*}
W_{\psi_\xi, g}\left(\!\mat(a,0,0,1)\!\right)=|a|_{F}\mathbb{I}_{\varpi_{F}^{c_\mu-r}
\intring{F}}(a).
\end{align*}
Therefore, in any case, for $a\in F^\times$ we have
\begin{align}\label{E:2.2.1}
W_{\psi_\xi, g}\left(\!\mat(a,0,0,1)\!\right)=|a|_{F}\mathbb{I}_{\varpi_{F}^{c_\mu-r}
\intring{F}}(a).
\end{align}

Note that for $x\in \intring{E}$ and $y \in E^{\times}$, we have
\begin{align*}
\rho(w_1)W_{\psi_\xi, f}\left(\!\mat(y,0,x,1)\!\right)\omega(y)^{-1}=\varepsilon(1/2, \pi, \psi_\xi)|\varpi_E^{c(\mu)}y|_E^{1/2}\mathbb{I}_{\intring{E}}(\varpi_E^{c(\mu)}y).
\end{align*}
Therefore, for $x\in\intring{F}$ and $y\in F^\times$, we have
\begin{align}\label{E:2.2.2}
\rho(w_1)W_{\psi_\xi, g}\left(\!\mat(y,0,x,1)\!\right)\omega(y)^{-1}=|\varpi_F|_F^{2r-\delta}\varepsilon(0, \pi, \psi_\xi)|y|_F\mathbb{I}_{\varpi_F^{-r+\delta}\intring{F}}(y).
\end{align}

Let $W=W_{\psi_\xi, f}$, $\Phi(x,y)=\mathbb{I}_{\varpi_F^{c_\mu}\intring{F}}(x)\cdot\mathbb{I}_{1+\varpi^{c_\mu}\intring{F}}(y)$, and $\chi=1$. With this datum,  the functional equation (\ref{FEasai}) reduces to
\begin{align*}
&\int_{\intring{F}}\int_{F^\times}\rho(w_1)W_{\psi_\xi, f}\left(\!\mat(y,0,x,1)\!\right)\omega^{-1}(y)|y|_F^{-s}\,d^\times_Fy\,d_Fx\\
&=\frac{\gamma_{\rm RS}(s, \As\pi, \psi,\xi)}{|\varpi_F|_F^{c_\mu s}\varepsilon(s,\mu|_{F^\times},\psi)\zeta_F(s)}|\varpi_F|_F^{r}\int_{F^\times}W_{\psi_\xi, g}\left(\!\mat(y,0,0,1)\!\right)|y|_F^{s-1}\,d^\times_Fy.
\end{align*}
By formulae (\ref{E:2.2.1}) and (\ref{E:2.2.2}), the above equation reduces to
\begin{align*}
&|\varpi_F|_F^{r(1-s)}\zeta_F(1-s)\varepsilon(s, \mu, \psi_\xi)=|\varpi_F|_F^{r(1-s)}\zeta_F(s)\frac{\gamma_{\rm RS}(s, \As\pi, \psi,\xi)}{\varepsilon(s,\mu|_{F^\times},\psi)}.
\end{align*}
In other words, we have
\[\gamma_{\rm RS}(s, \As\pi, \psi,\xi)=\gamma(s, \mu|_{F^\times}, \psi)\gamma(s,1,\psi)\gamma(s, \mu, \psi_\xi).\]
This completes the proof of case (1).

\noindent {\bf Case (\ref{ke-su 2}): $\mu|_{F^\times}$ is unramified}
Note that in this case, $c(\mu)$ is alway even when $E/F$ is ramified extension.
Let
\begin{align*}
h&:=\int_{\GL_2(\intring{F})}\rho(k)f dk\\
&=\sum_{u\in\intring{F}/\varpi_{F}^{r-1}\intring{F}}\rho\!\left(\!\mat(1,0,\varpi_{F}u,1)\!\right)f+\sum_{u\in\intring{F}/\varpi_{F}^r\intring{F}}\rho\!\left(\!w_1\mat(1,0,u,1)\!\right)f.
\end{align*}
be a $\GL_2(\intring{F})$-invariant vector.
We directly compute $W_{\psi_\xi, h}\left(\!\mat(a,0,0,1)\!\right)$
for $a\in F^\times$.
Let
\begin{align*}
h_1:=\sum_{u\in\intring{F}/\varpi_{F}^{r-1}\intring{F}}\pi\!\left(\!\mat(1,0,\varpi_{F}u,1)\!\right)f,\quad h_2:=\sum_{u\in\intring{F}/\varpi_{F}^r\intring{F}}\pi\!\left(w_1\!\mat(1,0,u,1)\!\right)f.
\end{align*}
The function $W_{\psi_\xi, h_1}$ is computed in the totally same way with that for $W_{\psi_\xi, g}$ in the previous case, and we have
\[W_{\psi_\xi, h_1}\left(\!\mat(a,0,0,1)\!\right)=|a|_F\mathbb{I}_{\varpi_F^{1-r}\intring{F}}(a).\]
We compute $W_{\psi_\xi, h_2}\left(\!\mat(a,0,0,1)\!\right)$ for $a\in F^\times$.  
\begin{align*}
&W_{\psi_\xi, h_2}\left(\!\mat(a,0,0,1)\!\right)\\
&=
|a|_{F}\sum_{u\in\intring{F}/\varpi_{F}^{r}\intring{F}}\int_{E}\mu^{-1}(x+u)|x+u|_{E}^{-1}\mathbb{I}_{|x|_E\ge1}(x+u)\psi_\xi(-ax)\,dx\\
&=\mu(a)|a|_{F}|\varpi_{F}|_{F}^{-r}\varepsilon(1,\mu,\psi_\xi)\mathbb{I}_{\varpi_{F}^{-r}\intring{F}}(a).
\end{align*}
By \cite[Th\'eor\`em 3.2]{Del76}, we have
\[\varepsilon(1,\mu,\psi_\xi)=\mu(\varpi_F)^r|\varpi_F|_F^{r}.\]
Therefore, for $a \in F^{\times}$ we have
\begin{align}\label{E:2.2.3}
W_{\psi_\xi, h}\left(\!\mat(a,0,0,1)\!\right)
=\mu(a)|a|_{F}\mu(\varpi_{F})^{r}\mathbb{I}_{\varpi_{F}^{-r}\intring{F}}(a)+|a|_{F}\mathbb{I}_{\varpi_{F}^{1-r}\intring{F}}(a).
\end{align}

Let $W=W_{\psi_\xi, f}$, $\Phi=\mathbb{I}_{\mathcal{O}_F\oplus\mathcal{O}_F}$, and $\chi=1$. With this datum, the functional equation (\ref{FEasai}) reduces to
\begin{align*}
&L(2-2s,\mu^{-1}\vert_{F^{\times}})\int_{\GL_2(\intring{F})}\int_{F^\times}W\left(\!\mat(y,0,0,1)k\!\right)\mu^{-1}(y)|y|_F^{-s}d^\times_{F}y\,dk\\
&=
{L(2s, \mu|_{F^\times})}\gamma_{RS}(s,{\rm As}\,\pi,\psi,\xi)\int_{\GL_2(\intring{F})}\int_{F^\times}W\left(\!\mat(y,0,0,1)k\!\right)|y|_F^{s-1}\,d^\times_{F}y\,dk .
\end{align*}
By formula (\ref{E:2.2.3}), the above equation reduces to 
\begin{align*}
\mu(\varpi_F)^r|\varpi_F|_F^{r(s-1)}\zeta_F(1-s)L(1-s,\mu^{-1}\vert_{F^{\times}}) = |\varpi_F|_F^{-rs}\zeta_F(s)L(s,\mu\vert_{F^{\times}})\gamma(s,{\rm Ad}\,\pi,\psi,\xi).
\end{align*}
In other words, we have
\[\gamma_{\rm RS}(s, \As\pi, \psi,\xi)=\mu(\varpi_F)^r|\varpi_F|_F^{r(2s-1)}\gamma(s, \mu|_{F^\times}, \psi)\gamma(s,1,\psi).\]
By \cite[Th\'eor\`em 3.2]{Del76}, we have
\[\varepsilon(s,\mu,\psi_\xi)=\mu(\varpi_F)^r|\varpi_F|_F^{r(2s-1)}.\]
This completes the proof of case (2).
\end{proof}

\subsubsection{The proof of case (ii)}
\begin{proof}
Denote $\mu \boxplus \nu = {\rm Ind}_{B(E)}^{\GL_2(E)}(\mu,\nu)$. Note that $${\rm As}\,(\mu\boxplus\nu) = {\rm As}\,(\mu\nu^{-1}\boxplus1)\otimes\nu\vert_{F^{\times}}= {\rm As}\,(1\boxplus\mu^{-1}\nu)\otimes\mu\vert_{F^{\times}}.$$
Therefore, we are in case (i) or case (iii) when one of $\mu\vert_{F^{\times}}$ or $\nu\vert_{F^{\times}}$ is unramified. 

Assume both $\mu\vert_{F^{\times}}$ and $\nu\vert_{F^{\times}}$ are ramified, and let $c_{\mu} = c(\mu\vert_{F^{\times}})$, $c_{\nu}=c(\nu\vert_{F^{\times}})$, and $c_{\omega} = c(\omega\vert_{F^{\times}})$. Let $\Phi(x,y)=\nu(x)^{-1}\mathbb{I}_{\varpi_F^{-c_{\nu}}\intring{F}^\times}(x)
\cdot\mathbb{I}_{1+\varpi_F^{\max\{c_\mu, c_\nu, c_{\omega}\}}\intring{F}}(y)$ and $\chi=1$. With this datum, the functional equation (\ref{FEasai}) reduces to 
\begin{align}\label{main formula 2}
\begin{split}
&\int_{\varpi_{F}^{-c_\mu}\intring{F}^\times}\int_{F^\times}W\left(\!\mat(y,0,x,1)\!\right)\omega^{-1}(y)|y|_F^{-s}\mu(x)|x|_F^{s-1}\,d^\times_{F}y\,dx
\\&=\frac{\mu(-1)\gamma(s,\As\pi,\psi,\xi)}{\varepsilon(s,\nu|_{F^\times},\psi)\varepsilon(s,\,\mu|_{F^\times},\,\psi)}\int_{\varpi_{F}^{-c_\nu}\intring{F}^\times}\int_{F}W\left(\!\mat(y,0,x,1)\!\right)|y|_F^{s-1}\nu^{-1}(x)|x|_F^{-s}\,d^\times_{F}y\,dx.
\end{split}
\end{align}
For $m,n \in \Z_{\geq 0}$, let $\phi_{m,n}(x,y,z):F\times F^\times\times E\rightarrow\cpxnum$ be a function with compact support defined by
\[\phi_{m,n}^{\nu}(x,y,z):=\mathbb{I}_{\varpi_F^{-c_\nu}\intring{F}^\times}(x)\cdot\mathbb{I}_{|\varpi_F|_F^m\le |y|_F \le |\varpi_F|_F^{-m}}(y)\cdot\mathbb{I}_{|\varpi_E|_E^{n}\le|z|_E\le|\varpi_E|_E^{-n}}(z).\]
For $f \in \mathcal{B}(\mu,\nu)$ and $m,n\in\Z_{\geq0}$, define the integral 
\begin{align*}
I_{m,n}(s,\mu,\nu;f):=&\int_F\int_{F^\times}\int_E f\left(\!\mat(0,-1,1,0)\mat(1,z,0,1)\mat(y,0, x,1)\!\right)\\
&\times\psi_\xi(-z)|y|_F^{s-1}\nu^{-1}(x)|x|_F^{-s}\phi_{m,n}^\nu(x,y,z)\,d_Fx\,d_F^\times y\,d_Ez.
\end{align*}
Note that the limit
\[\lim_{m\rightarrow\infty}\left(\lim_{n\rightarrow\infty}I_{m,n}(s,\mu,\nu;f)\right)\quad \left(\text{resp. }\lim_{m\rightarrow\infty}\left(\lim_{n\rightarrow\infty}I_{m,n}(1-s,\nu^{-1},\mu^{-1};f)\right)\right)\]
is equal to the integral in the right (resp. left) hand side of the formula (\ref{main formula 2}).
By direct computation, we have
\begin{align}\label{E:2.2.5}
\begin{split}
&I_{m,n}(s,\mu,\nu;f)\\
&=
\int\int\int\nu\mu^{-1}(z)|z|_E^{-1}f\left(\!\mat(y,0,yz^{-1}+x,1)\!\right)
\psi_\xi(-z)|y|_F^{s-1}\nu^{-1}(x)|x|_F^{-s}\\
&\hspace{10pt}\times\phi_{m,n}^\nu(x,y,z)\,d_Fx\,d_F^\times y\,d_Ez\\
&\overset{z\rightsquigarrow zy}{=}
\int\int\int\nu\mu^{-1}(zy)|zy|_E^{-1}f\left(\!\mat(y,0,z^{-1}+x,1)\!\right)
\psi_\xi(-zy)|y|_F^{s+1}\nu^{-1}(x)|x|_F^{-s}\\
&\hspace{20pt}\times\phi_{m,n}^\nu(x,y,zy)\,d_Fx\,d_F^\times y\,d_Ez\\
&\overset{y\rightsquigarrow y/(zz^\sigma)}{=}
\int\int\int\nu\mu^{-1}(y/z^\sigma)|y/z^\sigma|_E^{-1}f\left(\!\mat(y/(zz^\sigma),0,z^{-1}+x,1)\!\right)\\
&\hspace{30pt}\times\psi_\xi(-y/z^\sigma)|y/zz^\sigma|_F^{s+1}\nu^{-1}(x)|x|_F^{-s}\phi_{m,n}^\nu\left(x,y/(zz^\sigma),y/z^\sigma\right)\,d_Fx\,d_F^\times y\,d_Ez\\
&\overset{z\rightsquigarrow 1/z}{=}
\int\int\int\nu\mu^{-1}(yz^\sigma)|yz^\sigma|_E^{-1}|z|_E^{-2}f\left(\!\mat(yzz^\sigma,0,z+x,1)\!\right)\\
&\hspace{30pt}\times\psi_\xi(-yz^\sigma)|yzz^\sigma|_F^{s+1}\nu^{-1}(x)|x|_F^{-s}\phi_{m,n}^\nu\left(x,yzz^\sigma,yz^\sigma\right)\,d_Fx\,d_F^\times y\,d_Ez\\
&=
\int\int\int\nu^{-1}(x)|x|_F^{-s}\cdot\nu(y)|y|_F^{s}\cdot\mu\nu^\sigma(z)|z|_E^{s-1}f\left(\!\mat(1,0,z+x,1)\!\right)\\
&\hspace{30pt}\times\psi_\xi(yz)\phi_{m,n}^\nu\left(x,yzz^\sigma,yz^\sigma\right)\,d_Fx\,d_F^\times y\,d_Ez\\
&=
\frac{\zeta_F(1)}{\zeta_E(1)}\int\int\int\nu^{-1}(x)|x|_F^{-s}\cdot\nu(y)|y|_F^{s-1}\cdot\mu\nu^\sigma(z)|z|_E^{s}f\left(\!\mat(1,0,z+x,1)\!\right)\\
&\hspace{30pt}\times\psi_\xi(yz)\phi_{m,n}^\nu\left(x,yzz^\sigma,yz^\sigma\right)\,d_Fx\,d_Fy\,d_E^\times z.
\end{split}
\end{align}

Now we specify $f$. Let $\theta\in\intring{E}$ be an element such that
\begin{align}\label{E:theta}
 \intring{E}=\intring{F}[\theta],\quad \psi_\xi(a+b\theta)=\psi(b)
\end{align}
for any $a,b\in F$. For $x\in E$, we write $x=a_x+b_x\theta$ for a unique pair $(a_x,b_x)\in F$.
Fix a choice of $\eta \in F^{\times}$ and $M \in \Z_{\geq 0}$ such that
\begin{align}\label{E:2.2}
\begin{split}
&|\eta\varpi_F^M|_F<|\eta^2\varpi_F^{c_\omega}|_F<|\varpi_F|_F^{c(\mu\nu^{\sigma})/e_{E/F}},\quad|\varpi_F|_F^M \leq |\varpi_F|_F^{c_{\nu}}.
\end{split}
\end{align}
Let $\Phi_{\eta, M} \in \mf{S}(E)$ be a Bruhat-Schwartz on $E$ defined by
$$\Phi_{\eta, M}(x)=\psi(\eta a_x)\mathbb{I}_{\varpi_F^{-M}\intring{F}}(a_x)\cdot\mathbb{I}_{\eta+\varpi_F^M\intring{F}}(b_x).$$
Let $N=N(\eta,M) \in \Z_{\geq 0}$ such that for $z \in E$ with  $\Phi_{\eta,M}(z)\neq 0$, we have
$$|\varpi_E|_E^N \leq|z|_E\leq |\varpi_E|_E^{-N}.$$ Define $f \in \mathcal{B}(\mu,\nu)$ be the section characterized by
\begin{align*}
f\left(\!\mat(1,0,x,1)\!\right)&=\Phi_{\eta,M}(x).
\end{align*}
Take $m,n$ sufficiently large so that
\begin{align}\label{E:2.2.6}
|\varpi_F|_F^{m-N}|\eta|_F \leq |\varpi_F|_F^{-c_{\nu}} \leq |\varpi_F|_F^{-m+N}|\eta|_F, \quad |\varpi_E|_E^{n-N}|\eta|_E \leq |\varpi_F|_E^{-c_{\nu}} \leq |\varpi_E|_E^{-n+N}|\eta|_E.
\end{align}
With this choice of $m,n$, and $f$ in (\ref{E:2.2.5}), we have
\begin{align*}
&I_{m,n}(s,\mu,\nu;f)\\
&=
\frac{\zeta_F(1)}{\zeta_E(1)}\int\int\int\nu^{-1}(x)|x|_F^{-s}\psi(\eta x)\cdot\nu(\eta)^{-1}|\eta|_F^{-s}\nu(y)|y|_F^{s-1}\psi(y)\cdot\mu\nu^\sigma(z)|z|_E^{s}\psi(\eta a_z)\\
&\hspace{30pt}\times\mathbb{I}_{\varpi_F^{-M}\intring{F}}(a_z)\mathbb{I}_{\eta+\varpi_F^M\intring{F}}(b_z)\phi_{m,n}^\nu\left(x,\eta^{-1}yzz^\sigma,\eta^{-1}yz^\sigma\right)\,d_Fx\,d_Fy\,d_E^\times z\\
&=\frac{\zeta_F(1)}{\zeta_E(1)}|\eta|_F^{-1}\varepsilon(s,\nu\vert_{F^{\times}},\psi)\varepsilon(1-s,\nu\vert_{F^{\times}}^{-1},\psi)\\
&\hspace{30pt}\times\int\mu\nu^\sigma(z)|z|_E^{s}\psi(\eta a_z)\mathbb{I}_{\varpi_F^{-M}\intring{F}}(a_z)\mathbb{I}_{\eta+\varpi_F^M\intring{F}}(b_z)\\
&\hspace{41pt}\times\phi_{m,n}^\nu\left(\varpi^{-c_\nu},\eta^{-1}\varpi_F^{-c_\nu}zz^\sigma,\eta^{-1}\varpi_F^{-c_\nu}z^\sigma\right)d_E^\times z
\end{align*}
The first equality follow from a change of variable from $y$ to $yb^{-1}_z$ and (\ref{E:2.2}). The last equality follow from (\ref{E:2.2.6}) and the well-known formula that for $k \in \Z$, we have
$$\int_{\varpi_F^{k}\mathcal{O}_F^{\times}}\nu(y)|y|_F^{s-1}\psi(y)dy= \begin{cases}  0 & \mbox{ if $k\neq -c_{\nu}$},\\ \epsilon(1-s,\nu\vert_{F^{\times}}^{-1},\psi) & \mbox{ if $k =- c_{\nu}$}.\end{cases}$$
Thus, we conclude that
\[\lim_{m\rightarrow\infty}\lim_{n\rightarrow\infty}I_{m,n}(s,\mu,\nu;f)
=\nu(-1)|\eta|_F^{-1}\frac{\zeta_F(1)}{\zeta_E(1)}\int_E\Phi_{\eta, M}(z)\mu\nu^\sigma(z)|z|_E^{s}\,d_E^\times z.\]
Similarly, we also have
\[\lim_{m\rightarrow\infty}\lim_{n\rightarrow\infty}I_{m,n}(1-s,\nu^{-1},\mu^{-1};f)=\mu(-1)|\eta|_F^{-1}\frac{\zeta_F(1)}{\zeta_E(1)}\int_E\Phi_{\eta, M}(z)\nu\mu^\sigma(z)^{-1}|z|_E^{1-s}\,d_E^\times z.\]
By (\ref{E:theta}) and the condition that $\eta \varpi_F^M \in \mathcal{O}_F$, one can show that the Fourier transform of $\Phi_{\eta, M}$ with respect to $\psi_{\xi}$ satisfying
\[\widehat{\Phi}_{\eta, M}(z)=\Phi_{\eta, M}(z^\sigma).\]
Therefore, by the functional equation for the character $\mu\nu^\sigma$, we have
\[\lim_{m\rightarrow\infty}\lim_{n\rightarrow\infty}I_{m,n}(1-s,\nu^{-1},\mu^{-1};f)=\mu\nu(-1)\gamma(s,\mu\nu^\sigma,\psi_\xi)\lim_{m\rightarrow\infty}\lim_{n\rightarrow\infty}I_{m,n}(s,\mu,\nu;f).\]
Therefore, what remaining to prove is that the integral
\begin{align*}
\int_{E^\times}\Phi_{\eta, M}(z)\mu\nu^\sigma(z)|z|_E^{s}\,d_E^\times z
\end{align*}
is non zero. By direct computation, we have
\begin{align*}
&\int_{E^\times}\Phi_{\eta, M}(z)\mu\nu^{\sigma}(z)|z|_E^{s}\,d_E^\times z\\
&=\zeta_E(1)\int_{\eta+\varpi_F^M\intring{F}}\int_{\varpi_F^{-M}\intring{F}}\psi(\eta a)\mu\nu^{\sigma}(a+b\theta)|a+b\theta|_E^{s-1}\,d_Fa\,d_Fb\\
&=\zeta_E(1)\mu\nu^{\sigma}(\eta)|\eta|_F^{2s-1}\int_{\eta+\varpi_F^M\intring{F}}\int_{\eta^{-1}\varpi_F^{-M}\intring{F}}
\psi(\eta ab)\mu\nu^{\sigma}(a+\theta)|a+\theta|_E^{s-1}\,d_Fa\,d_Fb\\
&=\zeta_E(1)|\varpi_F|_F^{M}\mu\nu^{\sigma}(\eta)|\eta|_F^{2s-1}\int_{\eta^{-1}\varpi_F^{-M}\intring{F}}
\psi(\eta^{2}a)\mu\nu^{\sigma}(a+\theta)|a+\theta|_E^{s-1}\,d_Fa\\
&=\zeta_E(1)|\varpi_F|_F^{M}\mu\nu^{\sigma}(\eta)|\eta|_F^{2s-1}\\
&\hspace{11pt}\times\left[\int_{\intring{F}}\psi(\eta^{2}a)\mu\nu^{\sigma}|\cdot|_E^{s-1}(a+\theta)\,d_Fa+\sum_{r=-M-\ord_F(\eta)}^{-1}|\varpi_F|_E^{r(s-1)}\int_{\varpi_F^{r}\intring{F}^\times}
\psi(\eta^{2}a)\mu\nu^{\sigma}(a+\theta)\,d_Fa\right].
\end{align*}
It suffices to show that 
$$\int_{\varpi_F^{r}\intring{F}^\times}
\psi(\eta^{2}a)\mu\nu^{\sigma}(a+\theta)\,d_F \neq 0$$
for some $-M-{\rm ord}_F(\eta) \leq r \leq -1$. By (\ref{E:2.2}), $-M-{\rm ord}_F(\eta) \leq (-2\ord_F(\eta)-c_\omega) \leq -1$ and 
\begin{align*}
\int_{\eta^{-2}\varpi_F^{-c_\omega}\intring{F}^\times}\psi(\eta^2a)\mu\nu^{\sigma}(a+\theta)d_Fa &=\int_{\eta^{-2}\varpi_F^{-c_\omega}\intring{F}^\times}\psi(\eta^{2}a)\omega(a)\,d_Fa\\
&=\omega(\eta)^{-2}|\eta|_F^{-2}
\begin{cases}
\varepsilon(0,\omega^{-1}|_{F^\times},\psi)&\text{ if }c_\omega>0,\\
\zeta_F(1)^{-1}&\text{ if }c_\omega=0.
\end{cases}
\end{align*}
This completes the proof.
\end{proof}

\section{Asai local factors for the archimedean case}\label{S3}
In this section, we develop the theory concerning the analytic properties of the zeta integrals defined by \eqref{zeta integral} for the case $E=\C$ and $F=\R$, and prove analogy results to that of \cite[Theorem 17.2]{Jac72} and 
\cite[Theorem 2.1]{Jac09}. Instead of concerning the Harish-Chandra modules, we consider the Harish-Chandra representations.

\subsection{Harish-Chandra representation}
Let $(\pi,V)$ be an infinite-dimensional irreducible Harish-Chandra representation of 
${\rm GL}_2(\C)$ defined in \cite{Cas89}. Then $V$ is a Frechet space and the representation on 
$V$ is smooth. Moreover, $(\pi,V)$ is of moderate growth, a notion that we now recall. 
We follow the notation in \cite{Jac09}. For $g\in{\rm GL}_2(\C)$, we set
\[
\|g\|_H={\rm Tr}\left(g\,^t\b{g}\right)+{\rm Tr}\left(g^{-1}\,^t\b{g}^{-1}\right).
\]
Then for every continuous semi-norm $\|\cdot\|$ on $V$, there is a positive integer $M$ and another 
semi-norm $\|\cdot\|'$ such that for every $v\in V$ and $g\in{\rm GL}_2(\C)$, 
\[
\|\pi(g)v\|\leq \|g\|_H^M \|v\|'.
\]

Let $V_0$ denote the subspace of ${\rm U}_2(\R)$-finite elements in $V$. Then $V_0$ is a 
Harish-Chandra module. Its a result of Casselman and Wallach that $V$ is determined, up to 
topological equivalent by the equivalent class of $V_0$. In other words, $V$ is the canonical
Casselman-Wallach completion of $V_0$. In \cite{Jac09}, Jacquet called $V$ a 
Casselman-Wallach representation.

Classification of the Harish-Chandra module $V_0$ can be founded in \cite[Theorem 6.2]{JL70}. 
Combining this result with \cite[Propositions 4.1, 4.4]{Cas89}, we see that
$\pi={\rm Ind}_{\B(\C)}^{{\rm GL}_2(\C)}(\mu,\nu)$ is an induced representation. 
Therefore $V$ can be realized as 
$\cB(\mu,\nu)$ and the representation is now by the right translation. 

We describe the topology on $\cB(\mu,\nu)$.  
By the Iwasawa decomposition for ${\rm GL}_2(\C)$, we can identify the space 
of functions in $\cB(\mu,\nu)$ to the space of their restriction to ${\rm U}_2(\R)$, which we 
denote by $\cB(\mu,\nu)|_{{\rm U}_2(\R)}$. The topology of $\cB(\mu,\nu)|_{{\rm U}_2(\R)}$ is the 
one given by the semi-norms
\[
{\rm Sup}_{k\in {\rm U}_2(\R)}\left|\rho(X)f(k)\right|,
\]
where $X$ ranges over the universal enveloping algebra of the complexified Lie algebra of 
${\rm U}_2(\R)$. 

We let $\left(\pi^\vee, V^\vee\right)$ be a representation of ${\rm GL}_2(\C)$ which isomorphic to 
${\rm Ind}_{\B(\C)}^{{\rm GL}_2(\C)}(\mu^{-1},\nu^{-1})$. 
Then the Harish-Chandra modules $V_0$ and $V_0^\vee$ are dual to each other. 
By our definition, we have $\left(\pi^\vee\right)^\vee=\pi$.

Let $\psi'$ be a non-trivial additive character of $\C$. By \cite[Theorem 15.4.1]{Wal92}, 
there is a non-zero continuous functional $\lambda_{\psi'}$ on $\cB(\mu,\nu)$, and within a 
scalar factor, a unique one, such that for every $f\in\cB(\mu,\nu)$ and $u\in\U(\C)$, we have
\[
\lambda_{\psi'}(\pi(u)f)=\psi'(z)\lambda_{\psi'}(f).
\] 
Here $u=\pMX 1z01$ with $z\in\C$. 

For each $f\in\cB(\mu,\nu)$, we set
\begin{equation}\label{E:Whittaker function from induce representation}
W_{\psi',f}(g)=\lambda_{\psi'}(\rho(g)f),\quad g\in{\rm GL}_2(\C).
\end{equation}
We denote by $\cW(\pi,\psi')$ the space spanned by the functions $W_{\psi',f}$. By 
\cite[Section 6]{JL70}, its also spanned by the functions
\begin{equation}\label{E:integral representation of Whittaker function}
W^{\psi'}_{\Psi}(g)
=
\mu({\rm det}(g))|{\rm det}(g)|_\C^{\frac{1}{2}}
\int_{\C^{\x}}\omega_{\psi'}(g)\Psi(z,z^{-1})\mu\nu^{-1}(z)d^{\x}z.
\end{equation}
Here $\Psi$ is an element in the space $\cS(\C^2)$, and $(\omega_{\psi'},\cS(\C^2))$ is the 
Weil representation of ${\rm GL}_2(\C)$ defined in the first section of \cite{JL70}. 

Note that the space $\cW\left(\pi^\vee,\psi'\right)$ is spanned by the functions 
$\left(W\ot\omega^{-1}\right)(g):=W(g)\omega({\rm det}(g))^{-1}$ with $W\in\cW(\pi,\psi')$. 

\subsection{Main results for archimedean case}\label{SS:main result}
We state our main results of this section. Let 
\begin{equation}\label{E:equation of mu and nu}
\mu(z)=|z|_\C^{\lambda_1-\frac{n_1}{2}}z^{n_1}
\quad\text{and}\quad
\nu(z)=|z|_\C^{\lambda_2-\frac{n_2}{2}}z^{n_2},
\end{equation}
for some $\lambda_1, \lambda_2\in\C$ and $n_1, n_2\in\Z$.
Let $\psi(x)=e^{2\pi i ax}$ for some $a\in\R^{\x}$. Note that $\xi=c\sqrt{-1}$ for some non-zero
real number $c$.



Define a subspace of $\cS(\R^2)$  
\[
\cS\left(\R^2,\psi\right)
=
\stt{p(x,y)\,e^{-\pi |a|(x^2+y^2)}\mid p(x,y)\in\C[x,y]}.
\]
We note that $\cS(\R^2,\psi)$ is invariant under the Fourier transform $\Phi\mapsto\wh{\Phi}$ defined by \eqref{fourier transform}.

Following \cite{Jac09}, we let $\cL\left({\rm As}\,\pi\right)$ be the space of meromorphic 
functions $f(s)$ which are holomorphic multiples of $L\left(s,{\rm As}\,\pi\right)$ and furthermore
satisfy the following condition. Let $P(s)\in\C[s]$ be a polynomial such that 
$P(s)L\left(s,{\rm As}\,\pi\right)$ is holomorphic in the strip $a\leq{\rm Re}(s)\leq b$. Then
$P(s)f(s)$ is bounded in the same strip. 

We put 
\begin{align*}
L_{\rm Gal}\left(s, {\rm As}\,\pi\right)
&=
L\left(s,\mu|_{\R^{\x}}\right)
L\left(s,\nu|_{\R^{\x}}\right)
L\left(s,\mu\nu^\sigma\right),\\
L_{\rm Gal}\left(s, {\rm As}\,\pi^{\vee}\right)
&=
L\left(s,\mu^{-1}|_{\R^{\x}}\right)
L\left(s,\nu^{-1}|_{\R^{\x}}\right)
L\left(s,\mu^{-1}(\nu^{-1})^\sigma\right),\\
\varepsilon_{\rm Gal}\left(s,{\rm As}\,\pi,\psi\right)
&=
\lambda_{\C/\R}(\psi)
\varepsilon\left(s,\mu|_{\R^{\x}},\psi\right)
\varepsilon\left(s,\nu|_{\R^{\x}},\psi\right)
\varepsilon\left(s,\mu\nu^\sigma,\psi_\C\right).
\end{align*}
Notice that the Langlands constant is given by $\lambda_{\C/\R}(\psi)={\rm sgn}(a)\sqrt{-1}$. 

\begin{thm}\label{T:main theorem}
Let $W\in\cW(\pi,\psi_\xi)$ and $\Phi\in\cS(\R^2)$.
\begin{itemize}
\item[(1)]
The zeta integral $Z(s,W,\Phi)$ converges absolutely when
$Re(s)>2\,max\stt{-Re(\lambda_1),-Re(\lambda_2)}$ and has a meromorphic continuation to the 
whole complex plane. In fact, $Z(s,W,\Phi)$ defines an element in the space 
$\cL\left({\rm As}\,\pi\right)$. Moreover, if $W\in\cW(\pi,\psi_\xi)_0$ and $\Phi\in\cS(\R^2,\psi)$,
then $Z(s,W,\Phi)$ is of the form 
\[
|a|^{-2s}|\xi|_\C^{-s/2}\,P(s)\,L_{\rm Gal}\left(s,{\rm As}\,\pi\right),
\]
for some $P(s)\in\C[s]$.
\item[(2)]
The functional equation 
\[
\frac{Z\left(1-s,W\ot\omega^{-1},\wh{\Phi}\right)}
{L_{\rm Gal}\left(1-s,{\rm As}\,\pi^{\vee}\right)}
=
\varepsilon_{{\rm RS}}\left(s,{\rm As}\,\pi,\psi,\xi\right)
\frac{Z\left(s,W,\Phi\right)}
{L_{\rm Gal}\left(s,{\rm As}\,\pi\right)},
\]
holds in the sense of analytic continuation. Moreover, we have the following relation
\begin{equation}\label{E:relation between epsilon factors}
\omega^{-1}(\xi)|\xi|_\C^{-s+\frac{1}{2}}\lambda_{\C/\R}(\psi)
\varepsilon_{{\rm RS}}\left(s,{\rm As}\,\pi,\psi,\xi\right)
=
\varepsilon_{\rm Gal}\left(s,{\rm As}\,\pi,\psi\right).
\end{equation}
\item[(3)]
There exist $W\in\cW(\pi,\psi_\xi)_0$ and $\Phi\in\cS(\R^2,\psi)$ such that
\[
Z\left(s,W,\Phi\right)
=
|a|^{-2s}|\xi|_\C^{-s/2}\,L_{\rm Gal}\left(s,{\rm As}\,\pi\right).
\]
\end{itemize}
\end{thm}

\begin{rmk}\label{R:remark after main theorem}
Notice that if $L(s)$ is a meromorphic function on $\C$ which satisfies $(1)$ of 
\thmref{T:main theorem} and if there exist $W\in\cW(\pi,\psi_\xi)$ and $\Phi\in\cS(\R^2)$ so that
$Z(s,W,\Phi)=q(s)\,L\left(s\right)$ for some holomorphic function $q(s)$,
then $L(s)=h(s)\,L\left(s,{\rm As}\,\pi\right)$ for some holomorphic function $h(s)$ without zeros.
\end{rmk}

Following lemma is the core of our proof for \thmref{T:main theorem}, whose proof will occupy in
\secref{S:proof of main lemma and the relation between epsilon factors}.

\begin{lm}\label{L:main lemma}
There exist $W\in\cW(\pi,\psi_\xi)_0$ and $\Phi\in\cS(\R^2,\psi)$ such that 
\begin{align*}
Z(s,W,\Phi)
&=
c\,|a|^{-2s}|\xi|_\C^{-s/2}\,L_{\rm Gal}\left(s,{\rm As}\,\pi\right),\\
Z\left(1-s,W\ot\omega^{-1},\wh{\Phi}\right)
&=
c^{\vee}\,|a|^{-2+2s}|\xi|_\C^{-1/2+s/2}\,L_{\rm Gal}\left(1-s,{\rm As}\,\pi^{\vee}\right),
\end{align*}
for some non-zero constants $c$ and $c^{\vee}$.
\end{lm}

Suppose \thmref{T:main theorem} holds for some $\psi$ and $\xi$, we see what happen if we replace
$\psi$ by $\psi^b(x):=\psi(bx)$ and $\xi$ by $\xi'=c\xi$ for some $b, c\in\R^{\x}$. 
Consider the following $\C$-linear isomorphisms
$\cW(\pi,\psi_\xi)\stackrel{\sim}{\longto}\cW(\pi,\psi^b_{\xi'}),\, W\to W^{bc}$
and $\cS(\R^2)\stackrel{\sim}{\longto}\cS(\R^2),\,\Phi\to\Phi^b$, where
\[
W^{bc}\left(g\right)
=
W\left(
\begin{pmatrix}
bc&0\\0&1
\end{pmatrix}g
\right)
\quad\text{and}\quad
\Phi^b(x,y)
=
\Phi\left(|b|^{1/2}x,|b|^{1/2}y\right).
\]
Notice that $\Phi\mapsto\Phi^b$ induces an isomorphism between the spaces $\cS(\R^2,\psi)$ and
$\cS(\R^2,\psi^b)$.

One check that 
\[
\wh{\Phi^b}(x,y)=\wh{\Phi}\left(|b|^{1/2}x,|b|^{1/2}y\right).
\]
where the Fourier transform on the LHS is with respect to $\psi^b$, while the Fourier 
transform on the RHS is with respect to $\psi$. Moreover, we have
\begin{align*}
Z\left(s,W^{bc},\Phi^b\right)
&=
\omega\left(|b|^{1/2}\right)^{-1}|b^2c|^{-s}
Z(s,W,\Phi),\\
Z\left(s,W^{bc}\ot\omega^{-1},\wh{\Phi^b}\right)
&=
\omega\left(|b|^{1/2}\right)^{-1}|b^2c|^{-s}
Z\left(s,W\ot\omega^{-1},\wh{\Phi}\right).
\end{align*}
From these we obtain the following relation
\begin{equation}\label{E:dependence of epsilon factor on psi and xi}
\varepsilon_{{\rm RS}}\left(s,{\rm As}\,\pi,\psi^b,c\xi\right)
=
\omega\left(b^2c\right)|b^2c|^{2s-1}
\varepsilon_{{\rm RS}}\left(s,{\rm As}\,\pi,\psi,\xi\right).
\end{equation}
Since 
\[
\varepsilon\left(s,{\rm As}\,\pi,\psi^b\right)
=
{\rm sgn}(b)\omega(b^2)|b|^{4s-2}
\varepsilon\left(s,{\rm As}\,\pi,\psi\right),
\]
this shows that if \thmref{T:main theorem} holds for some $\psi$ and $\xi$, then it holds for all 
non-trivial additive character $\psi$ and all non-zero element $\xi$ with ${\rm tr}_{\C/\R}(\xi)=0$. 

\subsection{Proof of \thmref{T:main theorem}}
Taking \lmref{L:main lemma} for granted in this section, we prove \thmref{T:main theorem} except 
the relation \eqref{E:relation between epsilon factors}, whose proof will be given in 
\secref{S:proof of main lemma and the relation between epsilon factors}.
By the remark at the end of \subsecref{SS:main result}, we may assume $\psi(x)=e^{2\pi ix}$ and 
$\xi=\sqrt{-1}$. Also, it suffices to prove first two assertions, as the last is already contained 
in \lmref{L:main lemma}.  Let $\mu$ and $\nu$ be as in the equation 
\eqref{E:equation of mu and nu}. Since the character $\psi_\xi$ has fixed, we will suppress the notation $\psi_\xi$ in 
$W^{\psi_\xi}_{\Psi}$ and $W_{\psi_\xi,f}$ from now on. 
Put $n_0=|n_1-n_2|$.

\subsubsection{Convergence}
We show that the zeta integral $Z(s,W,\Phi)$ converges absolutely when 
\[
{\rm Re}(s)>2\,{\rm max}\stt{-{\rm Re}(\lambda_1),-{\rm Re}(\lambda_2)}.
\]
It will be convenience to use the following notation
\[
A\ll_{x,y,z,...}B
\] 
to indicate there is a constant $C>0$ depending at most upon $x,y,z,...$ so that $|A|\leq C|B|$. 

Let $r={\rm Re}(s)$, $r_1={\rm Re}(\lambda_1)$ and $r_2={\rm Re}(\lambda_2)$. Formally we have
\begin{align}\label{E:expansion of zeta integral}
\begin{split}
Z(s,W,\Phi)
&=
\int_{{\rm SO}_2(\R)}\int_{\R^{\x}}
W\left(\pMX y001 k\right)|y|^{s-1}
\int_{\R^{\x}_+}\Phi((0,t)k)\omega(t)t^{2s}d^{\x}t\,d^{\x}y\,dk,\\
&=
\int_{{\rm SO}_2(\R)}
\left(
\int_{\R^{\x}}
W\left(\pMX y001 k\right)|y|^{s-1}d^{\x}y
\right)
\left(
\int_{\R^{\x}_+}\Phi((0,t)k)\omega(t)t^{2s}d^{\x}t
\right)dk,
\end{split}
\end{align} 
by the Iwasawa decomposition.

We may assume $W=W_\Psi$ for some $\Psi\in\cS(\C^2)$. Then by equation
\eqref{E:integral representation of Whittaker function}, we have
\begin{equation}\label{E:expending Tate integral for C x R}
\int_{\R^{\x}}
W\left(
\pMX y001 k
\right)|y|^{s-1}d^{\x}y
=
\int_{\R^{\x}}\mu(y)|y|^s\int_{\C^{\x}}
\omega_{\psi_\xi}(k)\Psi\left(yz,z^{-1}\right)\mu\nu^{-1}(z)d^{\x}zd^{\x}y.
\end{equation}
Since ${\rm SO}_2(\R)$ is compact, by the continuity of the Weil representation and 
\cite[Lemme 5]{Wei64}, there exists $\Psi_0\in\cS(\C^2)$ such that
\[
\omega_{\psi_\xi}(k)\Psi\ll_{\Psi}\Psi_0,
\]
for all $k\in {\rm SO}_2(\R)$. On the other hand, for every positive integer $N$, we have
\[
\Psi_0\left(yz,z^{-1}\right)
\ll_N
\left(1+|yz|_\C+|z|_\C^{-1}\right)^{-N},
\]
by \cite[Lemma 3.1 (i)]{Jac09}. It follows that
\begin{align}\label{E:estimate for the first term}
\begin{split}
\int_{\R^{\x}}
W\left(
\pMX y001 k
\right)|y|^{s-1}d^{\x}y
&\ll_{\Psi, N}
\int_{\R^{\x}}\int_{\C^{\x}}
\frac{|y|^{r+2r_1}|z|_\C^{r_1-r_2}}{\left(1+|yz|_\C+|z|_\C^{-1}\right)^N}\,d^{\x}zd^{\x}y\\
&\ll_{\Psi,N}
\int_{\R_+^{\x}}\int_{\R_+^{\x}}
\frac{y^{r+2r_1} t^{2r_1-2r_2}}{\left(1+y^2t^2+t^{-2}\right)^N}\,d^{\x}yd^{\x}t\\
&\ll_{\Psi,N}
\int_0^\infty\int_0^\infty
\frac{y^{r+2r_1-1} t^{r+2r_2-1}}{\left(1+y^2+t^2\right)^N}\,dydt\\
&\ll_{\Psi,N}
\int_0^\infty\int_0^\infty
\frac{y^{r+2r_1-1} t^{r+2r_2-1}}{(1+y^2)^{N/2}(1+t^2)^{N/2}}\,dydt.
\end{split}
\end{align}
In the third line, we change the variables $y\mapsto t^{-1}y$, $t^{-1}\mapsto t$, while in the 
last line, we use the inequality
\[
(1+y^2+t^2)^2\geq (1+y^2)(1+t^2).
\]

Applying \cite[Lemma 3.1 (i)]{Jac09} again, we see that
\begin{equation}\label{E:estimate for the second term}
\int_{\R^{\x}_+}
\Phi((0,t)k)\omega(t)t^{2s}d^{\x}t
\ll_{\Phi,N}
\int_0^\infty
\frac{t^{2r+2r_1+2r_2-1}}{\left(1+t^2\right)^N}\,dt.
\end{equation}
By equations \eqref{E:estimate for the first term} and \eqref{E:estimate for the second term}, 
we find that $Z(s,W,\Phi)$ converges absolutely when
\[
0<r+2r_1<N,
\quad
0<r+2r_2<N,
\quad
0<2r+2r_1+2r_2<2N.
\]
Since $N>0$ is arbitrary, our assertion follows at once.
Of course the constants appeared in our estimate in \eqref{E:estimate for the first term} and
\eqref{E:estimate for the second term} depend on the choice of Haar measures, but different choice
certainly do not affect our claim.

\subsubsection{Reduction to the $K$-finite datum}
Using the results in \cite[Section 4]{Jac09}, we reduce the proof of \thmref{T:main theorem}
to the $K$-finite datum. More precisely, we show that if the assertions of \thmref{T:main theorem}  
hold for $W\in\cW(\pi,\psi_\xi)_0$ and $\Phi\in\cS(\R^2,\psi)$, then they also hold for 
$W\in\cW(\pi,\psi_\xi)$ and $\Phi\in\cS(\R^2)$.

Observed that if \thmref{T:main theorem} holds for $\pi$, then it also holds for 
$\pi\ot|\mbox{ }|_\C^u$ for every $u\in\C$. Here $\pi\ot|\mbox{ }|_\C^r$ denote the representation of 
${\rm GL}_2(\C)$ on the same space $V$ with the action 
\[
\left(\pi\ot|\mbox{ }|_\C^u\right)(g)v=|{\rm det(g)}|_\C^u\,\pi(g)v,
\quad 
g\in{\rm GL}_2(\C),\,v\in V.
\]
From this observation we may assume $Z(s,W,\Phi)$ converges absolutely when ${\rm Re}(s)\geq 0$.

Zeta integral can be written as
\[
Z(s,W,\Phi)
=
\int_{\R^{\x}}
|y|^s
\int_{\U(\R)\backslash{\rm SL}_2(\R)}
W\left(
\pMX y001 g
\right)
\Phi((0,1)g)dgd^{\x}y.
\]
Since $Z(s,W,\Phi)$ converges absolutely when ${\rm Re}(s)\geq 0$, following integrals are also
absolute convergent
\begin{align}\label{E:sigma psi pair}
\begin{split}
\varphi(y)
&=
\omega(\xi)\int_{\U(\R)\backslash{\rm SL}_2(\R)}
W\left(\pMX y001 g\right)\Phi((0,1)g)dg,\\
\t{\varphi}(y)
&=
\int_{\U(\R)\backslash{\rm SL}_2(\R)}
W\left(\pMX y001 g\right)\omega^{-1}(y)\wh{\Phi}((0,1)g)dg.
\end{split}
\end{align}
Recall that $\xi=\sqrt{-1}$.

Suppose \thmref{T:main theorem} holds for $W\in\cW(\pi,\psi_\xi)_0$ and $\Phi\in\cS(\R^2,\psi)$.
In this case, $Z(s,W,\Phi)$ is of the form 
\[
P(s)\,L_{{\rm Gal}}\left(s,{\rm As}\,\pi\right),
\]
for some $P(s)\in\C[s]$. We note that $Z(s,W,\Phi)$ is an element of the space $\cL({\rm As}\,\pi)$. 
Indeed, this follows immediately from the asymptotic behaviour of the gamma function
\[
\Gamma(x+iy)\sim (2\pi)^{\frac{1}{2}}|y|^{x-\frac{1}{2}}e^{-\frac{\pi}{2}|y|} 
\]
for $x$ fixed and $|y|\to\infty$, together with the Phragmen-Lindelof principle.

Let $\phi$ be the representation of $W_\R$ defined by 
\[
\phi=\mu|_{\R^{\x}}
\op
\nu|_{\R^{\x}}
\op{\rm Ind}_{W_\C}^{W_\R}\left(\mu\nu^\sigma\right).
\] 
In \cite[Section 4]{Jac09}, Jacquet defined a notion of so called $(\phi,\psi)$ pairs,
and he used this to reduce the proofs to the $K$-finite datum.

Let $\varphi, \t{\varphi}$ be the functions given by equation \eqref{E:sigma psi pair} with 
$W=W_f\in\cW(\pi,\psi_\xi)_0$ and $\Phi\in\cS(\R^2,\psi)$. Here $f\in\cB(\mu,\nu)_0$.
By our assumption and \cite[Proposition 4.2]{Jac09}, we have
\[
\int_{\R^{\x}}
\t{\varphi}(y)\t{h}(y)d^{\x}y
=
\int_{\R^{\x}}
\varphi(y)h(y)d^{\x}y,
\]
for every $(\phi,\psi)$ pair $(h,\t{h})$. By equation \eqref{E:sigma psi pair}, 
this is equal to the following equality
\begin{align}\label{E:main identity for the reduction}
\begin{split}
&\int_{\U(\R)\backslash\GL_2(\R)}
W_f\left(\pMX y001 g\right)\omega^{-1}({\rm det}(g))\wh{\Phi}((0,1)g)\t{h}({\rm det}(g))dg\\
&=
\omega(\xi)\int_{\U(\R)\backslash\GL_2(\R)}
W_f\left(\pMX y001 g\right)\Phi((0,1)g)h({\rm det}(g))dg,
\end{split}
\end{align}
for every $(\phi,\psi)$ pair $(h,\t{h})$, and for every $f\in\cB(\mu,\nu)_0$, 
$\Phi\in\cS(\R^2,\psi)$.

Notice that both sides of \eqref{E:main identity for the reduction} are absolute convergent
for every $f\in\cB(\mu,\nu)$ and $\Phi\in\cS(\R^2)$. Indeed, for every integer $N$, 
we have
\[
h(y)\ll_{h,N}|y|^N
\quad\text{and}\quad
\t{h}(y)\ll_{\t{h},N}|y|^N,
\]
for all $y\in\R^{\x}$. Combining these with \cite[Lemma 3.5, Proposition 3.3]{Jac09}, we see
that, when $h, \t{h}$ are fixed, both sides in \eqref{E:main identity for the reduction} are 
continuous functions of $(f,\Phi)$. As the equality holds for all $f\in\cB(\mu,\nu)_0$ and
$\Phi\in\cS(\R^2,\psi)$, we find that it also holds for all $f\in\cB(\mu,\nu)$ and 
$\Phi\in\cS(\R^2)$ by the continuity and an argument of density. Since $h, \t{h}$ are arbitrary
$(\phi,\psi)$ pair, we can now apply \cite[Proposition 4.1]{Jac09}, and then
\thmref{T:main theorem} follows. 

Because of this observation, for the rests of this section, we are devoted to prove \thmref{T:main theorem} for the
$K$-finite datum. 
\subsubsection{Notation}\label{SS:notation}
We introduce some notations which will be used in our proof of \thmref{T:main theorem} and
\lmref{L:main lemma} for the $K$-finite datum.
Let $\mathfrak{g}={\rm Lie}(\GL_2(\R))\ot_\R\C$. Then $\cS(\R^2,\psi)$ is a 
$(\mathfrak{g}, {\rm O}_2(\R))$-module with the actions
\[
\rho(X)\Phi(x,y)
=
\frac{d}{dt}\Phi\left((x,y)e^{tX}\right)|_{t=0},
\quad
\rho(k)\Phi(x,y)
=
\Phi((x,y)k),
\]
for $X\in\mathfrak{g}$ and $k\in {\rm O}_2(\R)$. 

For a pair $\ul{a}=(a_1,a_2)$ of non-negative integers, we set
\begin{equation}\label{E:basis for S(R^2,psi) and S(C,psi_xi)}
\Phi_{\ul{a}}(x,y)
=
(x+iy)^{a_1}(x-iy)^{a_2}\,e^{-\pi(x^2+y^2)}
\quad\text{and}\quad
\Psi_{\ul{a}}(z)=z^{a_1}\b{z}^{a_2}e^{-2\pi z\b{z}},
\end{equation}
for $x,y\in\R$ and $z\in\C$.

These functions have the following properties
\begin{equation}\label{E:translation properties of test fucntions}
\rho(k(\theta))\Phi_{\ul{a}}
=
e^{i(a_1-a_2)\theta}\Phi_{\ul{a}}
\quad\text{and}\quad
\Psi_{\ul{a}}(ze^{i\theta})
=
e^{i(a_1-a_2)\theta}\Psi_{\ul{a}}(z).
\end{equation}
Here
\begin{equation}\label{E:elements of SO(2)}
k(\theta)
:=
\begin{pmatrix}
{\rm cos}\theta&{\rm sin}\theta\\
-{\rm sin}\theta&{\rm cos}\theta
\end{pmatrix}\in {\rm SO}_2(\R).
\end{equation}

There is a similar space attached to $\C^2$ and $\psi_\xi$ given by
\[
\cS(\C^2,\psi_\xi)
=
\stt{p(z,\b{z},w,\b{w})\,e^{-2\pi(z\b{z}+w\b{w})}\mid p(x_1,x_2,x_3,x_4)\in\C[x_1,x_2,x_3,x_4]}.
\]
Let $\ul{b}=(b_1,b_2)$ be another pair of non-negative integers and define
\begin{equation}\label{E:basis for S(C^2,psi_xi)}
\Psi_{\ul{a}}\ot\Psi_{\ul{b}}(z,w)
:=
\Psi_{\ul{a}}(z)\Psi_{\ul{b}}(w),\quad z,w\in\C.
\end{equation}
Then we have
\begin{equation}\label{E:decomposition of S(R^2,psi) and S(R^2,psi)}
\cS(\R^2,\psi)
=
\bigoplus_{\ul{c}}\C\,\Phi_{\ul{c}}
\quad\text{and}\quad
\cS(\C^2,\psi_\xi)
=
\bigoplus_{\ul{a},\ul{b}}\C\,\Psi_{\ul{a}}\ot\Psi_{\ul{b}},
\end{equation}
where $\ul{a}$, $\ul{b}$ and $\ul{c}$ run through all pairs of non-negative integers. 

\subsubsection{Haar measures}\label{SS:Haar measures}
Although the choice of Haar measures on $\GL_2(\R)$ and $\U(\R)$ are not really important,
we fix the choice on various groups in this and the next section for the sake of convenience. 

The Haar measure $dx$ on $\R$ is the usual Lebesgue measure and the Haar measure $d^{\x}x$ on 
$\R^{\x}$ is given by $d^{\x}x=|x|^{-1}dx$. Since $\U(\R)\cong\R$, the measure on $\U(\R)$ is
defined. Haar measure $dz$ on $\C$ is twice of the usual Lebesgue measure and the Haar measure 
$d^{\x}z$ on $\C^{\x}$ is $d^{\x}z=|z|_\C^{-1}dz$. 

The Haar measure $dg$ on $\GL_2(\R)$ is given by
\[
dg=\frac{dt}{t}\frac{dxdy}{|y|^2}dk
\]
for 
$
g
= 
\pMX t00t
\begin{pmatrix}
1&x\\0&1
\end{pmatrix}
\begin{pmatrix}
y&0\\0&1
\end{pmatrix}k
$ 
with $x\in\R, y\in\R^{\x}, t\in\R_+^{\x}, k\in {\rm SO}_2(\R)$, where $dx, dy, dt$ are the usual Lebesgue
measures and $dk$ is the Haar measure on ${\rm SO}_2(\R)$ such that ${\rm Vol}({\rm SO}_2(\R), dk)=1$. 

Finally, the measure on the quotient space $\U(\R)\backslash \GL_2(\R)$ is the unique 
quotient measure induced from the measure $dg$ on $\GL_2(\R)$ and the measure $dx$ on $\U(\R)$. 

We stress that when we compute the Fourier transform $\wh{\Phi}$ of an element $\Phi\in\cS(\R^2)$,
the Haar measures $du, dv$ should be self-dual with respect to $\psi$. Since we have assume that
$\psi(x)=e^{2\pi i x}$, the measures $du, dv$ are just the usual Lebesgue measure on $\R$.

\subsubsection{Meromorphic continuation}
We show that the zeta integral $Z(s,W,\Phi)$ admits a meromorphic continuation to whole 
complex plane.

By \cite[Section 6]{JL70}, the space $\cW(\pi,\psi_\xi)_0$ is spanned by the functions $W_\Psi$ 
(Cf. equation \eqref{E:integral representation of Whittaker function}) with 
$\Psi\in\cS(\C^2,\psi_\xi)$. Let $\ul{a}=(a_1,a_2)$ and $\ul{b}=(b_1,b_2)$ be two pairs of 
non-negative integers. We put
\begin{equation}\label{E:basis for Whittaker functions}
W_{(\ul{a},\ul{b})}
=
W_{\Psi_{\ul{a}}\ot\Psi_{\ul{b}}}.
\end{equation}
Recall that $\Psi_{\ul{a}}\ot\Psi_{\ul{b}}$ is given by 
\eqref{E:basis for S(R^2,psi) and S(C,psi_xi)} and
\eqref{E:basis for S(C^2,psi_xi)}. Also, the elements $W_{(\ul{a},\ul{b})}$ span 
$\cW(\pi,\psi_\xi)_0$ by \eqref{E:decomposition of S(R^2,psi) and S(R^2,psi)}.

Let $s\in\C$, $W\in\cW(\pi,\psi_\xi)_0$ and $\chi$ be a character of $\R^{\x}$. We define an 
integral attached to $W$ and $\chi$, which is analogy to the classical Tate integral.
\begin{equation}\label{E:Tate integral for C x R}
\zeta(s,W,\chi)
=
\int_{\R^{\x}}W\left(\pMX y001\right) \chi(y)|y|^{s-1}d^{\x}y.
\end{equation}
\begin{lm}\label{L:integration formula for Tate C x R}
We have
\begin{itemize}
\item[(1)] 
If $a_1-a_2+n_1\neq b_1-b_2+n_2$, then $W_{(\ul{a},\ul{b})}=0$.
\item[(2)]
Suppose $\chi(y)=|y|^\lambda {\rm sgn}^m(y)$ for some $\lambda\in\C$ and $m\in\stt{0,1}$.
The integral $\zeta(s,W,\chi)$ converges absolutely for 
${\rm Re}(s)> 2\,{\rm max}\left\lbrace -{\rm Re}(\lambda_1),-{\rm Re}(\lambda_2)
\right\rbrace-{\rm Re}(\lambda)$ and
has meromorphic continuation to the whole complex plane. More precisely, if 
$W=W_{(\ul{a},\ul{b})}$ with $a_1-a_2+n_1=b_1-b_2+n_2$, then
\begin{align}\label{E:integration formula for Tate C x R}
\begin{split}
\zeta(s,W_{(\ul{a},\ul{b})},\chi)
&=
2^{-1}\left(1+(-1)^{n_1+m+a_1+a_2}\right)
(2\pi)^{-\left(s+\lambda+\lambda_1+\lambda_2-1+\frac{a_1+a_2+b_1+b_2}{2}\right)}\\
&\x\Gamma\left(\frac{1}{2}(s+\lambda+2\lambda_1+a_1+a_2)\right)
\Gamma\left(\frac{1}{2}(s+\lambda+2\lambda_2+b_1+b_2)\right).
\end{split}
\end{align}
In particular, the integral $\zeta(s,W_{(\ul{a},\ul{b})},\chi)$ vanishes if $n_1+m+a_1+a_2$ is odd.
\end{itemize}
\end{lm}

\begin{proof}
Recall that $W_{(\ul{a},\ul{b})}=0$ if and only if 
$W_{(\ul{a},\ul{b})}\left(\pMX \alpha001\right)=0$ for all $\alpha\in\C^{\x}$. By equation
\eqref{E:integral representation of Whittaker function} and the relation 
\eqref{E:translation properties of test fucntions}, we have
\begin{align}\label{E:Whittaker function on tours}
\begin{split}
W_{(\ul{a},\ul{b})}\left(\pMX \alpha001\right)
&=
\mu(\alpha)|\alpha|_\C^{\frac{1}{2}}
\int_{\C^{\x}}\Psi_{\ul{a}}(\alpha z)\Psi_{\ul{b}}(z^{-1})\mu\nu^{-1}(z)d^{\x}z\\
&=
2\mu(\alpha)|\alpha|_\C^{\frac{1}{2}}
\int_{\R^{\x}_+}\Psi_{\ul{a}}(\alpha t)\Psi_{\ul{b}}(t^{-1})\mu\nu^{-1}(t)
\int_0^{2\pi}e^{i(a_1-a_2+b_2-b_1+n_1-n_2)\theta}d\theta\,d^{\x}t.
\end{split}
\end{align}
In particular, $W_{(\ul{a},\ul{b})}=0$ if $a_1-a_2+n_1\neq b_1-b_2+n_2$. This proves $(1)$. 

To prove $(2)$, we may assume $W=W_{(\ul{a},\ul{b})}$ with $a_1-a_2+n_1=b_1-b_2+n_2$.
Let $y\in\R^{\x}$. By equation \eqref{E:Whittaker function on tours}, 
\[
W_{(\ul{a},\ul{b})}\left(\pMX y001\right)
=
4\pi\,\mu(y)|y|
\int_{\R^{\x}_+}\Psi_{\ul{a}}(yt)\Psi_{\ul{b}}(t^{-1})\mu\nu^{-1}(t)d^{\x}t.
\]
Formally, we have
\begin{align*}
\zeta(s,W_{(\ul{a},\ul{b})},\chi)
&=
4\pi\int_{\R^{\x}}
\mu(y)|y|_\C^s
\int_{\R^{\x}_+}\Psi_{\ul{a}}(yt)\Psi_{\ul{b}}(t^{-1})\mu\nu^{-1}(t)d^{\x}td^{\x}y\\
&=
4\pi\left(\int_{\R^{\x}}\Psi_{\ul{a}}(y)\mu\chi(y)|y|^s d^{\x}y\right)
\left(\int_{\R_+^{\x}}\Psi_{\ul{b}}(t)\nu\chi(t) t^s d^{\x}t\right),
\end{align*}
after changing variables. Since these two are the Tate integrals, our assertion $(2)$ follows 
at once.
\end{proof}

We now show that $Z(s,W,\Phi)$ has a meromorphic continuation to whole complex plane. Of course
this is a direct consequence of what we have proved in the remark above 
(Cf. equation \eqref{E:relating zeta integrals and invariant forms}). But we still give a proof 
here as we need more explicit results.

By equation \eqref{E:expansion of zeta integral}, and the right ${\rm SO}_2(\R)$-finiteness 
of $W$ and $\Phi$, we see that it suffices to prove the assertion for the integral
\begin{equation}\label{E:basis of zeta integrals}
\int_{\R^{\x}}W_{(\ul{a},\ul{b})}\left(\pMX y001 \right)|y|^{s-1}d^{\x}y
\int_{\R^{\x}_+}\Phi_{\ul{c}}((0,t))\omega(t)t^{2s}d^{\x}t,
\end{equation}
where $\ul{a}=(a_1,a_2)$, $\ul{b}=(b_1,b_2)$ and $\ul{c}=(c_1,c_2)$ are three pairs of 
non-negative integers with $a_1-a_2+n_1=b_1-b_2+n_2$. First integral is take care by
\lmref{L:integration formula for Tate C x R}, while the second integral is the Tate 
integral, whose properties are well-known. 

\subsubsection{G.C.D. of poles}\label{SS:G.C.D. of poles}
We show that there is a meromorphic function $L_{\rm RS}\left(s,{\rm As}\,\pi\right)$ such that if $W\in\cW(\pi,\psi_\xi)_0$ 
and $\Phi\in\cS(\R^2,\psi)$, then $Z(s,W,\Phi)$ is of the form $P(s)\,L_{\rm RS}\left(s,{\rm As}\,\pi\right)$ for some 
$P(s)\in\C[s]$. Furthermore, we show that there
exist $W_j\in\cW(\pi,\psi_\xi)_0$ and $\Phi_j\in\cS(\R^2,\psi)$ such that
\begin{equation}\label{E:sum of zeta integrals equal to L-factor}
\sum_j\,Z(s,W_j,\Phi_j)=L_{{\rm RS}}\left(s,{\rm As}\,\pi\right).
\end{equation}
Notice that these two conditions characterized $L_{{\rm RS}}\left(s,{\rm As}\,\pi\right)$ up to 
non-zero constants.

Recall that $n_0=|n_1-n_2|$. First we have.

\begin{lm}\label{L:zeta integral vanishes for certain test functions}
Let $\ul{c}=(c_1,c_2)$ be a pair of non-negative integers. Suppose $c_1+c_2\not\equiv n_0\pmod 2$,
then the zeta integrals $Z\left(s,W,\Phi_{\ul{c}}\right)$ vanish for all $W\in\cW(\pi,\psi_\xi)_0$.
\end{lm}

\begin{proof}
Let $n$ be a non-negative integer and $\rho_n$ denote the $n$-th symmetric power of the standard
two-dimensional representation of ${\rm GL}_2(\C)$ on $\C^2$. By \cite[Theorem 6.2]{JL70}, we have
\[
\cW(\pi,\psi_\xi)_0
=
\bigoplus_{n\geq n_0,\,n\equiv n_0\pmod 2}
\cW(\pi,\psi_\xi;\rho_n)
\]
as ${\rm SU}(2)$-modules. Here $\cW(\pi,\psi_\xi;\rho_n)$ is the $\rho_n$-isotypic component.
From this we find that if $W'$ is a non-zero element in $\cW(\pi,\psi_\xi)$ and $\ell$ is an 
integer such that $\rho(k(\theta))W'=e^{i\ell\theta}W'$, then $\ell\equiv n_0\pmod 2$. 

Now let $W\in\cW(\pi,\psi_\xi)_0$. For ${\rm Re}(s)\gg 0$, we have
\[
Z(s,W,\Phi_{\ul{c}})
=
\int_{\R^{\x}}W'\left(\pMX y001\right)|y|^{s-1}
\int_{\R^{\x}_+}\Phi_{\ul{c}}((0,t))\omega(t)t^{2s}d^{\x}t\,d^{\x}y,
\] 
by  equation \eqref{E:expansion of zeta integral}. Here 
\[
W'(g)
:=\frac{1}{2\pi}
\int_0^{2\pi}
W(gk(\theta))e^{i(c_1-c_2)\theta}d\theta.
\]
By right ${\rm SO}_2(\R)$-finiteness of $W$, we have $W'\in\cW(\pi,\psi_\xi)_0$. Moreover, it follows
immediately from the definition that $\rho(k(\theta))W'=e^{i(c_2-c_1)\theta}W'$ for every 
$k(\theta)$. This implies $W'=0$ if $c_1+c_2\not\equiv n_0\pmod 2$ as desired.
\end{proof}

By \lmref{L:integration formula for Tate C x R} and 
\lmref{L:zeta integral vanishes for certain test functions}, we see that $Z(s,W,\Phi)$ is 
a linear combination of the integrals given by the equation \eqref{E:basis of zeta integrals}
with
\begin{align}\label{E:non-zero conditions for zeta integrals}
\begin{split}
a_1-a_2+n_1&=b_1-b_2+n_2,\\
a_1+a_2&\equiv n_1\pmod 2,\\
c_1+c_2&\equiv n_0\pmod 2.
\end{split}
\end{align}
Notice that the first two imply $b_1+b_2\equiv n_2\pmod 2$. Applying equation
\eqref{E:integration formula for Tate C x R}, it turns out that $Z(s,W,\Phi)$ is a linear 
combination of the functions
\begin{equation}\label{E:basis for zeta integrals:explicit form}
2^{-s}\pi^{-2s}
\Gamma\left(\frac{1}{2}(s+2\lambda_1+a_1+a_2)\right)
\Gamma\left(\frac{1}{2}(s+2\lambda_1+b_1+b_2)\right)
\Gamma\left(s+2\lambda_1+\lambda_2+\frac{c_1+c_2}{2}\right),
\end{equation}
where $a_1, a_2$, $b_1, b_2$ and $c_1, c_2$ are non-negative integers which satisfy the 
conditions \eqref{E:non-zero conditions for zeta integrals}.

Let $\epsilon_j\in\stt{0,1}$ so that $n_j\equiv \epsilon_j\pmod 2$ for $j=0,1,2$. Define a
meromorphic function
\[
E_\pi(s)
=
\zeta_\R(s+2\lambda_1+\epsilon_1)
\zeta_\R(s+2\lambda_2+\epsilon_2)
\zeta_\C(s+\lambda_1+\lambda_2+\epsilon_0/2).
\]
Let $I(\pi)$ denote the subspace of the space of meromorphic functions spanned by the following 
set
\[
\stt{\left . \frac{Z(s,W,\Phi)}{E_\pi(s)}\mbox{ }\right\vert\mbox{ }  W\in\cW(\pi,\psi_\xi)_0,\,\Phi\in\cS(\R^2,\psi)}.
\]
By equation \eqref{E:basis for zeta integrals:explicit form}, one sees that $I(\pi)$ is a subspace 
of $\C[s]$. In fact, $I(\pi)$ is an $ideal$ of $\C[s]$. To see this, let
$\Phi'(x,y)=\frac{d}{dt}\Phi\left((x,y)e^{tI_2}\right)|_{t=0}$, where $I_2$ is the identity 
element in $\mathfrak{g}$. We have the following relation
\[
2(s+\lambda_1+\lambda_2)Z(s,W,\Phi)
+Z(s,W,\Phi')=0.
\]
Being an ideal of $\C[s]$, there exists $P_\pi(s)\in I(\pi)$ such that $I(\pi)=P_\pi(s)\C[s]$.
\lmref{L:main lemma} implies $P_\pi(s)\neq 0$. Put
\[
L_{{\rm RS}}\left(s,{\rm As}\,\pi\right)
=
E_\pi(s)P_\pi(s).
\]

Similarly, we let 
\[
E_{\pi^{\vee}}(s)
=
\zeta_\R(s-2\lambda_1+\epsilon_1)
\zeta_\R(s-2\lambda_2+\epsilon_2)
\zeta_\C(s-\lambda_1-\lambda_2+\epsilon_0/2).
\]
and $I\left(\pi^{\vee}\right)=P_{\pi^{\vee}}(s)\C[s]$ be the ideal of $\C[s]$ spanned by the set
\[
\stt{\left . \frac{Z(s,W\ot\omega^{-1},\Phi)}{E_{\pi^{\vee}}(s)}
\mbox{ }\right\vert\mbox{ } W\in\cW(\pi,\psi_\xi)_0,\,\Phi\in\cS(\R^2,\psi)}.
\]
We put
\[
L_{{\rm RS}}\left(s,{\rm As}\,\pi^{\vee}\right)
=
E_{\pi^{\vee}}(s)P_{\pi^\vee}(s).
\]

Notice that $L_{{\rm RS}}\left(s,{\rm As}\,\pi\right)$ satisfies the two conditions stated at the 
beginning of \subsecref{SS:G.C.D. of poles}. We will show that, up to a non-zero constant,
$L_{{\rm RS}}\left(s,{\rm As}\,\pi\right)$ is equal to $L_{\rm Gal}\left(s,{\rm As}\,\pi\right)$
in \subsecref{SS:matching L-factors}.

\subsubsection{The functional equation}\label{SS:the functional equation}
Functional equation between zeta integrals is indeed a consequence of a work of Loke 
\cite[Theorem 1.3]{Lok01}. 

Let $s\in\C$, $\Phi\in\cS(\R^2,\psi)$ and $g\in\GL_2(\R)$. Define two integrals
\begin{align}\label{E:Godement section}
\begin{split}
f^{(s)}_\Phi(g)
&=
|{\rm det}(g)|^s\int_{\R^{\x}}\Phi((0,t)g)\omega_0(t)|t|^{2s}d^{\x}t,\\
\t{f}^{(s)}_\Phi(g)
&=
\omega_0({\det}(g))^{-1}|{\rm det}(g)|^s\int_{\R^{\x}}\Phi((0,t)g)\omega^{-1}_0(t)|t|^{2s}d^{\x}t.
\end{split}
\end{align}
Here $\omega_0:=\omega|_{\R^{\x}}$. These two integrals converge absolutely for
${\rm Re}(s)\gg 0$ and have meromorphic continuations to whole complex plane.  Moreover, we have
\[
f^{(s)}_\Phi\in\cB\left(|\mbox{ }|^{s-1/2},\omega_0^{-1}|\mbox{ }|^{1/2-s}\right)_0
\quad\text{and}\quad
\t{f}^{(s)}_\Phi\in\cB\left(\omega_0^{-1}|\mbox{ }|^{s-1/2},|\mbox{ }|^{1/2-s}\right)_0,
\]
and the linear maps $\Phi\mapsto f^{(s)}_\Phi$ and $\Phi\mapsto \t{f}^{(s)}_\Phi$ are surjective
except for countable many $s$.

Let 
\[
Mf^{(s)}_\Phi(g)
:=
\int_\R
f^{(s)}_\Phi\left(
\begin{pmatrix}
0&1\\-1&0
\end{pmatrix}
\pMX 1x01 g
\right)dx
\quad\text{and}\quad
M^*:=
\omega_0(-1)\gamma\left(2s-1,\omega_0,\psi\right)M.
\]
The integral converges absolutely for ${\rm Re}(s)\gg 0$ and admits a 
meromorphic continuation to whole complex plane. It defines an intertwining operator
\[
M^*:\cB\left(|\mbox{ }|^{s-1/2},\omega_0^{-1}|\mbox{ }|^{1/2-s}\right)_0\longto
\cB\left(\omega_0^{-1}|\mbox{ }|^{1/2-s},|\mbox{ }|^{s-1/2}\right)_0,
\] 
between two Harish-Chandra modules and one check that
\begin{equation}\label{E:intertwining operator for induced representations}
M^*f^{(s)}_\Phi=\t{f}^{(1-s)}_{\wh{\Phi}}.
\end{equation}
These facts can be founded in \cite[section 4]{GJ79}.

By equations  
\eqref{E:expansion of zeta integral},
\eqref{E:Godement section} and 
\eqref{E:intertwining operator for induced representations}, we can write
\begin{align}\label{E:relating zeta integrals and invariant forms}
\begin{split}
Z(s,W,\Phi)
&=
\int_{\R^{\x}\U(\R)\backslash\GL_2(\R)}
W(g)f^{(s)}_\Phi(g)dg,\\
Z\left(1-s,W\ot\omega^{-1},\wh{\Phi}\right)
&=
\int_{\R^{\x}\U(\R)\backslash\GL_2(\R)}
W(g)M^*f^{(s)}_\Phi(g)dg.
\end{split}
\end{align}
Both integrals appeared in the RHS of \eqref{E:relating zeta integrals and invariant forms} 
define $\left(\mathfrak{g},{\rm O}(2)\right)$-invariant forms on the space
\[
\cW(\pi,\psi_\xi)_0\ot\cB\left(|\mbox{ }|^{s-1/2},\omega_0^{-1}|\mbox{ }|^{1/2-s}\right)_0.
\]
By a result of \cite[Theorem 1.3]{Lok01}, there is a meromorphic
function $\gamma_{{\rm RS}}\left(s,{\rm As}\,\pi,\psi,\xi\right)$, independent of $W$ and $\Phi$ 
such that
\[
Z\left(1-s,W\ot\omega^{-1},\wh{\Phi}\right)
=
\gamma_{{\rm RS}}\left(s,{\rm As}\,\pi,\psi,\xi\right)
Z(s,W,\Phi).
\]

Put
\[
\varepsilon_{{\rm RS}}\left(s,{\rm As}\,\pi,\psi,\xi\right)
=
\gamma_{{\rm RS}}\left(s,{\rm As}\,\pi,\psi,\xi\right)
\frac{L_{{\rm RS}}\left(s,{\rm As}\,\pi\right)}{L_{{\rm RS}}\left(1-s,{\rm As}\,\pi^{\vee}\right)}.
\]
Then we obtain the functional equation
\begin{equation}\label{E:functional equation}
\frac{Z\left(1-s,W\ot\omega^{-1},\wh{\Phi}\right)}
{L_{\rm RS}\left(1-s,{\rm As}\,\pi^{\vee}\right)}
=
\varepsilon_{{\rm RS}}\left(s,{\rm As}\,\pi,\psi,\xi\right)
\frac{Z\left(s,W,\Phi\right)}
{L_{\rm RS}\left(s,{\rm As}\,\pi\right)},
\end{equation}
for $W\in\cW(\pi,\psi_\xi)_0$ and $\Phi\in\cS(\R^2,\psi)$ as desired.

We show that $\varepsilon_{{\rm RS}}\left(s,{\rm As}\,\pi,\psi,\xi\right)\in\C^{\x}$. In fact,
by equations \eqref{E:sum of zeta integrals equal to L-factor} and 
\eqref{E:functional equation}, we see that 
\[
\varepsilon_{{\rm RS}}\left(s,{\rm As}\,\pi,\psi,\xi\right)\in\C[s].
\]
On the other hand, if we apply the functional equation twice and using again equation
\eqref{E:sum of zeta integrals equal to L-factor}, we find that
\[
\varepsilon_{{\rm RS}}\left(s,{\rm As}\,\pi,\psi,\xi\right)
\varepsilon_{{\rm RS}}\left(1-s,{\rm As}\,\pi^{\vee},\psi,\xi\right)
=
1.
\] 
This shows our claim.

\subsubsection{Matching $L$-factors}\label{SS:matching L-factors}
We show that, up to non-zero constants, $L_{{\rm RS}}\left(s,{\rm As}\,\pi\right)$ is equal to
$L_{\rm Gal}\left(s,{\rm As}\,\pi\right)$.
Recall that $\epsilon_j\in\stt{0,1}$ so that $n_j\equiv\epsilon_j\pmod 2$ for $j=0,1,2$, 
and we have
\[
L_{{\rm RS}}\left(s,{\rm As}\,\pi\right)
=
P_\pi(s)E_\pi(s)
\quad\text{and}\quad
L_{{\rm RS}}\left(s,{\rm As}\,\pi^{\vee}\right)
=
P_{\pi^{\vee}}(s)E_{\pi^{\vee}}(s),
\]
with $P_\pi\in\C[s]$ (resp. $P_{\pi^{\vee}}\in\C[s]$) is a generator of the ideal $I(\pi)$
(resp. $I(\pi^{\vee})$) defined in \subsecref{SS:G.C.D. of poles}. Here
\begin{align*}
E_\pi(s)
&=
\zeta_\R(s+2\lambda_1+\epsilon_1)
\zeta_\R(s+2\lambda_2+\epsilon_2)
\zeta_\C(s+\lambda_1+\lambda_2+\epsilon_0/2),\\
E_{\pi^{\vee}}(s)
&=
\zeta_\R(s-2\lambda_1+\epsilon_1)
\zeta_\R(s-2\lambda_2+\epsilon_2)
\zeta_\C(s-\lambda_1-\lambda_2+\epsilon_0/2).
\end{align*}
On the other hand,
\begin{align*}
L_{\rm Gal}\left(s,{\rm As}\,\pi\right)
&=
\zeta_\R(s+2\lambda_1+\epsilon_1)
\zeta_\R(s+2\lambda_2+\epsilon_2)
\zeta_\C(s+\lambda_1+\lambda_2+n_0/2),\\
L_{\rm Gal}\left(s,{\rm As}\,\pi^{\vee}\right)
&=
\zeta_\R(s-2\lambda_1+\epsilon_1)
\zeta_\R(s-2\lambda_2+\epsilon_2)
\zeta_\C(s-\lambda_1-\lambda_2+n_0/2).
\end{align*}

Let $n_0=\epsilon_0+2m$ for some non-negative integer $m$. By \lmref{L:main lemma} and the 
functional equation \eqref{E:functional equation}, we obtain the following equality
\begin{equation}\label{E:matching L-factors}
c^{\vee}\,
\frac{\prod_{j=0}^{m-1}\left(1-s-\lambda_1-\lambda_2+\epsilon_0/2+j\right)}
{P_{\pi^{\vee}}(1-s)}
=
c\,\varepsilon_{{\rm RS}}\left(s,{\rm As}\,\pi,\psi,\xi\right)
\frac{\prod_{j=0}^{m-1}\left(s+\lambda_1+\lambda_2+\epsilon_0/2+j\right)}
{P_{\pi}(s)}.
\end{equation}
Here we understand that $\prod_{j=0}^{m-1}=1$ if $m=0$. Since 
$\varepsilon_{{\rm RS}}\left(s,{\rm As}\,\pi,\psi,\xi\right)$ is a non-zero constant, we see that
both sides of the equation \eqref{E:matching L-factors} are elements of $\C[s]$ and they have no
common roots. This implies 
\begin{align*}
P_{\pi}(s)
=
\alpha\prod_{j=0}^{m-1}\left(s+\lambda_1+\lambda_2+\epsilon_0/2+j\right)
\quad\text{and}\quad
P_{\pi^{\vee}}(1-s)
=
\alpha^{\vee}\prod_{j=0}^{m-1}\left(1-s-\lambda_1-\lambda_2+\epsilon_0/2+j\right),
\end{align*}
for some non-zero constants $\alpha$ and $\alpha^{\vee}$. This proves the assertion.

\subsection{Proof of \lmref{L:main lemma} and the relation between epsilon factors}
\label{S:proof of main lemma and the relation between epsilon factors}
The purpose of this section is to prove \lmref{L:main lemma}. As a consequence, we obtain the 
relation \eqref{E:relation between epsilon factors}. To do this, we need some preparations.
We adopt the notations defined in \subsecref{SS:notation} as well as the Haar measures given 
in \subsecref{SS:Haar measures}. 
As in the proof of  \lmref{L:zeta integral vanishes for certain test functions}, for a 
non-negative integer $n$, let $\rho_n$ denote the $n$-th symmetric power of the standard 
two-dimensional representation of ${\rm GL}_2(\C)$ on $\C^2$. We realize $\rho_n$ as 
$\left(\rho,\cL_n(\C)\right)$ with
\begin{equation}\label{E:realization of symmetric power}
\cL_n(\C)=\bigoplus_{j=0}^{n}\C X^{j}Y^{n-j}
\quad\text{and}\quad
\rho(g)P(X,Y)=P((X,Y)g),
\end{equation}
for $P(X,Y)\in\cL_n(\C)$ and $g\in{\rm GL}_2(\C)$.
Let $\langle\cdot,\cdot\rangle_n$ denote the non-degenerated bilinear pairing on $\cL_n(\C)$ 
defined by
\begin{equation}\label{E:bilinear invariant pairing for symmetric power}
\langle X^{j}Y^{n-j}, X^{\ell}Y^{n-\ell}\rangle_n
=
\begin{cases}
(-1)^{j}\begin{pmatrix}n\\j\end{pmatrix}^{-1}\quad&\text{if $j+\ell=n$},\\
0\quad&\text{if $j+\ell\neq n$}.
\end{cases}
\end{equation} 
One check that 
\[
\langle \rho(g)P(X,Y), \rho(g)Q(X,Y)\rangle_n
=
{\rm det}^n(g)
\langle P(X,Y), Q(X,Y)\rangle_n.
\]
In particular, $\langle\cdot,\cdot\rangle_n$ defines an ${\rm SU}_2(\R)$-invariant bilinear pairing 
on $\cL_n(\C)$.

\subsubsection{Whittaker function of type $\rho_n$}
\label{SS:Whittaker function of type rho_n}
Let $\mu, \nu$ be as in the equation \eqref{E:equation of mu and nu}. Recall that $n_0=|n_1-n_2|$.
Let $n\geq n_0$ be an integer which has the same parity with $n_0$. By \cite[Lemma 6.1]{JL70},
we have
\[
{\rm dim}_\C{\rm Hom}_{{\rm SU}_2(\R)}\left(\rho_n,\pi\right)=1.
\]
As a consequence, there is a unique (up to constants) non-zero element 
$\vec{W}_\pi^{(n)}\in\cW(\pi,\psi_\xi)_0\ot\cL_n(\C)$ characterized by the following two conditions:
\begin{enumerate}
\item $\vec{W}_\pi^{(n)}(gu)=\rho_n(u)^{-1}\vec{W}_\pi^{(n)}(g)$ for every $g\in{\rm GL}_2(\C)$ and $u\in {\rm SU}_2(\R)$.
\item For every $v\in \cL_n(\C)$, the function $W_v(g):=\langle \vec{W}_\pi^{(n)}(g),v\rangle_n$ belongs
to $\cW(\pi,\psi_\xi;\rho_n)$, the isotypic component of $\rho_n$ in $\cW(\pi,\psi_\xi)_0$.
\end{enumerate}
Following \cite[section 18]{Jac72}, we refer to $\vec{W}_\pi^{(n)}$ as the Whittaker function of type 
$\rho_n$ attached to $\pi$. When $n=n_0$, we simply write $\vec{W}_\pi$ for $\vec{W}_\pi^{(n_0)}$.

\subsubsection{Construction of $\vec{W}_\pi^{(n)}$}
\label{SS:construction of W}
By the uniqueness of the space $\cW(\pi,\psi_\xi)_0$, we may assume $n_1\geq n_2$, so that
\[
n_0=n_1-n_2\geq 0.
\]
Let $m\in\Z$ such that $n=n_0+2m$. We review the construction of $\vec{W}_\pi^{(n)}$ given in 
\cite[section 18]{Jac72}.

Let $\Psi\in\cS(\C^2,\psi_\xi)$ and define the partial Fourier transform $\Psi^{\sim}$ of $\Psi$
by
\[
\Psi(z,w)^{\sim}
=
\int_{\C}\Psi(z,u)\psi_\xi(wu)du.
\]
Here $du$ is twice of the usual Lebesgue measure on $\C$.
One can define the partial Fourier transform of elements in $S(\C^2,\psi_\xi)\ot\cL_n(\C)$ in 
an obvious way. 

Let $\vec{\Psi}_\pi^{(n)}\in\cS(\C^2,\psi_\xi)\ot\cL_n(\C)$ be the element such that
\begin{equation}\label{E:test function for vector value Whittaker function}
\vec{\Psi}_\pi^{(n)}(z,w)^{\sim}
=
\left(wX-zY\right)^m\left(\b{z}X+\b{w}Y\right)^{n_0+m}e^{-2\pi(z\b{z}+w\b{w})}.
\end{equation}
Then 
\begin{equation}\label{E:vector value Whittaker function}
\vec{W}_\pi^{(n)}(g)
=
\mu({\rm det}(g))|{\rm det}(g)|_\C^{\frac{1}{2}}
\int_{\C^{\x}}
\omega_{\psi_\xi}(g)\vec{\Psi}_\pi^{(n)}\left(z,z^{-1}\right)\mu\nu^{-1}(z)d^{\x}z,\quad g\in{\rm GL}_2(\C).
\end{equation}
Here we use the same notation $\omega_{\psi_\xi}$ to indicate the representation 
$\omega_{\psi_\xi}\ot 1$ of ${\rm GL}_2(\C)$ on the space $\cS(\C^2)\ot\cL_n(\C)$. The Haar measure 
$d^{\x}z=|z|_\C^{-1}dz$, where $dz$ is twice of the usual Lebesgue measure on $\C$.

When $n=n_0$, we have
\begin{align*}
\vec{\Psi}_\pi^{(n_0)}(z,w)
&=
\sum_{\ell=0}^{n_0}
\begin{pmatrix}
n_0\\ \ell
\end{pmatrix}
\b{z}^\ell w^{n_0-\ell}e^{-2\pi(z\b{z}+w\b{w})}X^\ell Y^{n_0-\ell}\\
&=
\sum_{\ell=0}^{n_0}
\begin{pmatrix}
n_0\\ \ell
\end{pmatrix}
\Psi_{(0,\ell)}\ot\Psi_{(n_0-\ell,0)}(z,w) X^\ell Y^{n_0-\ell},
\end{align*}
in our previous notation (Cf. \subsecref{SS:notation}). As a consequence, we find that
(Cf. equation \eqref{E:basis for Whittaker functions})
\begin{equation}\label{E:Whittaker function of minimal type}
\vec{W}_\pi
=
\sum_{\ell=0}^{n_0}
\begin{pmatrix}
n_0\\ \ell
\end{pmatrix}
W_{((0,\ell),(n_0-\ell,0))}\,X^\ell Y^{n_0-\ell}.
\end{equation}

We also need the case when $n_0=0$ and $n=2$. In this case, we have
\begin{align*}
\vec{\Psi}_\pi^{(2)}(z,w)
&=
-\b{z}\b{w}e^{-2\pi(z\b{z}+w\b{w})}X^2+\left(\frac{1}{2\pi}-z\b{z}+w\b{w}\right)
e^{-2\pi(z\b{z}+w\b{w})}XY\\
&-zwe^{-2\pi(z\b{z}+w\b{w})}Y^2.
\end{align*}
Therefore we find that
\begin{align}\label{E:Whittaker function of type 2}
\begin{split}
\vec{W}_\pi^{(2)}
=
-W_{((0,1),(0,1))}\,X^2
+\left(
\frac{1}{2\pi}W_{((0,0),(0,0))}
-
W_{((1,1),(0,0))}
+
W_{((0,0),(1,1))}
\right)XY
-
W_{((1,0),(1,0))}\,Y^2.
\end{split}
\end{align}
\subsubsection{List of the results}\label{SS:list of the results}
We list the choices of $W$ and $\Phi$ in \lmref{L:main lemma} and the resulting epsilon factors in
the following. There are five cases which depend on the parities of $n_0$ and $n_1$. 
\begin{itemize}
\item[Case 1.]
$n_0\geq 0$ is even and $n_1$ is even.\\
Choose 
\[
W
=
\langle \vec{W}_\pi,\left(X+iY\right)^{\frac{n_0}{2}}\left(X-iY\right)^{\frac{n_0}{2}}\rangle_{n_0}
\quad\text{and}\quad
\Phi(x,y)
=
e^{-\pi(x^2+y^2)}.
\]
Then we have
\[
Z(s,W,\Phi)
=
\frac{\pi}{2}\,L_{\rm Gal}\left(s,{\rm As}\,\pi\right)
\quad\text{and}\quad
Z\left(1-s,W\ot\omega^{-1},\wh{\Phi}\right)
=
\frac{\pi}{2}\,L_{\rm Gal}\left(1-s,{\rm As}\,\pi^{\vee}\right).
\]
As a consequence,
\[
\varepsilon_{{\rm RS}}\left(s,{\rm As}\,\pi,\psi,\xi\right)=1.
\]
\item[Case 2.]
$n_0\geq 2$ is even and $n_1$ is odd.\\
Choose 
\[
W
=
\langle \vec{W}_\pi,\left(X+iY\right)^{\frac{n_0}{2}+1}\left(X-iY\right)^{\frac{n_0}{2}-1}\rangle_{n_0}
\quad\text{and}\quad
\Phi(x,y)
=
(x-iy)^2e^{-\pi(x^2+y^2)}.
\]
Then we have
\[
Z(s,W,\Phi)
=
\sqrt{-1}\,\pi\,L_{\rm Gal}\left(s,{\rm As}\,\pi\right)
\quad\text{and}\quad
Z\left(1-s,W\ot\omega^{-1},\wh{\Phi}\right)
=
\sqrt{-1}\,\pi\,L_{\rm Gal}\left(1-s,{\rm As}\,\pi^{\vee}\right).
\]
As a consequence,
\[
\varepsilon_{{\rm RS}}\left(s,{\rm As}\,\pi,\psi,\xi\right)=1.
\]
\item[Case 3.]
$n_0$ is odd and $n_1$ is even.\\
Choose 
\[
W
=
\langle \vec{W}_\pi,\left(X+iY\right)^{\frac{n_0+1}{2}}\left(X-iY\right)^{\frac{n_0-1}{2}}\rangle_{n_0}
\quad\text{and}\quad
\Phi(x,y)
=
(x-iy)e^{-\pi(x^2+y^2)}.
\]
Then we have
\[
Z(s,W,\Phi)
=
\frac{\pi}{2\sqrt{-1}}\,L_{\rm Gal}\left(s,{\rm As}\,\pi\right)
\quad\text{and}\quad
Z\left(1-s,W\ot\omega^{-1},\wh{\Phi}\right)
=
\frac{\pi}{2}\,L_{\rm Gal}\left(1-s,{\rm As}\,\pi^{\vee}\right).
\]
As a consequence,
\[
\varepsilon_{{\rm RS}}\left(s,{\rm As}\,\pi,\psi,\xi\right)=\sqrt{-1}.
\]
\item[Case 4.]
$n_0$ is odd and $n_1$ is odd.\\
Choose 
\[
W
=
\langle \vec{W}_\pi,
\left(X+iY\right)^{\frac{n_0+1}{2}}\left(X-iY\right)^{\frac{n_0-1}{2}}
\rangle_{n_0}
\quad\text{and}\quad
\Phi(x,y)
=
(x-iy)e^{-\pi(x^2+y^2)}.
\]
Then we have
\[
Z(s,W,\Phi)
=
-\frac{\pi}{2}\,L_{\rm Gal}\left(s,{\rm As}\,\pi\right)
\quad\text{and}\quad
Z\left(1-s,W\ot\omega^{-1},\wh{\Phi}\right)
=
\frac{\sqrt{-1}\,\pi}{2}\,L_{\rm Gal}\left(1-s,{\rm As}\,\pi^{\vee}\right).
\]
As a consequence,
\[
\varepsilon_{{\rm RS}}\left(s,{\rm As}\,\pi,\psi,\xi\right)=-\sqrt{-1}.
\]
\item[Case 5.]
$n_0=0$ and $n_1$ is odd.\\
Choose 
\[
W
=
\langle \vec{W}_\pi^{(2)},\left(X+iY\right)\left(X-iY\right)\rangle_{2}
\quad\text{and}\quad
\Phi(x,y)
=
e^{-\pi(x^2+y^2)}.
\]
Then we have
\[
Z(s,W,\Phi)
=
-\frac{\pi}{2}\,L_{\rm Gal}\left(s,{\rm As}\,\pi\right)
\quad\text{and}\quad
Z\left(1-s,W\ot\omega^{-1},\wh{\Phi}\right)
=
-\frac{\pi}{2}\,L_{\rm Gal}\left(1-s,{\rm As}\,\pi^{\vee}\right).
\]
As a consequence,
\[
\varepsilon_{{\rm RS}}\left(s,{\rm As}\,\pi,\psi,\xi\right)=1.
\]
\end{itemize}

\begin{rmk}
In the last case, one might wonder why we do not choice $W=\vec{W}_\pi$. The reason is 
$Z(s,\vec{W}_\pi,\Phi)=0$ for all $\Phi\in\cS(\R^2,\psi)$. 
\end{rmk}

With the recipes above, the relation \eqref{E:relation between epsilon factors} follows 
immediately. We remain that
\[
\varepsilon_{\rm Gal}\left(s,{\rm As}\,\pi,\psi\right)
=
\left(\sqrt{-1}\right)^{1+\epsilon_1+\epsilon_2+n_0}.
\]
Here $\epsilon_j\in\stt{0,1}$ such that $n_j\equiv\epsilon_j\pmod 2$ for $j=1,2$.

\subsubsection{Explicit computations}
In the following, we prove the assertions listed in \subsecref{SS:list of the results}. 
We only check case 3 and case 5 as the other cases are similar. We need a simple result.
\begin{lm}\label{L:combinatorial identity}
Let $N$ be a non-negative integer and $z,w\in\C$. We have
\[
\sum_{\ell=0}^N
\begin{pmatrix}
N\\\ell
\end{pmatrix}
\Gamma\left(z+\ell\right)\Gamma\left(w-\ell\right)
=
\frac{\Gamma(z)\Gamma(w-N)\Gamma(z+w)}{\Gamma(z+w-N)}.
\]
\end{lm}

\begin{proof}
Induction on $N$.
\end{proof}

First note that in any case, we have
\[
W(gk(\theta))\Phi((x,y)k(\theta))
=
W(g)\Phi(x,y),
\quad
g\in{\rm GL}_2(\C),\,k(\theta)\in {\rm SO}_2(\R)\,\,\text{and}\,\,x,y\in\R. 
\]
Similarly the function
$\left(W\ot\omega^{-1}\right)(g)\wh{\Phi}((x,y)g)$ is also right ${\rm SO}_2(\R)$-invariant.
As a consequence, we obtain 
\begin{align}\label{E:simplify expansion of ezta integrals}
\begin{split}
Z(s,W,\Phi)
&=
\left(
\int_{\R^{\x}}
W\left(\pMX y001\right)|y|^{s-1}d^{\x}y
\right)
\left(
\int_{\R^{\x}_+}
\Phi((0,t)\omega(t)t^{2s}d^{\x}t
\right),\\
Z\left(1-s,W\ot\omega^{-1},\wh{\Phi}\right)
&=
\left(
\int_{\R^{\x}}
W\left(\pMX y001\right)\omega^{-1}(y)|y|^{-s}d^{\x}y
\right)
\left(
\int_{\R^{\x}_+}
\wh{\Phi}(0,t)\omega^{-1}(t)t^{2-2s}d^{\x}t
\right).
\end{split}
\end{align}

In our notation 
\begin{align}
\begin{split}
\zeta(s,W,1)
&=
\int_{\R^{\x}}
W\left(\pMX y001\right)|y|^{s-1}d^{\x}y,\\
\zeta(1-s,W,\omega_0^{-1})
&=
\int_{\R^{\x}}
W\left(\pMX y001\right)\omega^{-1}(y)|y|^{-s}d^{\x}y.
\end{split}
\end{align} 
Here $1$ means the identity character of $\R^{\x}$ and $\omega_0:=\omega|_{\R^{\x}}$

Assume $n_0$ is odd and $n_1$ is even. By equations 
\eqref{E:integration formula for Tate C x R},
\eqref{E:bilinear invariant pairing for symmetric power}, 
\eqref{E:Whittaker function of minimal type}
and \lmref{L:combinatorial identity}, we find that
\begin{align*}
\zeta(s,W,1)
&=
\sum_{\ell=0}^{\frac{n_0-1}{2}}
\begin{pmatrix}
\frac{n_0-1}{2}\\ \ell
\end{pmatrix}
\zeta\left(s,W_{((0,2\ell),(n_0-2\ell,0))},1\right)\\
&=
(2\pi)^{-(s+\lambda_1+\lambda_2+n_0/2-1)}
\sum_{\ell=0}^{\frac{n_0}{2}-1}
\begin{pmatrix}
\frac{n_0-1}{2}\\ \ell
\end{pmatrix}
\Gamma\left(\frac{s}{2}+\lambda_1+\ell\right)
\Gamma\left(\frac{s}{2}+\lambda_2+\frac{n_0}{2}-\ell\right)\\
&=
(2\pi)^{-(s+\lambda_1+\lambda_2+n_0/2-1)}
\frac{\Gamma\left(\frac{s}{2}+\lambda_1\right)\Gamma\left(\frac{s+1}{2}+\lambda_2\right)
\Gamma\left(s+\lambda_1+\lambda_2+\frac{n_0}{2}\right)}
{\Gamma\left(s+\lambda_1+\lambda_2+\frac{1}{2}\right)}.
\end{align*}
It then follows from the equation \eqref{E:simplify expansion of ezta integrals} that
\begin{align*}
Z(s,W,\Phi)
&=
(2\pi)^{-(s+\lambda_1+\lambda_2+n_0/2-1)}
\frac{\Gamma\left(\frac{s}{2}+\lambda_1\right)\Gamma\left(\frac{s+1}{2}+\lambda_2\right)
\Gamma\left(s+\lambda_1+\lambda_2+\frac{n_0}{2}\right)}
{\Gamma\left(s+\lambda_1+\lambda_2+\frac{1}{2}\right)}\\
&\x
\left(-2^{-1}\sqrt{-1}\right)\pi^{-(s+\lambda_1+\lambda_2+1/2)}
\Gamma\left(s+\lambda_1+\lambda_2+\frac{1}{2}\right)\\
&=
\frac{\pi}{2\sqrt{-1}}\,L_{\rm Gal}\left(s,{\rm As}\,\pi\right).
\end{align*}
On the other hand, by using equations 
\eqref{E:integration formula for Tate C x R},
\eqref{E:bilinear invariant pairing for symmetric power}, 
\eqref{E:Whittaker function of minimal type}
and \lmref{L:combinatorial identity} again, we obtain
\begin{align*}
\zeta\left(1-s,W,\omega_0^{-1}\right)
&=
\left(-\sqrt{-1}\right)
\sum_{\ell=0}^{\frac{n_0-1}{2}}
\begin{pmatrix}
\frac{n_0-1}{2}\\ \ell
\end{pmatrix}
\zeta\left(1-s,W_{((0,2\ell+1),(n_0-2\ell-1,0))},\omega_0^{-1}\right)\\
&=
\left(-\sqrt{-1}\right)(2\pi)^{-(n_0/2-s-\lambda_1-\lambda_2)}
\sum_{\ell=0}^{\frac{n_0-1}{2}}
\begin{pmatrix}
\frac{n_0-1}{2}\\ \ell
\end{pmatrix}
\Gamma\left(1-\frac{s}{2}-\lambda_2+\ell\right)
\Gamma\left(\frac{n_0-s}{2}-\lambda_1-\ell\right)\\
&=
\left(-\sqrt{-1}\right)(2\pi)^{-(n_0/2-s-\lambda_1-\lambda_2)}
\frac
{\Gamma\left(\frac{1-s}{2}-\lambda_1\right)
\Gamma\left(1-\frac{s}{2}-\lambda_2\right)
\Gamma\left(1-s-\lambda_1-\lambda_2+\frac{n_0}{2}\right)}
{\Gamma\left(\frac{3}{2}-s-\lambda_1-\lambda_2\right)}.
\end{align*}
Since $\wh{\Phi}=-\Phi$, by using the equation 
\eqref{E:simplify expansion of ezta integrals} again, we find that 
\begin{align*}
Z\left(1-s,W\ot\omega^{-1},\wh{\Phi}\right)
&=
\left(-\sqrt{-1}\right)(2\pi)^{-(n_0/2-s-\lambda_1-\lambda_2)}
\frac
{\Gamma\left(\frac{1-s}{2}-\lambda_1\right)
\Gamma\left(1-\frac{s}{2}-\lambda_2\right)
\Gamma\left(1-s-\lambda_1-\lambda_2+\frac{n_0}{2}\right)}
{\Gamma\left(\frac{3}{2}-s-\lambda_1-\lambda_2\right)}\\
&\x
\left(2^{-1}\sqrt{-1}\right)\pi^{-(3/2-s-\lambda_1-\lambda_2)}
\Gamma\left(\frac{3}{2}-s-\lambda_1-\lambda_2\right)\\
&=
\frac{\pi}{2}\,L_{\rm Gal}\left(1-s,{\rm As}\,\pi^{\vee}\right).
\end{align*}
These show the assertions for case 3.

We now consider case 5. Hence we assume $n_0=0$ and $n_1$ is odd.
By equations 
\eqref{E:integration formula for Tate C x R},
\eqref{E:bilinear invariant pairing for symmetric power} and
\eqref{E:Whittaker function of type 2}, we have 
\begin{align*}
\zeta(s,W,1)
&=
-\zeta\left(s,W_{((0,1),(0,1))},1\right)
-\zeta\left(s,W_{((1,0),(1,0))},1\right)\\
&=
-2(2\pi)^{-(s+\lambda_1+\lambda_2)}
\Gamma\left(\frac{s+1}{2}+\lambda_1\right)
\Gamma\left(\frac{s+1}{2}+\lambda_2\right).
\end{align*}
Therefore by equation \eqref{E:simplify expansion of ezta integrals}, we find that
\begin{align*}
Z(s,W,\Phi)
=
-\frac{\pi}{2}\,L_{\rm Gal}\left(s,{\rm As}\,\pi\right).
\end{align*}
Similarly, we have
\begin{align*}
\zeta\left(1-s,W,\omega_0^{-1}\right)
&=
-\zeta\left(s,W_{((0,1),(0,1))},\omega_0^{-1}\right)
-\zeta\left(s,W_{((1,0),(1,0))},\omega_0^{-1}\right)\\
&=
-2(2\pi)^{-(1-s-\lambda_1-\lambda_2)}
\Gamma\left(1-\frac{s}{2}-\lambda_1\right)
\Gamma\left(1-\frac{s}{2}-\lambda_2\right),
\end{align*}
and hence
\[
Z\left(1-s,W\ot\omega^{-1},\wh{\Phi}\right)
=
-\frac{\pi}{2}\,L_{\rm Gal}\left(1-s,{\rm As}\,\pi^{\vee}\right).
\]
This proves \lmref{L:main lemma} and hence completes the proof of \thmref{T:main theorem}.

\section{Twisted Asai local factors}\label{S:4}

In this section, we consider the twisted Asai local factors. The main result is Theorem \ref{T:twisted Asai factor} whose proof uses the functoriality of global Asai transfer proved in \cite{Kri03} and Theorem \ref{T:twisted Asai gamma factor}. When $E=F\times F$, Theorem \ref{T:twisted Asai gamma factor} is a theorem in \cite{Ikeda1989}. As mentioned by Ikeda in the proof of \cite[Lemma 2.2]{Ikeda1992}, the result should be true when $E$ is a field. Following the idea of Ikeda and combine our results in the case $n=1$, we give a detailed proof of it.

\subsection{Notation}\label{S:4.1}
For $n \in \Z_{> 0}$, let $\GSp_n$ be the linear algebraic group over $\Q$ defined by
$$\GSp_n =\left \{ g \in \GL_{2n}\mbox{ } \left \vert \mbox{ } g\bp {\bf 0}_n & {\bf 1}_n \\ -{\bf 1}_n & {\bf 0}_n \ep {}^tg = \nu(g) \bp {\bf 0}_n & {\bf 1}_n \\ -{\bf 1}_n & {\bf 0}_n \ep ,\mbox{ }\nu(g) \in \mathbb{G}_m    \right .\right \}.$$
The map $\nu : \GSp_n \rightarrow \mathbb{G}_m$ is called the scale map. Let $\P_n$ and $\B_n$ be parabolic subgroups of $\GSp_n$ defined by
\begin{align*}
\P_n & = \left \{ \left .  \bp A & * \\ {\bf 0}_n & \nu{}^tA^{-1}\ep \in \GSp_n \mbox{ } \right\vert \mbox{ }A \in \GL_n, \nu \in {\mathbb G}_m\right \},\\
\B_n &=\left \{ \left .  \bp A & * \\ {\bf 0}_n & \nu{}^tA^{-1}\ep \in \GSp_n \mbox{ } \right\vert \mbox{ }A\mbox{ is upper triangular},\mbox{ } \nu \in {\mathbb G}_m \right \}.
\end{align*}
Let $U_n$ and $N_n$ be the unipotent subgroups of $\P_n$ and $\B_n$, respectively. Then
\begin{align*}
U_n &= \left \{ \left . \bp {\bf 1}_n & X \\ {\bf 0}_n & {\bf 1}_n\ep \in \P_n   \mbox{ } \right \vert \mbox{ } X ={}^tX  \right \},\\
N_n&= \left \{ \left . \bp A & * \\ {\bf 0}_n & \nu{}^tA^{-1}\ep \in \B_n \mbox{ } \right \vert \mbox{ } A\mbox{ is unipotent upper triangular}, \mbox{ }\nu \in {\mathbb G}_m \right \}.
\end{align*}
Let $M_n$ and $T_n$ be Levi subgroups of $\P_n$ and $\B_n$ given by
\begin{align*}
M_n & = \left \{ \left . \bp A & {\bf 0}_n \\ {\bf 0}_n & \nu{}^tA^{-1}\ep \in \P_n  \mbox{ } \right \vert \mbox{ } A \in \GL_n, \mbox{ }\nu \in {\mathbb G}_m   \right \},\\
T_n & = \left \{ \left . \bp A & {\bf 0}_n \\ {\bf 0}_n & \nu{}^tA^{-1}\ep \in \B_n \mbox{ } \right \vert \mbox{ } A \mbox{ is diagonal, }\nu \in {\mathbb G}_m \right \} .
\end{align*}
Let 
\begin{align*}
K_n= \left \{ \begin{array}{ll} \GSp_n(F)\cap\GL_{2n}(\mathcal{O}_F)& \mbox{ if $F$ is nonarchimedean},\\
\GSp_n(\R) \cap O_{2n}(\R) & \mbox{ if $F=\R$}
\end{array} \right .
\end{align*}
 be a maximal compact subgroup of $\GSp_n(F)$. Let $K_{n}^{1} = \Sp_n(F) \cap K_n$. For $A \in \GL_n$, $\nu \in {\mathbb G}_m$, and $X \in {\rm M}_n$ with $X={}^tX$, let
\begin{align*}
m(A,\nu) &=  \bp A & {\bf 0}_n \\ {\bf 0}_n & \nu{}^tA^{-1}\ep ,& u(X) &= \bp {\bf 1}_n & X \\ {\bf 0}_n & {\bf 1}_n \ep ,& u_-(X) &= \bp {\bf 1}_n & {\bf 0}_n \\ X & {\bf 1}_n \ep
\end{align*}

\subsection{Intertwining operators}\label{S:4.2}
Let $N(T_3)$ be the normalizer of $T_3$ in $\GSp_3$, and $W = N(T_3) / T_3$ be the corresponding Weyl group. By abuse of notation, we write $w$ for the coset $w T_3 \in W$. Let $J_3 \in \GL_3$ be the anti-diagonal matrix with non-zero entries all equal to $1$, and $w_3 \in W$ be a Weyl element given by 
\begin{align*}
w_3 = \bp {\bf 0}_3 & -J_3 \\ J_3 & {\bf 0}_3 \ep. 
\end{align*}

For $i=1,2,3$, let $x_i$ be the character of $T_3$ defined by
$$x_i \left ( m\left( \bp t_1 & 0 & 0\\0 & t_2 & 0 \\ 0 & 0 & t_3\ep ,\nu\right)\right )=t_i.$$
The set of simple roots of $T_3$ with respect to $B_3$ is given by $\Delta_3 =\{\alpha_1,\alpha_2,\alpha_3 \}$ with
\begin{align*}
&\alpha_1 = x_1-x_2,\quad \alpha_2 = x_2-x_3, \quad \alpha_3=2x_3.
\end{align*}
For $i=1,2,3$, let $\kappa_{(i)}$ be the embedding of $\GL_2$ into $\GSp_3$ defined by
\begin{align*}
\kappa_{(1)} : \GL_2 &\longrightarrow \GSp_3,\quad g  \longmapsto \bp g & & &  \\ & 1 & &\\ & & {}^tg^{-1}& \\  & & &1   \ep,\\
\kappa_{(2)} : \GL_2 &\longrightarrow \GSp_3,\quad g  \longmapsto \bp 1 & & &  \\ & g & &\\ & & 1& \\  & & &{}^tg^{-1}   \ep,\\
\kappa_{(1)} : \GL_2 &\longrightarrow \GSp_3,\quad \bp a & b \\ c & d \ep \longmapsto \bp ad-bc & 0 & 0 & 0& 0&0  \\ 0 & ad-bc & 0 & 0 & 0 & 0 \\ 0  & 0& a& 0&0&b\\ 0&0&0&1&0&0\\ 0&0&0&0&1&0\\ 0&0&c&0&0&d  \ep .
\end{align*}
Denote $w_{(i)} = \kappa_{(i)}\left ( \bp 0 & -1 \\ 1 & 0 \ep \right) \in N(T_3)$. Note that the image of $\kappa_{(i)}$ is equal to the centralizer of the connected component of ${\rm Ker}\,\alpha_i$ in $\GSp_3$.

For $\chi$ be a character of $T_3$. Define ${\rm Ind}_{B_3(F)}^{\GSp_3(F)}(\chi)$ to be the space consisting of functions $f : \GSp_3(F) \rightarrow \C$ which satisfies the following conditions:
\begin{itemize}
\item $f$ is right $K_3$-finite.
\item For $g \in \GSp_3(F)$ and $b\in B_3(F)$, we have
$$f(bg) = \chi(b)\delta_{B_3}(b)^{1/2}f(g).$$
Here $\delta_{B_3}$ is the modulus character of $B_3$.
\end{itemize}

For $w \in W$ and $\chi$ a character of $T_3$, we define the intertwining operator
\begin{align*}
&I_{w} : {\rm Ind}_{B_3(F)}^{\GSp_3(F)}(\chi)\longrightarrow {\rm Ind}_{B_3(F)}^{\GSp_3(F)}(\chi^{w}),\\
&I_{w}f(g) = \int_{N_3(F) \cap w N_3^{-}(F)w^{-1}}f(w^{-1}nh)dn.
\end{align*}
Here $N_3^- $ is the unipotent radical of the opposite Borel subgroup of $B_3.$ The Haar measure $dn$ is defined to be the product measure of the self-dual Haar measure with respect to $\psi$ of each coordinate. The integral is absolutely convergent if $\chi$ belongs to some open subset and can be meromorphically continued to all $\chi$.

Note that for $x,y,z,w,u,t \in F$, we have 
\begin{align*}
&\kappa_{(3)}\left(\bp 1 & 0 \\ x & 1 \ep \right)w_{(3)}^{-1}\kappa_{(2)}\left(\bp 1 & 0 \\ y & 1 \ep \right)w_{(2)}^{-1}\kappa_{(1)}\left(\bp 1 & 0 \\ z & 1 \ep \right)w_{(1)}^{-1}\\
&\times\kappa_{(3)}\left(\bp 1 & 0 \\ w & 1 \ep \right)w_{(3)}^{-1}\kappa_{(2)}\left(\bp 1 & 0 \\ u & 1 \ep \right)w_{(2)}^{-1}\kappa_{(3)}\left(\bp 1 & 0 \\ t & 1 \ep \right)w_{(3)}^{-1}\\ &= u_-\left(\bp x & y & z \\y & w &u \\ z & u & t\ep\right) m \left ( \bp 1&0&0\\0&-1&0\\0&0&1\ep ,1\right)w_3^{-1}.
\end{align*}
Therefore, for $\chi$ a character of $T_3$ and $f \in {\rm Ind}_{B_3(F)}^{\GSp_3(F)}(\chi)$, we have
\begin{align}\label{E:composite of intertwining operators}
I_{w_3}f = \chi\left(m \left ( \bp 1&0&0\\0&-1&0\\0&0&1\ep ,1\right)\right)I_{w_{(3)}}I_{w_{(2)}}I_{w_{(3)}}I_{w_{(1)}}I_{w_{(2)}}I_{w_{(3)}}f.
\end{align}

Let $\chi$ be a character of $T_3$, $\mu$ be a character of $F^{\times}$, and $f \in {\rm Ind}_{B_3(F)}^{\GSp_3(F)}(\chi)$. Define functions $W_{w_{(1)}}f$, $W_{w_{(3)}}f$, and $Z_{(1)}(s,\mu,W_{w_{(1)}}f,- )$ on $\GSp_3(F)$ by
\begin{align*}
W_{w_{(1)}}f(g) &= \int_{F}f\left(w_{(1)}^{-1}\kappa_{(1)}\left(\bp  1 & x \\ 0 & 1 \ep \right)g\right)\psi(-x)dx,\\
W_{w_{(3)}}f(g) &= \int_{F}f\left(w_{(3)}^{-1}\kappa_{(3)}\left(\bp  1 & x \\ 0 & 1 \ep \right)g\right)\psi(-x)dx,\\
Z_{(1)}(s,\mu,W_{w_{(1)}}f,g) &= \int_{F^{\times}}W_{w_{(1)}}f\left(\kappa_{(1)}\left(\bp a&0\\0&1 \ep\right) g\right)\mu(a)|a|_F^{s-1/2}d^{\times}a.
\end{align*}
The integrals defining $W_{w_{(1)}}f$ and $W_{w_{(3)}}f$ are absolutely converge for $\chi$ in some open sets and can be meromorphically continued to all $\chi$. The integral defining 
$Z_{(1)}(s,\mu,W_{w_{(1)}}f,- )$ are absolutely converge for $\mu$, and $s$ in some open subsets and can be meromprhically continued to all $\mu$ and $s$. 

Let $\omega$ be a character of $F^{\times}$. Define $I(\omega,s)$ be the space consisting of functions $f : \GSp_3(F) \longrightarrow \C$ which satisfies the following conditions:
\begin{itemize}
\item $f$ is right $K_3$-finite.
\item For $g \in \GSp_3(F)$, $A \in \GL_3(F)$, and $\nu \in F^{\times}$, we have
$$f \left ( m(A,\nu)g\right )  = \omega^{-2}(\nu)|\nu|_F^{-3s-3/2}\omega(\det(A))|\det(A)|_F^{2s+1}f(g).$$
\end{itemize}
Note that $I(\omega,s)\subseteq {\rm Ind}_{B_3(F)}^{\GSp_3(F)}(\chi_{\omega,s})$, where $\chi_{\omega, s}$ is a character of $T_3$ defined by
\begin{align*}
\chi_{\omega, s}\left ( m\left( \bp t_1 & 0 & 0\\0 & t_2 & 0 \\ 0 & 0 & t_3\ep ,\nu\right)\right ) = \omega^{-2}(\nu)|\nu|_F^{-3s+3/2}\omega(t_1)|t_1|_F^{2s-2}\omega (t_2)|t_2|_F^{2s-1}\omega(t_3)|t_3|_F^{2s}.
\end{align*}
For $f \in I(\omega,s)$, we have $\omega(\nu(g))I_{w_3}f(g) \in I(\omega^{-1},1-s)$. Let $I_{w_3}^* : I(\omega, s) \longrightarrow I(\omega^{-1},1-s)\otimes (\omega^{-1}\circ \nu)$ be the normalized intertwining operator defined by
$$I_{w_3}^*= \gamma(2s-2,\omega,\psi)\gamma(4s-3,\omega^2,\psi)I_{w_3}.$$
Note that $I_{w_3}^*$ is well-defined except for countably many values of $s$.

We say that a function $(s,g)\mapsto f^{(s)}(g)$ on $\GSp_3(F)\times\C$ is a holomorphic section of $I(\omega,s)$ if $f^{(s)}$ satisfies the following conditions:
\begin{itemize}
\item $f^{(s)}(g)$ is right $K_3$-finite.
\item For each $s \in \C$, $f^{(s)}(g)$ belongs to $I(\omega,s)$ as a function of $g \in \GSp_3(F)$.
\item For each $g \in \GSp_3(F)$, $f^{(s)}(g)$ is holomorphic in $s$.
\end{itemize}
A function $(s,g)\mapsto f^{(s)}(g)$ on $\GSp_3(F)\times\C$ is called a meromorphic section of $I(\omega, s)$ if there exist a non-zero entire function $\beta$ such that $\beta(s)f^{(s)}(g)$ is a holomorphic section of $I(\omega, s)$. We can define holomorphic sections and meromorphic sections of ${\rm Ind}_{B_3(F)}^{\GSp_3(F)}(\chi_{\omega,s})$ in a similar way. 

A meromorphic section $f^{(s)}(g)$ of $I(\omega, s)$ is called a good section if it satisfies the following condition:
\begin{align*}
a_w(2s-1)^{-1}I_{w}f^{(s)}\mbox{ is holomorphic for all }w\in\Omega_3.
\end{align*}
Here $\Omega_3 \subset N(T_3)$ is certain subset defined in \cite[page 191]{Ikeda1992} and $a_w(s)$ is a meromorphic functions defined in \cite[page 195]{Ikeda1992} for each $w\in\Omega_3$. By \cite[Lemma 1.2]{Ikeda1992}, $f^{(s)}(g)$ is a good section of $I(\omega,s)$ if and only if $\omega(\nu(g))I_{w_3}^*f^{(s)}(g)$ is a good section of $I(\omega^{-1},1-s)$. In particular, if $f^{(s)}$ is a good section of $I(\omega,s)$, then
\begin{align}\label{properties of good sections}
\begin{split}
&L(2s+1,\omega)^{-1}L(4s,\omega^2)^{-1}f^{(s)}(g)\mbox{ is a holomorphic section of $I(\omega, s)$,}\\
&\omega(\nu(g))L(1-2s,\omega^{-1})^{-1}L(4-4s,\omega^{-2})^{-1}I_{w_3}^*f^{(s)}(g)\mbox{ is a holomorphic section of $I(\omega^{-1},1-s).$}
\end{split}
\end{align}
Also, a holomorphic section $f^{(s)}$ of $I(\omega,s)$ is a good section (Cf. \cite[Lemma 1.3]{Ikeda1992}).

\subsection{Twisted Asai local factors vie local zeta integrals}\label{S:4.3}
Define algebraic groups over $F$ 
\begin{align*}
\G &= \{(g^{(1)},g^{(2)})\in ({\rm R}_{E /F}\GL_{2/E} )\times \GL_{2/F} \mbox{ }\vert \mbox{ }\det(g^{(1)})=\det(g^{(2)}) \},\\
{\bf U}_0 &= \{ (u(x),u(y))\in ({\rm R}_{E /F}\GL_{2/E} )\times \GL_{2/F}  \mbox{ } \vert \mbox{ }{\rm tr}_{E/F}(x)+y=0 \},\\
{\bf U}_0^{'} & = \{ (u(x),u(y))\in {\bf U}_0  \mbox{ } \vert \mbox{ }y=0 \},\\
\GL_2^{\circ} & = \{ g \in {\rm R}_{E/F}\GL_{2/ E} \mbox{ }\vert \mbox{ }\det(g) \in {\mathbb G}_{m/F}\},\\
U^{\circ} & = \{ u(x) \in {\rm R}_{E/F}\GL_{2/ E} \mbox{ }\vert \mbox{ } {\rm tr}_{E/F}(x)=0\}.
\end{align*}
Let $$K^{\circ}= \left \{ \begin{array}{ll} \GL_2^{\circ}(F)\cap \GL_2(\mathcal{O}_E) & \mbox{ if $F$ is nonarchimedean},\\
\GL_2^{\circ}(\R)\cap {\rm U}_2(\R) & \mbox{ if $F=\R$} \end{array} \right .$$ be an open compact subgroup of $\GL_2^{\circ}(F)$.

Fix $\xi \in E^{\times}$ such that ${\rm tr}_{E/F}(\xi)=0$. Define embeddings
\begin{align*}
\iota_{2,\xi} : \GL_2^{\circ}(F) &\longrightarrow \GSp_2(F)\\
\begin{pmatrix}a & b \\ c & d \end{pmatrix} &\longmapsto \left (\begin{array}{cccccc} a_1 & a_2 & 2b_1 & 2\xi^2b_2  \\
\xi^2a_2 & a_1  & 2\xi^2b_2 &2\xi^2b_1  \\
c_1/2 & c_2/2& d_1 & \xi^2d_2\\
c_2/2& \xi^{-2}c_1/2 & d_2 & d_1   \end{array} \right ),\\
\iota_{\xi} : \G(F) &\longrightarrow \GSp_3(F)\\
\left (\begin{pmatrix}a & b \\ c & d \end{pmatrix},\begin{pmatrix}a' & b' \\ c' & d' \end{pmatrix} \right ) &\longmapsto \left (\begin{array}{cccccc} a_1 & a_2 & 0 & 2b_1 & 2\xi^2b_2 & 0 \\
\xi^2a_2 & a_1 & 0 & 2\xi^2b_2 &2\xi^2b_1 & 0 \\
0 & 0 &a'&0&0&b'\\
c_1/2 & c_2/2& 0& d_1 & \xi^2d_2&0\\
c_2/2& \xi^{-2}c_1/2& 0 & d_2 & d_1 & 0\\
0&0&c'&0&0&d'
   \end{array} \right ),\\
\gamma : \GSp_2(F) &\longrightarrow \GSp_3(F)\\
   \bp A&B \\ C&D \ep  &\longmapsto  \bp A&0&B&0 \\ 0&\nu&0&0\\ C&0&D&0 \\ 0&0&0&1 \ep,
\end{align*}
here $a=a_1+\xi a_2,b=b_1+\xi b_2,c=c_1+\xi c_2,d=d_1+\xi d_2$, and $\nu = \nu \left ( \bp A&B \\ C&D \ep\right )$. Define $\eta, \eta' \in \Sp_3(\Z)$ by
\begin{align*}
 \eta &= \bp 0 & 0 & 0 & -1 & 0 & 0 \\ 0 & 1 & 0 & 0 & 0 &0 \\ 0&0&1&0&0&0 \\ 1&0&1&0&0&0 \\ 0&0&0&0&1&0 \\ 0&0&0&-1&0&1 \ep,& \eta'&=\eta\gamma \left ( \bp 0&0&0&1\\-1&0&0&0\\1&-1&0&0\\0&0&-1&-1 \ep \right ).
\end{align*}
Note that 
\begin{align}\label{E:identity for congugation by eta}
\begin{split}
\eta \iota_{\xi}(g) \eta^{-1} & = \bp d_1 & -c_2/2 & c_1/2 & -c_1/2 & -\xi^2d_2 & 0 \\ -2\xi^2b_2 & a_1 & -\xi^2a_2  & \xi^2 a_2 & 2\xi^2 b_1 &0\\ -b' & 0 & a' & 0 & 0  & b' \\ -2b_1-b' & a_2 & -a_1+a' & a_1 & 2\xi^2 b_2 & b' \\ -d_2 & \xi^{-2}c_1/2 & -c_2/2 & c_2/2 & d_1 & 0 \\  d_1-d' & -c_2/2 & c_1/2+c' & -c_1/2 & -\xi^2 d_2 & d'  \ep.
\end{split}
\end{align}
In particular,for $x,y,z \in F$, $a \in F^{\times}$,
\begin{align}\label{E:identity for conjugation by eta special form}
\begin{split}
\eta \iota_{\xi} \left( \bp 1 & 0 \\ z+\xi x & 1  \ep  ,\bp a & 0 \\ ay & a^{-1} \ep  \right ) \eta^{-1}&= m\left ( \bp 1&-x/2&az/2 \\ 0 & 1 & 0 \\ 0 & 0 & a  \ep , 1\right) u\left( \bp -z/2 & 0 & 0 \\ 0&0&0 \\ 0&0&0\ep  \right )\\
&\times u_{-}\left(\bp   0&0&a-1\\ 0&\xi^{-2}z/2&-ax/2 \\ a-1 & -ax/2 & a^2y+a^2z/2    \ep \right).
\end{split}
\end{align}

Let $\pi_1$ and $\pi_2$ be generic irreducible admissible representations of $\GL_2(E)$ and $\GL_2(F)$ with central characters $\omega_1$ and $\omega_2$, respectively. Put $\omega=\omega_1\vert_{F^{\times}}\cdot\omega_2$. Let $f^{(s)}$ be a good section of $I(\omega,s)$, and $W_1 \in \mathscr{W}(\pi_1,\psi_E)_0$, $W_2 \in \mathscr{W}(\pi_2,\psi)_0$. Define the local zeta integral (Cf. \cite{PSR1987} and \cite{Ikeda1989})
\begin{align*}
\Psi(f^{(s)},W_1,W_2) = \int_{F^{\times}{\bf U}_0(F) \backslash {\bf G}(F)} f^{(s)}(\eta \iota_{\xi}(g)) W_1(g^{(1)})W_2(g^{(2)})dg.
\end{align*}
In \cite{PSR1987} and \cite{Ikeda1989}, it is proved that the local zeta integral is absolutely convergent for ${\rm Re}(s) \gg 0$ and can be meromorphically continued to $s \in \C$. Moreover, there exist a meromorphic function $\gamma(s,{\rm As}\,\pi_1 \otimes \pi_2,\psi,\xi)$ such that the following functional equation holds
\begin{align}\label{E:functional equation for twisted Asai}
\Psi(I_{w_3}^*f^{(s)},W_1,W_2) = \gamma_{\rm PSR}(s,{\rm As}\,\pi_1 \otimes \pi_2,\psi,\xi)\Psi(f^{(s)},W_1,W_2).
\end{align}
Note that for $a \in F^{\times}$, we have
\begin{align}\label{E:dependence on psi 2}
\begin{split}
\gamma_{\rm PSR}(s,{\rm As}\,\pi_1 \otimes \pi_2,\psi^{a},\xi)&=\omega(a)^4|a|_{F}^{8s-4}\gamma_{\rm PSR}(s,{\rm As}\,\pi_1 \otimes \pi_2,\psi,\xi),\\
\gamma_{\rm PSR}(s,{\rm As}\,\pi_1 \otimes \pi_2,\psi,a\xi)&= \omega(a)^{-2}|a|_F^{-4s+2}\gamma_{\rm PSR}(s,{\rm As}\,\pi_1 \otimes \pi_2,\psi,\xi).
\end{split}
\end{align}

Assume $F$ is non-archimedean. Let $q$ be the cardinality of the residue field of $F$. By \cite[Appendix 3 to \S3]{PSR1987}, the ideal generated over $\C[q^s,q^{-s}]$ by $\Psi(f^{(s)},W_1,W_2)$ for good sections $f^{(s)}$ and $W_1 \in \mathscr{W}(\pi_1,\psi_E)_0$, $W_2 \in \mathscr{W}(\pi_2,\psi)_0$ is a fractional ideal of $\C[q^s,q^{-s}]$ containing $1$. Therefore, there is a unique generator of the form $P(q^{-s})$ with $P(X) \in \C[X]$ and $P(0)=1$. Put 
\begin{align*}
L_{\rm PSR}(s,{\rm As}\,\pi_1\otimes\pi_2)&=P(q^{-s})^{-1},\\
\varepsilon_{\rm PSR}(s,{\rm As}\,\pi_1\otimes\pi_2,\psi,\xi)&=\gamma_{\rm PSR}(s,{\rm As}\,\pi_1\otimes\pi_2,\psi,\xi)L_{\rm PSR}(s,{\rm As}\,\pi_1\otimes\pi_2)L_{\rm PSR}(1-s,{\rm As}\,\pi_1^{\vee}\otimes\pi_2^{\vee})^{-1}.
\end{align*}
By (\ref{E:functional equation for twisted Asai}), $\varepsilon_{\rm PSR}(s,{\rm As}\,\pi_1\otimes\pi_2,\psi,\xi) \in \C[q^{s},q^{-s}]^{\times}$.

Assume $F$ is archimedean. One can deduece from the proofs of \cite[Proposition 5.1]{Ikeda1989} and the following proposition that there exist a meromorphic function $\alpha(s)$ without zeros such that $\alpha(s)^{-1}Z(f^{(s)},W_1,W_2)$ is holomorphic for any good section $f^{(s)}$ and $W_1 \in \mathscr{W}(\pi_1,\psi_E)_0$, $W_2 \in \mathscr{W}(\pi_2,\psi)_0$. As explained in \cite[page 228]{Ikeda1992}, it follows that, up to holomorphic functions without zeros, there exist a unique meromorphic function $L_{\rm PRS}(s,{\rm As}\,\pi_1\otimes\pi_2)$ without zeros satisfying the following conditions:
\begin{itemize}
\item $L_{\rm PRS}(s,{\rm As}\,\pi_1\otimes\pi_2)^{-1}Z(f^{(s)},W_1,W_2)$ is holomorphic for any good section $f^{(s)}$ and $W_1 \in \mathscr{W}(\pi_1,\psi_E)_0$, $W_2 \in \mathscr{W}(\pi_2,\psi)_0$.
\item For each $s_0\in\C$, there exist a good section $f^{(s)}$ and $W_1 \in \mathscr{W}(\pi_1,\psi_E)_0$, $W_2 \in \mathscr{W}(\pi_2,\psi)_0$ such that $L_{\rm PRS}(s,{\rm As}\,\pi_1\otimes\pi_2)^{-1}Z(f^{(s)},W_1,W_2)$ is non-zero at $s=s_0.$
\end{itemize}
Put 
\begin{align*}
\varepsilon_{\rm PSR}(s,{\rm As}\,\pi_1\otimes\pi_2,\psi,\xi)&=\gamma_{\rm PSR}(s,{\rm As}\,\pi_1\otimes\pi_2,\psi,\xi)L_{\rm PSR}(s,{\rm As}\,\pi_1\otimes\pi_2)L_{\rm PSR}(1-s,{\rm As}\,\pi_1^{\vee}\otimes\pi_2^{\vee})^{-1}.
\end{align*}
It is well-defined up to holomorphic function without zeros. By the properties characterizing the $L$-factors and (\ref{E:functional equation for twisted Asai}), $\varepsilon_{\rm PSR}(s,{\rm As}\,\pi_1\otimes\pi_2,\psi,\xi)$ is a holomorphic function without zeros.

Following the idea in the proof of \cite[Theorem 3]{Ikeda1989}, we prove the following proposition.
\begin{thm}\label{T:twisted Asai gamma factor}
Assume $\pi_2$ is a subquotient of a principal series representation ${\rm Ind}_{B(F)}^{\GL_2(F)}(\mu_2|\mbox{ }|_F^{v_2} , \nu_2|\mbox{ }|_F^{-v_2})$. Then 
\begin{align*}
\gamma_{\rm PSR}(s,{\rm As}\,\pi_1 \otimes \pi_2,\psi,\xi) &= \omega(4\xi^4)^{-1}|4\xi^4|_F^{-2s+1}\gamma_{\rm RS}(s+v_2,{\rm As}\,\pi_1\otimes \mu_2 ,\psi,\xi ) \gamma_{\rm RS}(s-v_2,{\rm As}\,\pi_1\otimes \nu_2 ,\psi,\xi )\\
&=\omega(4\xi^2)^{-1}|4\xi^2|_F^{-2s+1}\omega_{E/F}(-1)\gamma_{\Gal}(s,{\rm As}\,\pi_1 \otimes \pi_2,\psi).
\end{align*}
Here $\omega_{E/F}$ is the quadratic character of $F^{\times}$ associated to $E/F$ by local class field theory.
\end{thm}

If $\pi_1$ and $\pi_2$ are both unitary, let $\Lambda({\rm As}\,\pi_1\otimes\pi_2)$ be the non-negative real number defined in \cite[page 228]{Ikeda1992}.
\begin{corollary}\label{C:local factors for twisted Asai}
Assume $\pi_1$ and $\pi_2$ are both unitary and $\Lambda({\rm As}\,\pi_1\otimes\pi_2)<1/2$. If $\pi_2$ is a subquotient of a principal series representation, then
\begin{align*}
L_{\rm PRS}(s,{\rm As}\,\pi_1\otimes\pi_2) &= L_{\rm Gal} (s,{\rm As}\,\pi_1\otimes\pi_2),\\
\varepsilon_{\rm PRS} (s,{\rm As}\,\pi_1\otimes\pi_2,\psi,\xi) &= \omega(4\xi^2)^{-1}|4\xi^2|_F^{-2s+1}\omega_{E/F}(-1)\varepsilon_{\rm Gal} (s,{\rm As}\,\pi_1\otimes\pi_2,\psi).
\end{align*}
\end{corollary}

\begin{proof}
By \cite[Lemma 2.1]{Ikeda1992}, for any good section $f^{(s)}$ and $W_1 \in \mathscr{W}(\pi_1,\psi_E)_0$, $W_2 \in \mathscr{W}(\pi_2,\psi)_0$, the local zeta integral defining $Z(f^{(s)},W_1,W_2)$ is absolutely converge for ${\rm Re}(s)>\Lambda({\rm As}\,\pi_1\otimes\pi_2)$. Therefore, $L_{\rm PRS}(s,{\rm As}\,\pi_1\otimes\pi_2)$ is holomorphic for ${\rm Re}(s)>\Lambda({\rm As}\,\pi_1\otimes\pi_2)$. Similarly, $L_{\rm PRS}(1-s,{\rm As}\,\pi_1^{\vee}\otimes\pi_2^{\vee})$ is holomorphic for ${\rm Re}(s)<1-\Lambda({\rm As}\,\pi_1\otimes\pi_2)$. The assumption $\Lambda({\rm As}\,\pi_1\otimes\pi_2)<1/2$ then implies that $L_{\rm PRS}(s,{\rm As}\,\pi_1\otimes\pi_2)$ and $L_{\rm PRS}(1-s,{\rm As}\,\pi_1^{\vee}\otimes\pi_2^{\vee})$ have no common pole. On the other hand, it can be verified directly that $L_{\rm Gal}(s,{\rm As}\,\pi_1\otimes\pi_2)$ and $L_{\rm Gal}(1-s,{\rm As}\,\pi_1^{\vee}\otimes\pi_2^{\vee})$ have no poles for ${\rm Re}(s)>\Lambda({\rm As}\,\pi_1\otimes\pi_2)$ and ${\rm Re}(s)<1-\Lambda({\rm As}\,\pi_1\otimes\pi_2)$, respectively. The assertion then follows from Theorem \ref{T:twisted Asai gamma factor}.
\end{proof}

\begin{rmk}
If $E=F\times F$, then Corollary \ref{C:local factors for twisted Asai} holds for any $\pi_1$ and $\pi_2$ by the results of \cite{Rama2000}.

Assume $F$ is non-archimedean and $E$ is a field. When both $\pi_1$ and $\pi_2$ are supercuspidal representations, Corollary \ref{C:local factors for twisted Asai} can be proved by using a global-to-local argument and the global functoriality of the Asai transfer to $\GL_4$ (Cf. The proof of Corollary \ref{C:dichotomy of trilinear forms}). When $\pi_1$ is not supercuspidal and $\pi_2$ is supercuspidal, so far we are not able to prove Corollary \ref{C:local factors for twisted Asai} in this case. Note that the last mentioned case was assumed to be true in the proof of \cite[Lemma 7.3]{Rama2002}.

\end{rmk}

\begin{corollary}\label{C:dichotomy of trilinear forms}
Assume $\omega=1$. Then ${\rm Hom}_{\GL_2(F)}(\pi_1\otimes \pi_2,\C) \neq 0$ if and only if $$\omega_{E/F}(-1)\varepsilon_{\Gal}\left(\frac{1}{2},{\rm As}\,\pi_1\otimes \pi_2\right)=1.$$ 
Here $\omega_{E/F}$ is the quadratic character of $F^{\times}$ associated with $E/F$ by local class field theory.
\end{corollary}

\begin{proof}
Except when both $\pi_1$ and $\pi_2$ are supercuspidal, the result is proved in \cite{Pra92}. Therefore, we assume both $\pi_1$ and $\pi_2$ are supercuspidal. 

Note that 
\begin{align*}
L_{\Gal}(s,{\rm As}\,\pi_1\otimes\pi_2)&=L_{\rm PSR}(s,{\rm As}\,\pi_1\otimes\pi_2)=1, \\
L_{\Gal}(s,{\rm As}\,\pi_1^{\vee}\otimes\pi_2^{\vee})&=L_{\rm PSR}(s,{\rm As}\,\pi_1^{\vee}\otimes\pi_2^{\vee})=1.
\end{align*}
 By \cite[Theorem 1.2]{WTG2008}, it suffices to prove
\begin{align*}
\gamma_{\rm PRS} (s,{\rm As}\,\pi_1\otimes\pi_2,\psi,\xi) &= |4\xi^2|_{F}^{-2s+1}\omega_{E/F}(-1)\gamma_{\rm Gal} (s,{\rm As}\,\pi_1\otimes\pi_2,\psi).
\end{align*}
Note that the definition of $\varepsilon$-factor in \cite{WTG2008} is different from the one defined in this article by $\omega_{E/F}(-1)$ (Cf. \cite[p.\,13]{WTG2008}). 

Let ${\bf E}/\bfb{F}$ be a quadratic extension of number fields such that there exist a finite place $v_0$ of $\bf F$ such that ${\bf E}_{v_0}=E$ and ${\bf F}_{v_0}=F$. Fix a non-trivial additive character ${\bm \psi}$ of $\A_{\bf F}/{\bf F}$ and an element $\bm{ \xi} \in {\bf E}^{\times}$ such that ${\rm tr}_{{\bf E}/{\bf F}}({\bm \xi})=0$. By \cite[Proposition 5.1]{Sha90}, there exist irreducible cuspidal automorphic representations $\bm \pi_1$ and $\bm \pi_2$ of $\GL_2(\A_{\bf E})$ and $\GL_2(\A_{\bf F})$ respectively such that
\begin{itemize}
\item ${\bm \pi}_{1,v_0}=\pi_1$,\, ${\bm \pi}_{2,v_0}=\pi_2$
\item ${\bm \pi}_{1,v}$ and ${\bm \pi}_{2,v}$ are spherical for any finite place $v \neq v_0$.
\end{itemize}
Put ${\omega} = {\omega_{{\bm \pi}_1}} \vert_{\A_{\bf F}^{\times}}\cdot \omega_{{\bm \pi}_2}$. By Theorem \ref{T:twisted Asai gamma factor}, for each place $v\neq v_0$ of $F$, we have
\begin{align}\label{E:4.3.4}
\gamma_{\rm PSR}(s,{\rm As}\,{\bm \pi}_{1,v} \otimes {\bm \pi}_{2,v},{\bm \psi}_v,{\bm \xi}) &={\omega}_v(4{\bm \xi}^2)^{-1}|4{\bm \xi}^2|_{{\bf F}_v}^{-2s+1}\omega_{E_v/F_v}(-1)\gamma_{\Gal}(s,{\rm As}\,{\bm \pi}_{1,v} \otimes {\bm \pi}_{2,v},{\bm \psi}_v).
\end{align}
On the other hand, by \cite[Theorem 6.7]{Kri03}, the irreducible admissible representation ${\rm As}\,{\bm \pi}_1 = \otimes_v{\rm As}\,{\bm \pi}_{1,v}$ is an isobaric automorphic representation of $\GL_4(\A_{\bm F})$. Since ${\rm As}\,{\bm \pi}_1$ is isobaric, it follows from the global functional equation for Rankin-Selberg $L$-functions that the global automorphic $L$-function $$L_{\rm Gal}(s,{\rm As}\,{\bm \pi}_1\otimes{\bm \pi}_2) := \prod_{v}L_{\rm Gal}(s,{\rm As}\,{\bm \pi}_{1,v}\otimes{\bm \pi}_2)$$ has meromorphic continuation to $s \in \C$ and satisfies the functional equation 
\begin{align}\label{E:4.3.5}
L_{\rm Gal}(s,{\rm As}\,{\bm \pi}_1\otimes{\bm \pi}_2) = \varepsilon_{\rm Gal}(s,{\rm As}\,{\bm \pi}_1\otimes{\bm \pi}_2) L_{\rm Gal}(1-s,{\rm As}\,{\bm \pi}_1^{\vee}\otimes{\bm \pi}_2^{\vee}).
\end{align}
The assertion then follows from (\ref{E:dependence on psi 2}), (\ref{E:4.3.4}), (\ref{E:4.3.5}), and the global functional equation for $L_{\rm PSR}(s,{\rm As}\,{\bm \pi}_1\otimes{\bm \pi}_2)$. This completes the proof.

\end{proof}

Now we switch to global setting. Let $E/F$ be quadratic extension of number fields. Fix a non-trivial additive character $\psi$ of $\A_F / F$ and $\xi \in E^{\times}$ with ${\rm tr}_{E/F}(\xi)=0$. Let $\pi_1$ and $\pi_2$ be irreducible cuspidal automorphic representations of $\GL_2(\A_E)$ and $\GL_2(\A_F)$, respectively. Let $\omega_1$ and $\omega_2$ be the central characters of $\pi_1$ and $\pi_2$, respectively. Put $\omega=\omega_1\vert_{\A_F^{\times}}\cdot\omega_2$. Define
\begin{align*}
L_{\rm PSR}(s,{\rm As}\,\pi_1\otimes\pi_2) &:= \prod_{v}L_{\rm PSR}(s,{\rm As}\,\pi_{1,v}\otimes\pi_{2,v}),\\
\varepsilon_{\rm PSR}(s,{\rm As}\,\pi_{1}\otimes\pi_{2}) &:= \prod_{v}\varepsilon_{\rm PSR}(s,{\rm As}\,\pi_{1,v}\otimes\pi_{2,v},\psi_v,\xi).
\end{align*}
In \cite[\S 5]{PSR1987} and \cite[\S 6]{Ikeda1989}, it is proved that $L_{\rm PSR}(s,{\rm As}\,\pi_1\otimes\pi_2)$ can be meromorphically continued to $s \in \C$ and satisfies the functional equation
\begin{align}\label{E:global functional equation for twisted Asai}
L_{\rm PSR}(s,{\rm As}\,\pi_1\otimes\pi_2) = \varepsilon_{\rm PSR}(s,{\rm As}\,\pi_{1}\otimes\pi_{2})L_{\rm PSR}(1-s,{\rm As}\,\pi_1^{\vee}\otimes\pi_2^{\vee}).
\end{align}

\begin{thm} \label{T:twisted Asai factor}
For each place $v$ of $F$, we have
\begin{align*}
L_{\rm PRS}(s,{\rm As}\,\pi_{1,v}\otimes\pi_{2,v}) &= L_{\rm Gal} (s,{\rm As}\,\pi_{1,v}\otimes\pi_{2,v}),\\
\varepsilon_{\rm PRS} (s,{\rm As}\,\pi_{1,v}\otimes\pi_{2,v},\psi_v,\xi) &= \omega_v(4\xi^2)^{-1}|4\xi^2|_{F_v}^{-2s+1}\omega_{E_v/F_v}(-1)\varepsilon_{\rm Gal} (s,{\rm As}\,\pi_{1,v}\otimes\pi_{2,v},\psi_v).
\end{align*}
\end{thm}

\begin{proof}
We may assume both $\pi_1$ and $\pi_2$ are unitary. By the result of \cite{KM2002}, for any place $v$ of F, we have $$\Lambda({\rm As}\,\pi_{1,v}\otimes\pi_{2,v})<1/2.$$
By \cite[Theorem 6.7]{Kri03}, the irreducible admissible representation ${\rm As}\,\pi_1 = \otimes_v{\rm As}\,\pi_{1,v}$ is an isobaric automorphic representation of $\GL_4(\A_F)$.

Fix a place $v_0$ of $F$. If $\pi_{2,v_0}$ is a subquotient of a principal series representation, then the theorem follows from Corollary \ref{C:local factors for twisted Asai}. Assume $v_0$ is a finite place and $\pi_{2,v_0}$ is a supercuspidal representation. By \cite[Proposition 5.1]{Sha90}, there exist an irreducible unitary cuspidal automorphic representation $\sigma$ of $\GL_2(\A_F)$ such that
\begin{itemize}
\item $\sigma_{v_0}=\pi_{2,v_0}$.
\item $\sigma_{v}$ is spherical for any finite place $v \neq v_0$.
\end{itemize}
Let $\omega_2'$ be the central character of $\sigma$. Put $\omega'=\omega_1\vert_{F^{\times}}\cdot\omega_2'.$ By Theorem \ref{T:twisted Asai gamma factor}, for each place $v\neq v_0$ of $F$, we have
\begin{align}\label{E1}
\gamma_{\rm PSR}(s,{\rm As}\,\pi_{1,v} \otimes \sigma_v,\psi_v,\xi) &=\omega_v'(4\xi^2)^{-1}|4\xi^2|_{F_v}^{-2s+1}\omega_{E_v/F_v}(-1)\gamma_{\Gal}(s,{\rm As}\,\pi_{1,v} \otimes \sigma_v,\psi_v).
\end{align}
On the other hand, for each finite place $v$ of $F$, the gamma factor $\gamma_{\Gal}(s,{\rm As}\,\pi_{1,v} \otimes \sigma_v,\psi_v)$ is equal to the gamma factor of defined by the Rankin-Selberg local zeta integrals (Cf. \cite{BH1999} and \cite{JPSS83}). Since ${\rm As}\,\pi_1$ is isobaric, it follows from the global functional equation for Rankin-Selberg $L$-functions that the global automorphic $L$-function $L_{\rm Gal}(s,{\rm As}\,\pi_1\otimes\sigma) = \prod_{v}L_{\rm Gal}(s,{\rm As}\,\pi_{1,v}\otimes\sigma_v)$ has meromorphic continuation to $s \in \C$ and satisfies the functional equation 
\begin{align}\label{E2}
L_{\rm Gal}(s,{\rm As}\,\pi_1\otimes\sigma) = \varepsilon_{\rm Gal}(s,{\rm As}\,\pi_1\otimes\sigma) L_{\rm Gal}(1-s,{\rm As}\,\pi_1^{\vee}\otimes\sigma^{\vee}).
\end{align}
By (\ref{E:global functional equation for twisted Asai}) with $\sigma$ in place of $\pi_2$, (\ref{E1}), and (\ref{E2}), we have
\begin{align*}
\gamma_{\rm PSR} (s,{\rm As}\,\pi_{1,v_0}\otimes\sigma_{v_0},\psi_{v_0},\xi) &= \omega_{v_0}'(4\xi^2)^{-1}|4\xi^2|_{F_{v_0}}^{-2s+1}\omega_{E_{v_0}/F_{v_0}}(-1)\gamma_{\rm Gal} (s,{\rm As}\,\pi_{1,v_0}\otimes\sigma_{v_0},\psi_{v_0}).
\end{align*}
Since $\Lambda({\rm As}\,\pi_{1,v_0}\otimes\sigma_{v_0})<1/2$, following the same argument as in the proof of Corollary \ref{C:local factors for twisted Asai}, we conclude that
\begin{align*}
L_{\rm PRS}(s,{\rm As}\,\pi_{1,v_0}\otimes\sigma_{v_0}) &= L_{\rm Gal} (s,{\rm As}\,\pi_{1,v_0}\otimes\sigma_{v_0}),\\
\varepsilon_{\rm PRS} (s,{\rm As}\,\pi_{1,v_0}\otimes\sigma_{v_0},\psi_{v_0},\xi) &= \omega_{v_0}'(4\xi^2)^{-1}|4\xi^2|_{F_{v_0}}^{-2s+1}\omega_{E_{v_0}/F_{v_0}}(-1)\varepsilon_{\rm Gal} (s,{\rm As}\,\pi_{1,v_0}\otimes\sigma_{v_0},\psi_{v_0}).
\end{align*}
This completes the proof.

\end{proof}

\subsection{Proof of Theorem \ref{T:twisted Asai gamma factor}}\label{S:4.4}
\begin{proof}
The case $E=F\times F$ is proved in \cite[Theorem 3]{Ikeda1989}. We assume $E$ is a field. For brevity, we denote $\iota=\iota_{\xi}$, $\iota_2=\iota_{2,\xi}$, and $a(\nu) = \bp \nu & 0 \\ 0 & 1\ep \in\GL_2$. 

Let $f^{(s)}$ be a good section of $I(\omega,s)$, and $W_1 \in \mathscr{W}(\pi_1,\psi_E)$, $W_2 \in \mathscr{W}(\pi_2,\psi)$. Fix a holomorphic section $h_2^{(u_2)}$ of $\mathcal{B}(\mu_2|\mbox{ }|_F^{u_2},\nu_2|\mbox{ }|_F^{-u_2})_0$ such that $W_2=W_{\psi,h_2^{(v_2)}}$. Denote $W_2^{(u_2)}=W_{\psi,h_2^{(u_2)}}$. 
\\If $\pi_1$ is a subquotient of ${\rm Ind}_{B(E)}^{\GL_2(E)}(\mu_1|\mbox{ }|_F^{v_1},\nu_1|\mbox{ }|_F^{-v_1})$, fix a holomorphic section $h_1^{(u_1)}$ of $\mathcal{B}(\mu_1|\mbox{ }|_F^{u_1},\nu_1|\mbox{ }|_F^{-u_1})_0$ such that $W_1=W_{\psi_E,h_1^{(v_1)}}$. Denote $\pi_1^{(u_1)}={\rm Ind}_{B(E)}^{\GL_2(E)}(\mu_1|\mbox{ }|_F^{u_1},\nu_1|\mbox{ }|_F^{-u_1})$ and $W_1^{(u_1)}=W_{\psi_E,h_1^{(u_1)}}$. 
\\If $\pi_1$ is supercuspidal, denote $\pi_1^{(u_1)}=\pi_1$ and $W_1^{(u_1)}=W_1$.
\\Let $W_{1,\xi}^{(u_1)} \in \mathscr{W}(\pi_1^{(u_1)},\psi_{\xi})_0$ defined by
$$W_{1,\xi}^{(u_1)}(g)=W_{1}^{(u_1)}(a(\xi)ga(\xi)^{-1}).$$
After twisting $\pi_1$ and $\pi_2$ by unramified characters if necessary, we may assume $\mu_i$, $\nu_i$ are unitary for $i=1,2$. 

Let $D$, $D^{\vee}$, and $\overline{D}$ be subsets of $\C^3$ defined by
\begin{align*}
D&= \left\{(s,u_1,u_2)\in\C^3 \mbox{ }\vert \mbox{ } {\rm Re}(s)>2|{\rm Re}(u_1)|+|{\rm Re}(u_2)| \   \right\},\\
D^{\vee}&= \left\{(s,u_1,u_2)\in\C^3 \mbox{ }\vert \mbox{ } {\rm Re}(1-s)>2|{\rm Re}(u_1)|+|{\rm Re}(u_2)| \   \right\},\\
\mathscr{D} & =\left\{(s,u_1,u_2)\in\C^3 \mbox{ }\left\vert \mbox{ } {\rm Re}(u_2)>2|{\rm Re}(u_1)|+\left |{\rm Re}(s)-\frac{1}{2}  \right |-\frac{1}{2} \right.   \right\}.
\end{align*}
By \cite[Lemma 2.1]{Ikeda1992}, the integrals defining $\Psi(f^{(s)},W_1^{(u_1)},W_2^{(u_2)})$ and $\Psi(I_{w_3}^*f^{(s)},W_1^{(u_1)},W_2^{(u_2)})$ are absolutely converge for $(s,u_1,u_2) \in D$ and $(s,u_1,u_2) \in D^{\vee}$,
respectively. If furthermore ${\rm Re}(u_2)>0$, a similar estimation shows that the function 
\begin{align*}
\G(F) &\longrightarrow \C\\
    g=(g_1,g_2) & \longmapsto f^{(s)}(\eta \iota(g))W_1^{(u_1)}(g^{(1)})h_2^{(u_2)}(w_1g^{(2)}) 
\end{align*}
define an integrable function on $F^{\times}{\bf U}_0^{'}(F) \backslash \G(F)$. Therefore, for $(s,u_1,u_2) \in D$ and ${\rm Re}(u_2)>0$, we have
\begin{align*}
&\Psi(f^{(s)},W_1^{(u_1)},W_2^{(u_2)}) \\
&= \int_{F^{\times}{\bf U}_0(F) \backslash {\bf G}(F)}f^{(s)}(\eta\iota(g))W_1^{(u_1)}(g^{(1)})\left (\int_{F}h_2^{(u_2)} ( w_1u(x)g^{(2)} )  \psi(-x)dx\right)dg\\
&=\int_{F^{\times}{\bf U}_0^{'}(F) \backslash \G(F)}f^{(s)}(\eta \iota(g))W_1^{(u_1)}(g^{(1)})h_2^{(u_2)}(w_1g^{(2)})dg\\
&=\int_{\SL_2(F)\backslash \SL_2(E)}\int_{F^{\times}U(F) \backslash \GL_2(F)}\int_{\SL_2(F)} f^{(s)}(\eta \iota(a(\xi)gg_1a(\xi)^{-1},gg_2)    )W_{1,\xi}^{(u_1)}(gg_1)h_2^{(u_2)}(w_1gg_2)dg_2dgdg_1.
\end{align*}
We have similar identity for $\Psi(I_{w_3}^*f^{(s)},W_1^{(u_1)},W_2^{(u_2)})$ when $(s,u_1,u_2) \in D^{\vee}$ and ${\rm Re}(u_2)>0$.
\\Let $F_{f,h_2}^{(s,u_2)}$ and $F_{I_{w_3}^*f,h_2}^{(s,u_2)}$ be functions on $\GL_2(F)\times \SL_2(E)$ defined by the integrals
\begin{align*}
\begin{split}
F_{f,h_2}^{(s,u_2)}(g,g_1) &= \int_{\SL_2(F)}f^{(s)}(\eta \iota(a(\xi)gg_1a(\xi)^{-1},gg_2)    )h_2^{(u_2)}(w_1gg_2)dg_2,\\
F_{I_{w_3}^*f,h_2}^{(s,u_2)}(g,g_1) &= \int_{\SL_2(F)}I_{w_3}^*f^{(s)}(\eta \iota(a(\xi)gg_1a(\xi)^{-1},gg_2)    )h_2^{(u_2)}(w_1gg_2)dg_2,
\end{split}
\end{align*}
which are absolutely converge for 
\begin{align*}
{\rm Re}(s)>0, {\rm Re}(s-u_2)>0\mbox{, and }{\rm Re}(s+u_2)+1>0,\\
{\rm Re}(1-s)>0, {\rm Re}(1-s-u_2)>0\mbox{, and }{\rm Re}(1-s+u_2)+1>0,
\end{align*}
respectively. As functions of $g \in \GL_2(F)$ and $(s,u_2)\in\C^2$, 
\begin{align*}
\begin{split}
F_{f,h_2}^{(s,u_2)}(-,g_1) &\in \mathcal{B}(\nu_2|\mbox{ }|_F^{s-u_2-1/2},\omega_1^{-1}\nu_2^{-1}|\mbox{ }|_F^{-s+u_2+1/2}),\\
F_{I_{w_3}^*f,h_2}^{(s,u_2)}(-,g_1) &\in \mathcal{B}(\omega_1^{-1}\mu_2^{-1}|\mbox{ }|_F^{-s-u_2+1/2},\mu_2|\mbox{ }|_F^{s+u_2-1/2})
\end{split}
\end{align*}
are meromorphic sections. By \cite[Appendix, Theorem]{F93} and Proposition A, we have
\begin{align*}
\begin{split}
\Psi(f^{(s)},W_1^{(u_1)},W_2^{(u_2)}) &= \gamma(2s-2u_2-1,\omega_1\nu_2^2,\psi)\gamma_{\rm RS}(s-u_2,{\rm As}\,\pi_1^{(u_1)} \otimes \nu_2,\psi,\xi)^{-1}\\
&\times\int_{\SL_2(F)\backslash \SL_2(E)}Z(MF_{f,h_2}^{(s,u_2)},\rho(g_1)W_{1,\xi}^{(u_1)})dg_1,\\
\Psi(I_{w_3}^*f^{(s)},W_1^{(u_1)},W_2^{(u_2)}) &= \gamma(-2s-2u_2+1,\omega_1^{-1}\mu_2^{-2},\psi)\gamma_{\rm RS}(1-s-u_2,{\rm As}\,\pi_1^{(u_1)} \otimes \omega_1^{-1}\mu_2^{-1},\psi,\xi)^{-1}\\
&\times\int_{\SL_2(F)\backslash \SL_2(E)}{Z}(MF_{I_{w_3}^*f,h_2}^{(s,u_2)},\rho(g_1)W_{1,\xi}^{(u_1)})dg_1\\
\end{split}
\end{align*}
for 
\begin{align*}
{\rm Re}(s)>2|{\rm Re}(u_1)|+|{\rm Re}(u_2)|\mbox{ and }{\rm Re}(u_2)>0,\\
{\rm Re}(1-s)>2|{\rm Re}(u_1)|+|{\rm Re}(u_2)|\mbox{ and }{\rm Re}(u_2)>0,
\end{align*}
respectively. Here $M$ denote the intertwining operator defined in \S \ref{SS:the functional equation}.
 Note that the local zeta integrals defining ${Z}(MF_{f,h_2}^{(s,u_2)},\rho(g_1)W_{1,\xi}^{(u_1)})$ and ${Z}(MF_{I_{w_3}^*f,h_2}^{(s,u_2)},\rho(g_1)W_{1,\xi}^{(u_1)})$ are absolutely converge for $${\rm Re}(1-s+u_2)>2|{\rm Re}(u_1)|\mbox{ and }{\rm Re}(s+u_2)>2|{\rm Re}(u_1)|,$$
respectively. 
Therefore, we conclude that
\begin{align}\label{E:1}
\begin{split}
\Psi(f^{(s)},W_1^{(u_1)},W_2^{(u_2)}) &=\gamma(2s-2u_2-1,\omega_1\nu_2^2,\psi)\gamma_{\rm RS}(s-u_2,{\rm As}\,\pi_1^{(u_1)} \otimes \nu_2,\psi,\xi)^{-1}\\
&\times\int_{\SL_2(F)\backslash \SL_2(E)}\left ( \int_{F^{\times}U(F) \backslash \GL_2(F)}MF_{f,h_2}^{(s,u_2)}(g,g_1)W_{1,\xi}^{(u_1)}(gg_1)dg \right )dg_1,\\
\Psi(I_{w_3}^*f^{(s)},W_1^{(u_1)},W_2^{(u_2)}) &=\gamma(-2s-2u_2+1,\omega_1^{-1}\mu_2^{-2},\psi)\gamma_{\rm RS}(1-s-u_2,{\rm As}\,\pi_1^{(u_1)} \otimes \omega_1^{-1}\mu_2^{-1},\psi,\xi)^{-1}\\
&\times\int_{\SL_2(F)\backslash \SL_2(E)}\left ( \int_{F^{\times}U(F) \backslash \GL_2(F)}MF_{I_{w_3}^*f,h_2}^{(s,u_2)}(g,g_1)W_{1,\xi}^{(u_1)}(gg_1)dg \right )dg_1
\end{split}
\end{align}
for 
\begin{align*}
{\rm Re}(u_2)>0\mbox{ and }2|{\rm Re}(u_1)|< {\rm Re}(s-u_2) <1-2 |{\rm Re}(u_1)|,\\
{\rm Re}(u_2)>0\mbox{ and }2|{\rm Re}(u_1)|< {\rm Re}(1-s-u_2) <1-2 |{\rm Re}(u_1)|,
\end{align*}
respectively. 

Define functions $\mathscr{F}_{f,h_2}^{(s,u_2)}$ and $\mathscr{F}_{I_{w_3}^*f,h_2}^{(s,u_2)}$ on $\GL_2^{\circ}(F)$ by 
\begin{align}\label{E:2}
\begin{split}
\mathscr{F}_{f,h_2}^{(s,u_2)}(g)&=MF_{f,h_2}^{(s,u_2)} (a(\nu(g)),a(\nu(g))^{-1}a(\xi)^{-1}ga(\xi)),\\
\mathscr{F}_{I_{w_3}^*f,h_2}^{(s,u_2)}(g)&=MF_{I_{w_3}^*f,h_2}^{(s,u_2)} (a(\nu(g)),a(\nu(g))^{-1}a(\xi)^{-1}ga(\xi)),
\end{split}
\end{align}
for $g \in \GL_2^{\circ}(F)$. Put 
\begin{align*}
\overline{\mathscr{F}}_{f,h_2}^{(s,u_2)}(g) &= \int_{F}{\mathscr{F}}_{f,h_2}^{(s,u_2)}(u(x)g)\psi(2x)dx,\\
\overline{\mathscr{F}}_{I_{w_3}^*f,h_2}^{(s,u_2)}(g) &= \int_{F}{\mathscr{F}}_{I_{w_3}^*f,h_2}^{(s,u_2)}(u(x)g)\psi(2x)dx,
\end{align*}
for $g \in \GL_2^{\circ}(F).$
By (\ref{E:1}), formally we have
\begin{align}\label{E:first step}
\begin{split}
\Psi(f^{(s)},W_1^{(u_1)},W_2^{(u_2)}) &=\gamma(2s-2u_2-1,\omega_1\nu_2^2,\psi)\gamma_{\rm RS}(s-u_2,{\rm As}\,\pi_1^{(u_1)} \otimes \nu_2,\psi,\xi)^{-1}\\
&\times\int_{F^{\times}U(E) \backslash \GL_2^{\circ}(F)}\overline{\mathscr F}_{f,h_2}^{(s,u_2)}(g)W_1^{(u_1)}(g)dg,\\
\Psi(I_{w_3}^*f^{(s)},W_1^{(u_1)},W_2^{(u_2)}) &=\gamma(-2s-2u_2+1,\omega_1^{-1}\mu_2^{-2},\psi)\gamma_{\rm RS}(1-s-u_2,{\rm As}\,\pi_1^{(u_1)} \otimes \omega_1^{-1}\mu_2^{-1},\psi,\xi)^{-1}\\
&\times\int_{F^{\times}U(E) \backslash \GL_2^{\circ}(F)}\overline{\mathscr F}_{I_{w_3}^*f,h_2}^{(s,u_2)}(g)W_1^{(u_1)}(g)dg.
\end{split}
\end{align}
In the following we shall prove that $\overline{\mathscr F}_{f,h_2}^{(s,u_2)}$ and $\overline{\mathscr F}_{I_{w_3}^*f,h_2}^{(s,u_2)}$ can be meromorphically continued to $(s,u_2)\in\C^2$ and satisfying functional equation (\ref{E:key identity 2}). Then we will prove that the integrals in (\ref{E:first step}) are both absolutely converge for $(s,u_1,u_2) \in \mathscr{D}$.

Fix a compact subset $\Omega \subseteq \GL_2^{\circ}(F)$ consisting of a set of representatives of the image of the map
\begin{align*}
F^{\times} \backslash E^{\times} \times K^{\circ} &\longrightarrow F^{\times} \backslash \GL_2^{\circ}(F)\\
(F^{\times}a,k) &\longmapsto F^{\times}\bp a & 0 \\ 0 & a^{-1}\ep k.
\end{align*}
Note that when $F=\R$, we can and will assume $\iota_2(\Omega) \subseteq K_2$.

Define functions ${G}_{f,h_2}^{(s,u_2)}$ and ${G}_{I_{w_3}^*f,h_2}^{(s,u_2)}$ on $\GSp_2(F)$ by 
\begin{align*}
G_{f,h_2}^{(s,u_2)}(g) &= \int_{F}\int_{\SL_2(F)} f^{(s)}  \left(u_-\left( \bp 0&0&0 \\ 0&0&x \\ 0&x&0\ep \right )\eta'\iota(1,g_2)\gamma(g)\right )h_2^{(u_2)}(w_1g_2a(\nu(g)))dg_2dx,\\
G_{I_{w_3}^*f,h_2}^{(s,u_2)}(g) &= \int_{F}\int_{\SL_2(F)} I_{w_3}^*f^{(s)}  \left(u_-\left( \bp 0&0&0 \\ 0&0&x \\ 0&x&0\ep \right )\eta'\iota(1,g_2)\gamma(g)\right )h_2^{(u_2)}(w_1g_2a(\nu(g)))dg_2dx
\end{align*}
Note that the integrals defining $G_{f,h_2}^{(s,u_2)}$ and ${G}_{I_{w_3}^*f,h_2}^{(s,u_2)}$ are absolutely converge for 
\begin{align*}
{\rm Re}(s)>0\mbox{, }{\rm Re}(s-u_2)>\frac{1}{2}\mbox{, and }{\rm Re}(s+u_2)>-1,\\
{\rm Re}(1-s)>0\mbox{, }{\rm Re}(1-s-u_2)>\frac{1}{2}\mbox{, and }{\rm Re}(1-s+u_2)>-1,
\end{align*}
respectively.
By the $K_3$-finiteness of $f^{(s)}$ and (\ref{E:identity for conjugation by eta special form}), we have the following properties:
\begin{itemize}
\item $G_{f,h_2}^{(s,u_2)}$ and ${G}_{I_{w_3}^*f,h_2}^{(s,u_2)}$ are right $K_2$-finite.
\item
For $g \in \GSp_2(F)$ and $p=  \bp a_1 & 0 &*&* \\ * & a_2 & * & * \\ 0 & 0 & \nu a_1^{-1} & * \\ 0 & 0 & 0 & \nu a_2^{-1}     \ep $, we have 
\begin{align*}
G_{f,h_2}^{(s,u_2)}(pg) &= \nu_2\omega^{-1}(\nu)|\nu|_F^{-s-u_2-1}\omega(a_1)|a_1|_F^{2s}\mu_2\nu_2^{-1}(a_2)|a_2|_F^{2u_2+2}G_{f,h_2}^{(s,u_2)}(g),\\
G_{I_{w_3}^*f,h_2}^{(s,u_2)}(pg) &= \nu_2(\nu)|\nu|_F^{-(1-s)-u_2-1}\omega^{-1}(a_1)|a_1|_F^{2(1-s)}\mu_2\nu_2^{-1}(a_2)|a_2|_F^{2u_2+2}G_{I_{w_3}^*f,h_2}^{(s,u_2)}(g)
\end{align*}
\end{itemize}
By (\ref{E:identity for conjugation by eta special form}), 
\begin{align*}
G_{f,h_2}^{(s,u_2)}(g) =& \int_{k \in K_1^1}h_2^{(u_2)}(w_1ka(\nu(g)))\int_{F^{\times}}\int_{F}\int_{F}\omega_1\nu_2^2(a)|a|_F^{2s-2u_2-1}\\
&f^{(s)}\left ( u_- \left ( \bp 0 & 0 & a-1 \\0 & z & x \\ a-1 & x & y \ep \right )\eta'\iota(1,k)\gamma(g)\right )dydxd^{\times}adk,\\
G_{I_{w_3}^*f,h_2}^{(s,u_2)}(g) =& \int_{k \in K_1^1}h_2^{(u_2)}(w_1ka(\nu(g)))\int_{F^{\times}}\int_{F}\int_{F}\omega_1^{-1}\mu_2^{-2}(a)|a|_F^{2(1-s)-2u_2-1}\\
&I_{w_3}^*f^{(s)}\left ( u_- \left ( \bp 0 & 0 & a-1 \\0 & z & x \\ a-1 & x & y \ep \right )\eta'\iota(1,k)\gamma(g)\right )dydxd^{\times}adk.
\end{align*}
Note that for $x,y,a \in F$,
\begin{align*}
u_- \left ( \bp 0 & 0 & a \\0 & 0 & x \\ a & x & y \ep \right ) = \kappa_{(3)}\left ( \bp 1 & 0 \\ y & 1\ep \right )w_{(3)}^{-1}\kappa_{(2)}\left ( \bp 1 & 0 \\ x & 1\ep \right )w_{(2)}^{-1}w_{(1)}^{-1}\kappa_{(1)}\left ( \bp 1 & -a \\ 0 & 1\ep \right )w_{(1)}w_{(2)}w_{(3)},
\end{align*}
and \begin{align*}
I_{w_{(2)}}I_{w_{(3)}}f^{(s)} \in \mathcal{B}(\omega |\mbox{ }|_F^{2s-4},\omega^{-1}|\mbox{ }|_F^{-2s-1/2})_0,\\
I_{w_{(2)}}I_{w_{(3)}}I_{w_3}^*f^{(s)} \in \mathcal{B}(\omega^{-1} |\mbox{ }|_F^{-2s-2},\omega |\mbox{ }|_F^{2s-5/2})_0
\end{align*}
are meromorphic sections. Therefore, by \cite[Lemma 5.1]{Ikeda1989}, we have
\begin{align*}
G_{f,h_2}^{(s,u_2)}(g) =& \gamma (2s-2u_2-1,\omega_1\nu_2^2,\psi)^{-1}\int_{K_1^1}h_2^{(u_2)}(w_1ka(\nu(g)) \\
&Z_{(1)}\left (2u_1+\frac{5}{2},\mu_2\nu_2^{-1},W_{w_{(1)}}I_{w_{(2)}}I_{w_{(3)}}f^{(s)}, w_{(1)}w_{(2)}w_{(3)}u_- \left ( \bp 0&0&-1 \\ 0&0&0 \\ -1&0&0 \ep \right) \eta' \iota(1,k)\gamma(g)\right )dk,\\
G_{I_{w_3}^*f,h_2}^{(s,u_2)}(g) =& \gamma (-2s-2u_2+1,\omega_1^{-1}\mu_2^{-2},\psi)^{-1}\int_{K_1^1}h_2^{(u_2)}(w_1ka(\nu(g)) \\
&Z_{(1)}\left (2u_1+\frac{5}{2},\mu_2\nu_2^{-1},W_{w_{(1)}}I_{w_{(2)}}I_{w_{(3)}}I_{w_3}^*f^{(s)}, w_{(1)}w_{(2)}w_{(3)}u_- \left ( \bp 0&0&-1 \\ 0&0&0 \\ -1&0&0 \ep \right) \eta' \iota(1,k)\gamma(g)\right )dk.
\end{align*}
In particular, both $G_{f,h_2}^{(s,u_2)}$ and $G_{I_{w_3}^*f,h_2}^{(s,u_2)}$ can be meromorphically continued to $(s,u_2)\in\C^2.$
\\By an argument similar to \cite[page 209-212]{Ikeda1989}, we have
\begin{align}
\begin{split}
\mathscr{F}_{f,h_2}^{(s,u_2)}(g) &= \omega_1(-1)|2\xi^2|_FG_{f,h_2}^{(s,u_2)}\left ( \bp 0&-1&0&0 \\ 0&-1&-1&0 \\ -1&0&0&-1 \\ 1&0&0&0\ep \iota_2 (a(\xi)w_1a(\xi)^{-1}g) \right ),\\
\mathscr{F}_{I_{w_3}^*f,h_2}^{(s,u_2)}(g) &=\omega_1(-1) |2\xi^2|_FG_{I_{w_3}^*f,h_2}^{(s,u_2)}\left ( \bp 0&-1&0&0 \\ 0&-1&-1&0 \\ -1&0&0&-1 \\ 1&0&0&0\ep \iota_2 (a(\xi)w_1a(\xi)^{-1}g) \right ).
\end{split}
\end{align}
By \cite[(5.2.11) and (5.2.12)]{Ikeda1989} and the properties of $G_{f,h_2}^{(s,u_2)}$ and $G_{I_{w_3}^*f,h_2}^{(s,u_2)}$, we deduce that
\begin{align*}
\begin{split}
\left \vert {\mathscr F}_{f,h_2}^{(s,u_2)}(a(\nu)u(x)g) \right \vert &\ll_{s,u_2}  |\nu|_F^{{\rm Re}(u_2-s)+1}N( 2\xi^2,2^{-1}\xi^{-2},\xi^{-2}x)^{2{\rm Re}(s-u_2)-2}N(1,2x)^{-2{\rm Re}(s)},\\
\left \vert {\mathscr F}_{I_{w_3}^*f,h_2}^{(s,u_2)}(a(\nu)u(x)g) \right \vert &\ll_{s,u_2}  |\nu|_F^{{\rm Re}(u_2-1+s)+1}N( 2\xi^2,2^{-1}\xi^{-2},\xi^{-2}x)^{2{\rm Re}(1-s-u_2)-2}N(1,2x)^{-2{\rm Re}(1-s)}
\end{split}
\end{align*}
for $g \in \Omega$, $\nu \in F^{\times}$, and $x \in F$. Therefore, the integrals
\begin{align*}
\int_{F^{\times}U^{\circ}(F) \backslash \GL_2^{\circ}(F)}{\mathscr F}_{f,h_2}^{(s,u_2)}(g)W_1^{(u_1)}(g)dg\mbox{ and }\int_{F^{\times}U^{\circ}(F) \backslash \GL_2^{\circ}(F)}{\mathscr F}_{I_{w_3}^*f,h_2}^{(s,u_2)}(g)W_1^{(u_1)}(g)dg,
\end{align*}
are absolutely converge for 
\begin{align*}
{\rm Re}(s)>\frac{1}{2}\mbox{, }{\rm Re}(u_2)>2|{\rm Re}(u_1)|+{\rm Re}(s)-1,\\
{\rm Re}(1-s)>\frac{1}{2}\mbox{, }{\rm Re}(u_2)>2|{\rm Re}(u_1)|+{\rm Re}(1-s)-1,
\end{align*}
respectively.
\\Define functions $\overline{G}_{f,h_2}^{(s,u_2)}$ and $\overline{G}_{I_{w_3}^*f,h_2}^{(s,u_2)}$ on $\GSp_2(F)$ by
\begin{align*}
\begin{split}
\overline{G}_{f,h_2}^{(s,u_2)}(g) &= \int_{F}{G}_{f,h_2}^{(s,u_2)} \left ( \bp  0&0&1&0\\0&1&0&0\\-1&0&0&0\\0&0&0&1\ep u\left( \bp x &0\\0&0 \ep \right )g\right )\psi(-x)dx,\\
\overline{G}_{I_{w_3}^*f,h_2}^{(s,u_2)}(g) &= \int_{F}{G}_{I_{w_3}^*f,h_2}^{(s,u_2)} \left ( \bp  0&0&1&0\\0&1&0&0\\-1&0&0&0\\0&0&0&1\ep u\left( \bp x &0\\0&0 \ep \right )g\right )\psi(-x)dx.
\end{split}
\end{align*}
By the properties of $G_{f,h_2}^{(s,u_2)}$ and $G_{I_{w_3}^*f,h_2}^{(s,u_2)}$, we have the following properties.
\begin{itemize}
\item $\overline{G}_{f,h_2}^{(s,u_2)}$ is right $K_2$-finite. 
\item For $k_2 \in K_2$, as functions in $g=\bp a & b \\ c&d \ep \in \GL_2(F)$, we have
\begin{align*}
\overline{G}_{f,h_2}^{(s,u_2)}\left( \bp a&0&b&0\\ 0&\det(g)&0&0\\c&0&d&0\\0&0&0&1    \ep k_2\right)&\in \mathscr{W}(\mu_2|\mbox{ }|_F^{s+u_2+1/2},\mu_2\omega^{-1}|\mbox{ }|_F^{-s+u_2+3/2})_0,\\
\overline{G}_{I_{w_3}^*f,h_2}^{(s,u_2)}\left( \bp a&0&b&0\\ 0&\det(g)&0&0\\c&0&d&0\\0&0&0&1    \ep k_2\right)&\in \mathscr{W}(\mu_2\omega^{-1}|\mbox{ }|_F^{-s+u_2+3/2},\mu_2|\mbox{ }|_F^{s+u_2+1/2})_0.\\
\end{align*}
\end{itemize}
By (\ref{E:identity for conjugation by eta special form}),
\begin{align*}
\overline{G}_{f,h_2}^{(s,u_2)}(g) =& \int_{k \in K_1^1}h_2^{(u_2)}(w_1ka(\nu(g)))\int_{F^{\times}}\int_{F}\int_{F}\int_{F}\psi(z)\omega_1\nu_2^2(a)|a|_F^{2s-2u_2-1}\\
&f^{(s)}\left ( u_- \left ( \bp 0 & 0 & a-1 \\0 & z & x \\ a-1 & x & y \ep \right )\eta'\iota(1,k)w_{(1)}\gamma(g)\right )dzdydxd^{\times}adk,\\
\overline{G}_{I_{w_3}^*f,h_2}^{(s,u_2)}(g) =& \int_{k \in K_1^1}h_2^{(u_2)}(w_1ka(\nu(g)))\int_{F^{\times}}\int_{F}\int_{F}\int_{F}\psi(z)\omega_1^{-1}\mu_2^{-2}(a)|a|_F^{2(1-s)-2u_2-1}\\
&I_{w_3}^*f^{(s)}\left ( u_- \left ( \bp 0 & 0 & a-1 \\0 & z & x \\ a-1 & x & y \ep \right )\eta'\iota(1,k)w_{(1)}\gamma(g)\right )dzdydxd^{\times}adk.
\end{align*}
Note that for $x,y,z,a\in F$, 
\begin{align*}
u_- \left ( \bp 0 & 0 & a \\0 & z & x \\ a & x & y \ep \right ) &= \kappa_{(3)}\left ( \bp 1 & 0 \\ y & 1\ep \right )w_{(3)}^{-1}\kappa_{(2)}\left ( \bp 1 & 0 \\ x & 1\ep \right )w_{(2)}^{-1} \kappa_{(3)}\left ( \bp 1 & 0 \\ z & 1\ep \right )  \\
&\times w_{(1)}^{-1}\kappa_{(1)}\left ( u(-a) \right )w_{(1)}w_{(2)}w_{(3)}.
\end{align*}
Therefore, by \cite[Lemma 5.1]{Ikeda1989}, we have
\begin{align*}
\overline{G}_{f,h_2}^{(s,u_2)}(g) =& \gamma (2s-2u_2-1,\omega_1\nu_2^2,\psi)^{-1}\int_{K_1^1}h_2^{(u_2)}(w_1ka(\nu(g)) \\
&Z_{(1)}\left (2u_2+\frac{5}{2},\mu_2\nu_2^{-1},W_{w_{(1)}}W_{w_{(3)}}I_{w_{(2)}}I_{w_{(3)}}f^{(s)}, \begin{array}{lll} & & \\ & & \\ & & \end{array} \right .\\
&\left . \mbox{ }\mbox{ }\mbox{ }\mbox{ }\mbox{ }\mbox{ }w_{(3)}w_{(1)}w_{(2)}w_{(3)}u_- \left ( \bp 0&0&-1 \\ 0&0&0 \\ -1&0&0 \ep \right)  \eta' \iota(1,k)w_{(1)}\gamma(g)\right )dk,\\
\overline{G}_{I_{w_3}^*f,h_2}^{(s,u_2)}(g) =& \gamma (-2s-2u_2+1,\omega_1^{-1}\mu_2^{-2},\psi)^{-1}\int_{K_1^1}h_2^{(u_2)}(w_1ka(\nu(g)) \\
&Z_{(1)}\left (2u_2+\frac{5}{2},\mu_2\nu_2^{-1},W_{w_{(1)}}W_{w_{(3)}}I_{w_{(2)}}I_{w_{(3)}}I_{w_3}^*f^{(s)}, \begin{array}{lll} & & \\ & & \\ & & \end{array} \right .\\
&\left . \mbox{ }\mbox{ }\mbox{ }\mbox{ }\mbox{ }\mbox{ }w_{(3)}w_{(1)}w_{(2)}w_{(3)}u_- \left ( \bp 0&0&-1 \\ 0&0&0 \\ -1&0&0 \ep \right)  \eta' \iota(1,k)w_{(1)}\gamma(g)\right )dk.
\end{align*}
In particular, both $\overline{G}_{f,h_2}^{(s,u_2)}$ and $\overline{G}_{I_{w_3}^*f,h_2}^{(s,u_2)}$ can be meromorphically continued to $(s,u_2) \in \C^2.$ By (\ref{E:composite of intertwining operators}) and proceeding as in \cite[page 215]{Ikeda1989}, we have
$$W_{w_{(1)}}W_{w_{(3)}}I_{w_{(2)}}I_{w_{(3)}}f^{(s)}=\omega(-1)W_{w_{(1)}}W_{w_{(3)}}I_{w_{(2)}}I_{w_{(3)}}I_{w_3}^*f^{(s)}.$$
If follows that
\begin{align}\label{E:key identity}
\overline{G}_{I_{w_3}^*f,h_2}^{(s,u_2)}=\omega(-1)\gamma (2s-2u_2-1,\omega_1\nu_2^2,\psi)\gamma (-2s-2u_2+1,\omega_1^{-1}\mu_2^{-2},\psi)^{-1}\overline{G}_{f,h_2}^{(s,u_2)}.
\end{align}
\\Note that for $x \in F$, we have
\begin{align*}
&\bp 0&-1&0&0 \\ 0&-1&-1&0 \\ -1&0&0&-1 \\ 1&0&0&0\ep \iota_2 (a(\xi)w_1a(\xi)^{-1}u(x)) \\
&= m \left (\bp 1&0 \\ 1&1 \ep ,1\right )u\left ( \bp 0&0 \\ 0 &2\xi^{-2}x \ep \right)m \left (\bp 2\xi^2 &0 \\ 0&-2^{-1}\xi^{-2} \ep ,-1\right )\\
&\times \bp  0&0&1&0\\0&1&0&0\\-1&0&0&0\\0&0&0&1\ep u\left( \bp -2x &0\\0&0 \ep \right )m\left (\bp 1&0 \\ 0&-1 \ep ,-1\right ).
\end{align*}
Therefore, for $g \in \GL_2^{\circ}(F)$
\begin{align}\label{E:identity for sections 2}
\begin{split}
\overline{\mathscr{F}}_{f,h_2}^{(s,u_2)}(g) &= |2|_F^{-1}|2\xi^2|_F^{2s-2u_2-1}\nu_2(-1)\omega\mu_2^{-1}\nu_2(2\xi^2)\overline{G}_{f,h_2}^{(s,u_2)}\left( m\left( \bp 1 & 0 \\ 0 & -1 \ep ,-1\right )\iota_2(g) \right),\\
\overline{\mathscr{F}}_{I_{w_3}^*f,h_2}^{(s,u_2)}(g) &= |2|_F^{-1}|2\xi^2|_F^{2(1-s)-2u_2-1}\omega_1\mu_2(-1)\omega^{-1}\mu_2^{-1}\nu_2(2\xi^2)\overline{G}_{I_{w_3}^*f,h_2}^{(s,u_2)}\left( m\left( \bp 1 & 0 \\ 0 & -1 \ep ,-1\right )\iota_2(g) \right).
\end{split}
\end{align}
By (\ref{E:key identity}) and (\ref{E:identity for sections 2}), we have
\begin{align}\label{E:key identity 2}
\overline{\mathscr{F}}_{I_{w_3}^*f,h_2}^{(s,u_2)} = \omega(4\xi^4)^{-1}|4\xi^4|_F^{-2s+1}\gamma (2s-2u_2-1,\omega_1\nu_2^2,\psi)\gamma (-2s-2u_2+1,\omega_1^{-1}\mu_2^{-2},\psi)^{-1}\overline{\mathscr{F}}_{f,h_2}^{(s,u_2)}.
\end{align}
By the properties of $\overline{G}_{f,h_2}^{(s,u_2)}$ and $\overline{G}_{I_{w_3}^*f,h_2}^{(s,u_2)}$, and (\ref{E:identity for sections 2}), for each $\epsilon>0$ there exist $\varphi \in S(F)$ depending on $\epsilon$, $s$ and $u_2$ such that
\begin{align*}
\begin{split}
\overline{\mathscr{F}}_{f,h_2}^{(s,u_2)}(a(\nu)g)&\ll_{s,u_2}|\nu|_F^{u_2+3/2-|{\rm Re}(s)-1/2|-\epsilon}\varphi(\nu),\\
\overline{\mathscr{F}}_{I_{w_3}^*f,h_2}^{(s,u_2)}(a(\nu)g)&\ll_{s,u_2}|\nu|_F^{u_2+3/2-|{\rm Re}(s)-1/2|-\epsilon}\varphi(\nu)
\end{split}
\end{align*}
for $g \in \Omega$ and $\nu \in F^{\times}$. Therefore, the integrals
\begin{align*}
\int_{F^{\times}U^{}(E) \backslash \GL_2^{\circ}(F)}\overline{{\mathscr F}}_{f,h_2}^{(s,u_2)}(g)W_1^{(u_1)}(g)dg
,\quad
\int_{F^{\times}U^{}(E) \backslash \GL_2^{\circ}(F)}\overline{{\mathscr F}}_{I_{w_3}^*f,h_2}^{(s,u_2)}(g)W_1^{(u_1)}(g)dg
\end{align*}
are both absolutely converge for $(s,u_1,u_2)\in \mathscr{D}$. It follows that (\ref{E:first step}) holds for $(s,u_1,u_2)$ in the nonempty open set $D \cap D^{\vee} \cap \mathscr{D}$. By \cite[Appendix, Theorem]{F93} and Theorem \ref{T:main theorem}-(1), $$L_{\rm RS}(s+u_2,{\rm As}\,\pi_1^{(u_1)}\otimes \mu_1)L_{\rm RS}(s-u_2,{\rm As}\,\pi_1^{(u_1)}\otimes \nu_2)$$ and $$L_{\rm RS}(1-s-u_2,{\rm As}\,\pi_1^{(u_1)}\otimes \omega_1^{-1}\mu_2^{-1})L_{\rm RS}(1-s+u_2,{\rm As}\,\pi_1^{(u_1)}\otimes \omega_1^{-1}\nu_2^{-1})$$ are non-zero and has no poles in $D$ and $D^{\vee}$, respectively. Therefore, by the uniqueness of the analytic continuation, (\ref{E:functional equation for twisted Asai}), (\ref{E:first step}), and (\ref{E:key identity 2}), we conclude that both
\begin{align*}
\frac{\Psi(I_{w_3}^*f^{(s)},W_1^{(u_1)},W_2^{(u_2)})}{L_{\rm RS}(1-s-u_2,{\rm As}\,\pi_1^{(u_1)}\otimes \omega_1^{-1}\mu_2^{-1})L_{\rm RS}(1-s+u_2,{\rm As}\,\pi_1^{(u_1)}\otimes \omega_1^{-1}\nu_2^{-1})}
\end{align*}
and 
\begin{align*}
\frac{\Psi(f^{(s)},W_1^{(u_1)},W_2^{(u_2)})}{L_{\rm RS}(s+u_2,{\rm As}\,\pi_1^{(u_1)}\otimes \mu_2)L_{\rm RS}(s-u_2,{\rm As}\,\pi_1^{(u_1)}\otimes \nu_2)}
\end{align*}
can be analytically continued to $D \cup D^{\vee}$, and for $(s,u_1,u_2) \in D  \cup D^{\vee}$
\begin{align}\label{E:step two}
\begin{split}
&\frac{\Psi(I_{w_3}^*f^{(s)},W_1^{(u_1)},W_2^{(u_2)})}{L_{\rm RS}(1-s-u_2,{\rm As}\,\pi_1^{(u_1)}\otimes \omega_1^{-1}\mu_2^{-1})L_{\rm RS}(1-s+u_2,{\rm As}\,\pi_1^{(u_1)}\otimes \omega_1^{-1}\nu_2^{-1})} \\&= \omega(4\xi^4)^{-1}|4\xi^4|_F^{-2s+1}\varepsilon(s+u_2,{\rm As}\,\pi_1^{(u_1)}\otimes \mu_2 ,\psi,\xi ) \varepsilon(s-u_2,{\rm As}\,\pi_1^{(u_1)}\otimes \nu_2 ,\psi,\xi )\\
&\times \frac{\Psi(f^{(s)},W_1^{(u_1)},W_2^{(u_2)})}{L_{\rm RS}(s+u_2,{\rm As}\,\pi_1^{(u_1)}\otimes \mu_2)L_{\rm RS}(s-u_2,{\rm As}\,\pi_1^{(u_1)}\otimes \nu_2)}.
\end{split}
\end{align}
Note that $\varepsilon_{\rm RS}(s+u_2,{\rm As}\,\pi_1^{(u_1)}\otimes \mu_2 ,\psi,\xi ) \varepsilon_{\rm RS}(s-u_2,{\rm As}\,\pi_1^{(u_1)}\otimes \nu_2 ,\psi,\xi )$ is an entire function without zeros. Since $D \cup D^{\vee}$ is a connected tube domain with convex hull equal to $\C^3$, both sides of (\ref{E:step two}) can be analytically continued to $\C^3$ and satisfying the functional equation (\ref{E:step two}). Substituting $u_1=v_1$ and $u_2=v_2$, we have
\begin{align*}
\gamma_{\rm PSR}(s,{\rm As}\,\pi_1 \otimes \pi_2,\psi,\xi) &= \omega(4\xi^4)^{-1}|4\xi^4|_F^{-2s+1}\gamma_{\rm RS}(s+v_2,{\rm As}\,\pi_1\otimes \mu_2 ,\psi,\xi ) \gamma_{\rm RS}(s-v_2,{\rm As}\,\pi_1\otimes \nu_2 ,\psi,\xi ).
\end{align*}
The second assertion follows from Corollary \ref{equality of epsilon factors} and Theorem \ref{T:main theorem}. This completes the proof.
\end{proof}

\appendix
\section*{Appendix}
Let $F$ be a local field of characteristic zero and $E/F$ be a quadratic field extension. Fix an element $\xi \in E^{\times}$ such that ${\rm tr}_{E/F}(\xi)=0$. Let $\psi$ be a non-trivial additive character of $F$.

Let $\mu$, $\nu$ be two characters of $E^{\x}$.
Let $\pi={\rm Ind}_{B(E)}^{{\rm GL}_2(E)}(\mu,\nu)$ be an infinite-dimensional irreducible Harish-Chandra
representation of ${\rm GL}_2(\C)$ if $F=\R$, and a smooth admissible representation of $\GL_2(E)$ if $F$ is non-archimedean. Let $W\in\cW(\pi,\psi_\xi)$ and $\Phi\in\cS(F^2,\psi)$. Here $\cS(F^2,\psi) = \cS(F^2)$ if $F$ is non-archimedean. It is easy to verify that
\[
Z(s,W,\Phi)
=
\int_{\R^{\x}\U(\R)\backslash\GL_2(\R)}
W(g)f^{(s)}_\Phi(g)dg,
\]
where $f^{(s)}_\Phi$ is the Godement section associated to $\Phi\in\cS(F^2,\psi)$ 
defined by the integral 
$$f^{(s)}_\Phi(g)
=
|{\rm det}(g)|^s\int_{F^{\x}}\Phi((0,t)g)\omega_0(t)|t|_F^{2s}d^{\x}t.$$
 Recall that 
\[
f^{(s)}_\Phi
\in
I(s,\omega_0)_0
\quad\text{with}\quad
I(s,\omega)
=
\cB\left(|\mbox{ }|^{s-1/2},\omega_0^{-1}|\mbox{ }|^{1/2-s}\right).
\]
Here $\omega_0=\mu\nu|_{\R^{\x}}$. 

In this appendix, we show that the zeta integrals have meromorphic continuations to the whole complex
plane, and satisfy the functional equation even when we replace $f^{(s)}_\Phi$ by a holomorphic section 
$f^{(s)}$ of $I(s,\omega_0)$.   By a holomorphic section of $I(s,\omega_0)$ we mean a function 
$f^{(s)}(g):\GL_2(F)\x\C\to\C$ which satisfies following two conditions:
\begin{itemize}
\item
for each $s\in\C$, $f^{(s)}(g)\in I(s,\omega_0)$.
\item
for each $g\in\GL_2(F)$, $f^{(s)}(g)$ is a holomorphic function in $s$.
\end{itemize}
Notice that we do not require $f^{(s)}$ to be right ${\rm O}_2(\R)$-finite when $F=\R$. 
Let $W\in\cW(\pi,\psi_\xi)$ and $f^{(s)}$ be a holomorphic section of $I(s,\omega_0)$.
Define
\[
Z(f^{(s)}, W)
=
\int_{F^{\x}\U(F)\backslash\GL_2(F)}
W(g)f^{(s)}(g)dg.
\]
Analogous to \subsecref{SS:the functional equation}, we define the intertwining operator 
\[
M^*:\cB\left(|\mbox{ }|^{s-1/2},\omega_0^{-1}|\mbox{ }|^{1/2-s}\right)\longto
\cB\left(\omega_0^{-1}|\mbox{ }|^{1/2-s},|\mbox{ }|^{s-1/2}\right)
\] 
by
\[
Mf^{(s)}(g)
:=
\int_F
f^{(s)}\left(
\begin{pmatrix}
0&1\\-1&0
\end{pmatrix}
\pMX 1x01 g
\right)dx
\quad\text{and}\quad
M^*:=
\omega_0(-1)\gamma\left(2s-1,\omega_0,\psi\right)M.
\]
The integral converges absolutely for ${\rm Re}(s)\gg 0$ and admits a 
meromorphic continuation to whole complex plane. Moreover, $M^*f^{(s)}$ is a holomorphic section.

The purpose of this appendix is to prove following results, which are used in \secref{S:4}. 

\begin{propa}\label{P:A}
Write
$$\mu = \chi_1 |\mbox{ }|_E^{\lambda_1},\quad \nu = \chi_2 |\mbox{ }|_E^{\lambda_2}$$
for some unitary characters $\chi_1$ and $\chi_2$ of $E^{\times}$ and some $\lambda_1, \lambda_2 \in \C$. 

Let $W \in \cW(\pi,\psi_\xi)$ and $f^{(s)}$ be a holomorphic section of $I(s,\omega_0)$. The integral $Z(f^{(s)},W)$ converges absolutely when 
${\rm Re}(s)>2\,{\rm max}\stt{-{\rm Re}(\lambda_1),-{\rm Re}(\lambda_2)}$ and has a 
meromorphic continuation to the whole complex plane. Moreover, it satisfies the functional equation
\[
Z(M^*f^{(s)},W)
=
\gamma_{{\rm RS}}\left(s,{\rm As}\,\pi,\psi,\xi\right)
Z(f^{(s)},W).
\]
\end{propa}

\begin{proof}
When $F$ is non-archimedean, the assertions are easy to prove. Since any holomorphic section of $I(s,\omega_0)$ is right $\GL_2(\mathcal{O}_F)$-finite.

Assume $F=\R$. Let $\chi$ be a character of $\R^{\x}$. For $W\in\cW(\pi,\psi_\xi)$, let $\zeta(s,W,\chi)$ be the integral
defined by the equation \eqref{E:Tate integral for C x R}. We first show that the second assertion of
\lmref{L:integration formula for Tate C x R} holds for $W\in\cW(\pi,\psi_\xi)$.
This can be argued as 
follows: Certainly we can assume $W=W_\Psi$ for some $\Psi\in\cS(\C^2)$. Then by equation 
\eqref{E:integral representation of Whittaker function} and by changing the variables as in the proof 
of \lmref{L:integration formula for Tate C x R}, we arrive
\[
\zeta(s,W_\Psi,\chi)
=
2\int_{\R^{\x}}\int_{\R^{\x}_+}
\Psi'(y,t)\mu\chi(y)|y|^s\nu\chi(t)t^sd^{\x}yd^{\x}t,
\]
where we define, for $\Psi\in\cS(\C^2)$, $y,t\in\R$, an element $\Psi'\in\cS(\R^2)$ by
\[
\Psi'(y,t)
=
\int_0^{2\pi}\Psi(ye^{i\theta},te^{-i\theta})\mu\nu^{-1}(e^{i\theta})d\theta.
\]
One check easily that
\[
\Psi'(y,-t)
=
\mu\nu(-1)\Psi'(-y,t).
\]
It follows that 
\[
\zeta(s,W_\Psi,\chi)
=
\int_{\R^{\x}}\int_{\R^{\x}}
\Psi'(y,t)\mu\chi(y)|y|^s\nu\chi(t)|t|^sd^{\x}yd^{\x}t.
\]
Our assertion now follows immediately from \cite[Lemma 5.15.1]{JL70}. 
This proves the claim.

We show the integral $Z(f^{(s)},W)$ converges absolutely when 
\[
{\rm Re}(s)>2\,{\rm max}\stt{-{\rm Re}(\lambda_1),-{\rm Re}(\lambda_2)}
\]
and has a meromorphic continuation to the whole complex plane. Indeed, by the Iwasawa decomposition,
we have
\[
Z(f^{(s)},W)
=
\int_{{\rm SO}_2(\R)}
f^{(s)}(k)\zeta(s,\rho(k)W,1)dk,
\]
where $1$ is the trivial character of $\R^{\x}$. The first assertion of Proposition A follows immediately
from the claim and the compactness of ${\rm SO}_2(\R)$.

Next we prove the functional equation. Its known that 
$M^*f^{(s)}(g)$ converges absolutely for ${\rm Re}(s)$ is large and extends to a meromorphic 
function on $\C$. Moreover, except countable many $s$, $M^*f^{(s)}$ defines an element in 
$\cB\left(\omega_0^{-1}|\mbox{ }|^{1/2-s},|\mbox{ }|^{s-1/2}\right)$. 
Same argument as above shows the integral $Z(M^*f^{(s)},W)$ has a 
meromorphic continuation to whole complex plane.

On the other hand, except countable many $s$, both integrals $Z(f^{(s)},W)$ and $Z(M^*f^{(s)},W)$
define continuous $\GL_2(\R)$-invariant functional on 
\[
\cB(\mu,\nu)\ot I(s,\omega_0).
\]
Here we index $W=W_f$ by an element $f\in\cB(\mu,\nu)$. Since $\cB(\mu,\nu)_0$
(resp. $I(s,\omega_0)_0$)
 is dense in $\cB(\mu,\nu)$
(resp. $I(s,\omega_0)$),
the uniqueness result of
\cite[Theorem 1.3]{Lok01} implies these two functional are proportion. Therefore there is a 
meromorphic function $\gamma(s)$ which is independent of $f^{(s)}$ and $W$ such that
\[
Z(M^*f^{(s)},W)=\gamma(s)Z(f^{(s)},W),
\]
for all such $f^{(s)}$ and $W\in\cW(\pi,\psi_\xi)$. It remains to show that 
$\gamma(s)=\gamma_{{\rm RS}}(s,{\rm As}\,\pi,\psi,\xi)$. But this follows immediately if we let 
$f^{(s)}=L(2s,\omega_0)^{-1}f^{(s)}_\Phi$ and choice any $W\in\cW(\pi,\psi_\xi)$ so that 
$Z(s,W,\Phi)$ is non-vanishing. The section assertion of Proposition A then follows from equation
\eqref{E:relating zeta integrals and invariant forms}.
\end{proof}

\end{document}